\documentclass[ 11pt]{article}

\usepackage[utf8]{inputenc}  
\usepackage[T1]{fontenc}         
\usepackage{geometry}         
\usepackage[english]{babel}  
\usepackage{cite}
\usepackage{amssymb}   
\usepackage{amsmath} 
\usepackage{amscd}
\usepackage{graphicx, color}    
\usepackage{amsthm}                 
\usepackage{latexsym}     
\usepackage[all]{xy}

\newtheorem{thm}{Theorem}[section]
\newtheorem{cor}[thm]{Corollary}
\newtheorem{prop}[thm]{Proposition}
\newtheorem{lem}[thm]{Lemma}

\theoremstyle{definition}
\newtheorem{definition}[thm]{Definition}

\newcommand{\bbN}{{\mathbb N}}

\newcommand{\bbQ}{{\mathbb Q}}

\newcommand{\cF}{{\mathcal F}}

\newcommand{\cO}{{\mathcal O}}

\newcommand{\cS}{{\mathcal S}}

\newcommand{\cU}{{\mathcal U}}
\newcommand{\cV}{{\mathcal V}}

\newcommand{\cX}{{\mathcal X}}
\newcommand{\cY}{{\mathcal Y}}
\newcommand{\cZ}{{\mathcal Z}}

\newcommand{\ra}{\rightarrow}

\newcommand{\hra}{\hookrightarrow}

\title{\textbf{Complexes of Groups and Boundaries}}           
\author{Alexandre Martin}
\date{}

\begin{document}
\maketitle

\begin{abstract} \noindent Given a complex of groups over a finite simplicial complex in the sense of Haefliger \cite{HaefligerOrbihedra}, we give conditions under which it is possible to build an $E\cZ$-structure in the sense of Farrell-Lafont \cite{FarrellLafontEZStructures} for its fundamental group out of such structures for its local groups. As an application, we prove a combination theorem that yields a procedure for getting hyperbolic groups as fundamental groups of simple complexes of hyperbolic groups. The construction provides a description of the Gromov boundary of such groups.\end{abstract}

In \cite{BestvinaZBoundaries}, Bestvina defined a fundamental notion of boundary that is relevant to geometric group theory. He showed how the topology of the boundary $\partial G$ of a group $G$ encodes the cohomology with group ring coefficients of $G$. This notion of boundary was further generalised to the notion of an equivariant compactification by the work of Farrell and Lafont \cite{FarrellLafontEZStructures}, who proved the Novikov conjecture for groups admitting what they call an $E\cZ$-structure, that is to say a classifying space for proper actions together with an equivariant $\cZ$-compactification.

The existence of $E\cZ$-structures, and their generalisation for groups with torsion, is known for groups that admit a classifying space for proper actions with a sufficiently nice geometry. For a group $G$ acting properly and cocompactly on a CAT(0) space $X$, the compactification of $X$ obtained by adding the visual boundary $\partial X$ yields an $E\cZ$-structure for $G$. In the case of a torsionfree hyperbolic group $G$ (\cite{GromovHyperbolicGroups} ), a classifying space is given by an appropriate Rips complex (see \cite{CoornaertDelzantPapadopoulos}). Bestvina and Mess \cite{BestvinaMessBoundaryHyperbolic} proved that such a space can be compactified by adding the Gromov boundary of $G$ to get an $E\cZ$-structure for $G$. This result was further generalised in \cite{MeintrupSchickEGhyperbolic} to the case of hyperbolic groups with torsion, where they show that such a compactification yields an $E\cZ$-structure in the sense of Carlsson-Pedersen \cite{CarlssonPedersenEZBoundaries}. The existence of such an $E\cZ$-structure is also know for systolic groups by work of Osajda and Przytycki \cite{OsajdaPrzytyckiBoundarySystolic}. \\

In this article, we adress the following combination problem: Given a group $G$ acting cocompactly by simplicial isometries on a simplicial complex $X$, are there natural conditions under which it is possible to build an $E\cZ$-structure for $G$, assuming that the stabilisers of simplices all admit such a structure?

There are already some special cases for which such a combination theorem is know to hold. For instance, Tirel \cite{TirelProducts} explained how to build a $\cZ$-boundary for free and direct products of groups admitting $\cZ$-boundaries. Furthermore, Dahmani \cite{DahmaniBoundaryRelativelyHyperbolic} built an $E\cZ$-structure for a torsionfree group that is hyperbolic relative to a group admitting an $E\cZ$-structure.

The main result of this article deals with acylindrical actions on CAT(0) spaces; recall that an action is called \textit{acylindrical} if the diameter of sets with infinite pointwise stabiliser is uniformly bounded above\footnote{The original definition of acylindricity by Sela \cite{SelaAccessibility} considers nontrivial stabilisers instead of infinite ones. Here we use a more general notion of acylindricity introduced by Delzant \cite{DelzantAcylindrique} that is more suitable for proper actions.}. More precisely, we consider a non-positively curved, hence developable (see \cite{BridsonHaefliger}), complex of groups $G(\cY)= (G_\sigma, \psi_a, g_{a,b})$ over a finite simplicial complex $Y$ endowed with a $M_\kappa$-structure, $\kappa \leq 0$, in the sense of Bridson \cite{BridsonPhD}, such that the stabiliser of every simplex $\sigma$ of $Y$ admits an $E\cZ$-structure $(\overline{EG_\sigma}, \partial G_\sigma)$. We further assume that these structures define an $E\cZ$-\textit{complex of space compatible with} $G(\cY)$ (see \ref{EZcomplexofspaces} for a precise definition), that is, there are embeddings $\phi_{\sigma, \sigma'}: \overline{EG_{\sigma'}} \hra \overline{EG_\sigma}$, for all $\sigma \subset \sigma'$, that are equivariant with respect to the local maps of $G(\cY)$, and such that the induced diagram of embeddings is commutative up to multiplication by twisting elements of $G(\cY)$.

\begin{thm}[Combination Theorem for Boundaries of Groups]

Let $G(\cY)$ be a non-positively curved complex of groups over a finite simplicial complex $Y$ endowed with a $M_\kappa$-structure, $\kappa \leq 0$. Let $G$ be the fundamental group of $G(\cY)$ and $X$ be a universal covering \footnote{The simplicial complex $X$ naturally inheritates a $M_\kappa$-structure from that of $Y$, which makes it a complete geodesic metric space by work of Bridson \cite{BridsonPhD}; the CAT(0) property follows from the Cartan-Hadamard theorem.} of $G(\cY)$. Suppose that the following global condition holds:
\begin{itemize}
      \item[$(i)$] The action of $G$ on $X$ is acylindrical.
\end{itemize}			
Further assume that there is an $E\cZ$-complex of spaces compatible with $G(\cY)$ that satisfies each of the following local conditions: 
\begin{itemize}  
      \item[$(ii)$]  the \textit{limit set property}: For every pair of simplices $\sigma \subset \sigma'$ of $Y$, the embedding $\overline{EG_{\sigma'}} \hra \overline{EG_\sigma}$ realises an equivariant homeomorphism from $\partial G_{\sigma'}$ to the limit set $\Lambda G_{\sigma'} \subset \partial G_\sigma$. Furthermore, for every simplex $\sigma$ of $Y$, and every pair of subgroups $H_1, H_2$ in the family $\cF_\sigma = \left\lbrace  \bigcap_{i=1}^{n} g_i G_{\sigma_i} g_i^{-1}  | ~ g_1, \ldots, g_n \in G_\sigma, \sigma_1, \ldots,  \sigma_n \subset \mbox{st}(\sigma), n \in \bbN \right\rbrace,$ we have $ \Lambda H_1 \cap \Lambda H_2 = \Lambda(H_1 \cap H_2)  \subset \partial G_\sigma.$ 
      \item[$(iii)$]  the \textit{convergence property}: for every pair of simplices $\sigma \subset \sigma'$ in $Y$ and every sequence $(g_n)$ of $G_\sigma$ whose projection is injective in $G_\sigma / G_{\sigma'}$, there exists a subsequence such that $(g_{\varphi(n)}\overline{EG_{\sigma'}})$ uniformly converges to a point in $\overline{EG_\sigma}$.
       \item[$(iv)$] the \textit{finite height property}: for every pair of simplices $\sigma \subset \sigma'$ of $Y$, $G_{\sigma'}$ has finite height in $G_\sigma$ (see \cite{SageevWidth}), that is, there exist an upper bound on the number of distinct cosets $\gamma_1 G_{\sigma'} ,\ldots,\gamma_n G_{\sigma'} \in G_\sigma/G_{\sigma'}$ such that the intersection $\gamma_1 G_{\sigma'} \gamma_1^{-1} \cap \ldots \cap \gamma_n G_{\sigma'} \gamma_n^{-1}$ is infinite.
\end{itemize}
Then $G$ admits an $E\cZ$-structure $(\overline{EG}, \partial G)$ in the sense of Farrell-Lafont. 

Furthermore, the following properties hold:
\begin{itemize}
 \item[$(ii')$]  For every simplex $\sigma$ of $Y$, the map $\overline{EG_\sigma} \ra \overline{EG}$ realises an equivariant embedding from $\partial G_\sigma$ to $\Lambda G_\sigma \subset \Lambda G$. Moreover, for every pair $H_1, H_2$ of subgroups in the family $\cF = \left\lbrace \bigcap_{i=1}^{n} g_i G_{\sigma_i} g_i^{-1} | ~ g_1, \ldots, g_n \in G, ~ \sigma_1, \ldots, \sigma_n \in \mbox{S}(Y), ~ n\in \bbN \right\rbrace,$ we have $\Lambda H_1 \cap \Lambda H_2 = \Lambda(H_1 \cap H_2) \subset  \partial G.$
      \item[$(iii')$] For every simplex $\sigma$ of $Y$, the embedding $\overline{EG_\sigma} \hra \overline{EG}$ satisfies the convergence property.
       \item[$(iv')$] For every simplex $\sigma$ of $Y$, the local group $G_\sigma$ has finite height in $G$.
\end{itemize}
\label{maintheorem1}
\end{thm}

As an application of the previous construction, we prove a higher dimensional combination theorem for hyperbolic groups, in the case of acylindrical complexes of groups of arbitrary dimension.\\

\begin{thm}[Combination Theorem for Hyperbolic Groups] Let $G(\cY)$ be a strictly developable non-positively curved simple complex of groups over a finite simplicial complex $Y$ endowed with a $M_\kappa$-structure, $\kappa \leq 0$. Let $G$ be the fundamental group of $G(\cY)$ and $X$ be a universal covering of $G(\cY)$. Assume that:
\begin{itemize}
\item The universal covering $X$ is hyperbolic\footnote{For instance, when $\kappa <0$.}, when endowed with the simplicial metric coming from the induced $M_\kappa$-complex structure,
\item The local groups are hyperbolic and all the local maps are quasiconvex embeddings,
\item The action of $G$ on $X$ is acylindrical. 
\end{itemize}
Then $G$ is hyperbolic. Furthermore, the local groups embed in $G$ as quasiconvex subgroups.
\label{maintheorem2}
\end{thm}

\noindent Note that a complex of groups over a simply connected simplicial complex is developable if and only if it is strictly developable (see \cite{BridsonHaefliger}). Hence one might try to create new hyperbolic groups as fundamental groups of non-positively curved complexes of hyperbolic groups over a simply-connected finite complex (see Theorem II.12.28 of \cite{BridsonHaefliger}). 

Such a result is already known for acylindrical graphs of groups: the hyperbolicity is a direct consequence of the much more general combination theorem of Bestvina and Feighn \cite{BestvinaFeighnCombinationHyperbolic}, while the quasiconvexity of vertex stabilisers follows from a result of Kapovich \cite{KapovichCombinationQuasiconvexity}. To the best of our knowledge, Theorem 0.2 is the first combination theorem for higher dimensional complexes of groups.

Our construction follows the strategy of Dahmani \cite{DahmaniCombination}, who applied this idea to amalgamate Bowditch boundaries of relatively hyperbolic groups in the case of acylindrical graphs of groups. The proofs in our case are significantly more involved as the topology of $X$ can be much more complicated than that of a tree. Generalizing an argument of Dahmani, we prove that $G$ is a uniform convergence group on $\partial G$ (see Section 6 for definitions), which implies the hyperbolicity of $G$ by a celebrated result of Bowditch \cite{BowditchTopologicalCharacterization} and Tukia \cite{TukiaUniformConvergenceGroups}.\\

The article is organised as follows. In Section 1, we study complexes of spaces over a simplicial complex. These spaces are direct generalisations of graphs of spaces studied by Scott and Wall in the context of Bass-Serre theory \cite{ScottWall}. In Section 2, we give conditions under which it is possible to build a classifying space for proper actions of the fundamental group of a complex of groups as a complex of spaces over its universal covering. We also define the boundary $\partial G$ of $G$ and the compactification $\overline{EG}$ of $EG$ as sets. In Section 3, we investigate geometric properties of CAT(0) spaces and we define a topology on $\overline{EG}$. We prove that this topology makes $\overline{EG}$ a compact metrisable space in Section 4. The proof of Theorem \ref{maintheorem1} is completed in Section 5, where the properties of $\partial G$ are investigated. Finally, Section 6 is devoted to the dynamics of $G$ on its boundary and to the proof of the combination theorem \ref{maintheorem2} for hyperbolic groups. 
\\

\noindent \textbf{Notations.} Throughout this paper, $X$ is a simplicial complex. For an element $x$ of $X$, we denote by $\sigma_x$ the unique simplex containing $x$ in its interior. For a simplex $\sigma$ of $X$, we denote by $\mbox{st}(\sigma)$ the open star of $\sigma$ and by $\overline{\mbox{st}}(\sigma)$ its closed star. We denote by $S(X)$ the set of simplices of $X$, and by $V(X)$ the set of its vertices. For an element $x$ of $X$ and a constant $r>0$, we denote by $B(x,r)$ (resp. $\overline{B}(x,r)$) the open (resp. closed) $r$-ball centered at $x$.

All the types of classifying spaces we will consider in this paper are classifying spaces for proper actions (see Section 2 for definitions). Consequently, we will simply speak of \textit{classifying spaces} rather than \textit{classifying spaces for proper actions}. Moreover, although the notation $\underbar{E}G$ is well spread in the litterature to mean a classifying space for proper actions of a discrete group $G$, we will simply use the notation $EG$ so as to avoid the somehow unaesthetic notation $\overline{\underbar{E}G}$ when speaking of an $E\cZ$-compactification of a classifying space for proper actions of $G$.\\

\noindent \textbf{Acknowledgements.} I am grateful to Thomas Delzant for proposing this problem to me, as well as for his help and advice during this work.

\section{Complexes of spaces and their topology.}
In this section, we study a class of spaces with a projection to a given simplicial complex $X$, called \textit{complexes of spaces over $X$}, which are high-dimensional analogues of graphs of spaces studied in the context of actions on trees (see \cite{ScottWall}). This notion of complexes of spaces is close but more precise than the one studied by various authours (see \cite{CorsonComplexesofGroups}, \cite{HaefligerExtension}).

\subsection{A few geometric facts about $M_\kappa$-simplical complexes.}
Since the present article deals with nonproper actions of a group, the simplicial complex on which it acts is generally non locally finite. In \cite{BridsonPhD}, Bridson defined a class of spaces that is suitable for a geometric approach.

\begin{definition}[$M_\kappa$-simplicial complexes with $\kappa \leq 0$, \cite{BridsonPhD}]
 Let $\kappa \leq 0$. A simplicial complex $X$ is called a $M_\kappa$-simplicial if it satisfies the following two conditions:
\begin{itemize}
 \item Each simplex of $X$ is modeled after a geodesic simplex in some $M_\kappa^n$, where $M_\kappa^n$ is the simply-connected $n$-Riemannian manifold of constant curvature $\kappa$.
 \item If $\sigma$ and $\sigma'$ are two simplices of $X$ sharing a common face $\tau$, the identity map from $\tau \subset \sigma$ to $\tau \subset \sigma'$ is an isometry.
\end{itemize}
To such a $M_\kappa$-complex $X$ is associated a canonical simplicial metric.
\end{definition}

\begin{thm}[Bridson \cite{BridsonPhD}]
 If $X$ is a $M_\kappa$-simplicial complex, $\kappa \leq 0$, with finitely many isometry types of simplices, the simplicial metric is complete and geodesic.
\label{Mkappacomplexes}
\end{thm}

\begin{center}
\textit{From now on, every simplicial complex will implicitely be given the structure of a $M_\kappa$-complex, $\kappa \leq 0$.}
\end {center}

\subsection{Complexes of spaces.}
\begin{definition} A \textit{complex of spaces $C(\cX)$ over} $X$ consists of the following data:
\begin{itemize}
 \item  for every simplex $\sigma$ of $X$, a pointed CW-complex $ C_\sigma$, called a \textit{fibre}, 
 \item for every pair of simplices $\sigma \subset \sigma'$, a pointed embedding $\phi_{\sigma', \sigma}: \displaystyle{ C_{\sigma'} \hra C_\sigma}$, called a \textit{gluing map}, such that for every $\sigma \subset \sigma' \subset \sigma''$, we have $\phi_{\sigma, \sigma''} = \phi_{\sigma, \sigma'} \circ \phi_{\sigma', \sigma''}$.
\end{itemize}
\end{definition}

\begin{definition}[realisation of a complex of spaces] Let $C(\cX)$ be a complex of spaces over $X$. The \textit{realisation} of $C(\cX)$  is the quotient space
$$ |C(\cX)|= \big( \coprod_{\sigma  \in S(X)} \sigma \times C_\sigma~ \big) / \simeq  $$
where 
$$ (i_{\sigma,\sigma'}(x), s) \simeq(x , \phi_{\sigma, \sigma'}(s)) \mbox{ for } x \in \sigma \subset \sigma'\mbox{ and } s \in C_{\sigma'}, $$
where $i_{\sigma,\sigma'}: \sigma \hookrightarrow \sigma' $ is the natural inclusion. The class in $|C(\cX)|$ of a point $(x,s)$ will be denoted $[x, s]$.
\end{definition}

\begin{definition} A complex of spaces $C(\cX)$ will be called \textit{locally finite} if for every simplex $\sigma$ of $X$ and every element $x \in C_\sigma$, there exists an open set $U$ of $C_\sigma$ containing $x$ and such that there are only finitely many simplices $ \sigma'$ in the open star of $\sigma$ satisfying $ U \cap \mbox{Im}( \phi_{\sigma, \sigma'} )\neq \varnothing $.
\end{definition}
\begin{prop} Let $C(\cX)$ be a locally finite complex of spaces. Then $| C(\cX) |$ admits a natural locally finite CW-complex structure, for which the $\sigma \times C_\sigma$ embed as subcomplexes. \qed
\end{prop}

\subsection{Topology of complexes of spaces with contractible fibres.}
\begin{definition}[quotient complex of spaces]
Let $C(\cX)$ be a complex of spaces over $X$ and $Y\subset X$ a subcomplex. We denote $C_Y(\cX)$ the complex of spaces over $X$ defined as follows:
\begin{itemize}
\item $(C_Y)_\sigma= C_\sigma$ if $\sigma \nsubseteq Y$, $(C_Y)_\sigma$ is the basepoint of $C_\sigma$ otherwise,
\item $\phi_{\sigma,\sigma'}^Y$ is the composition $(C_Y)_{\sigma'} \hra C_{\sigma'} \xrightarrow{\phi_{\sigma, \sigma'}} C_\sigma  \twoheadrightarrow (C_Y)_\sigma$.
\end{itemize}
\end{definition}

We denote by $p_Y: |C(\cX)| \rightarrow |C_Y(\cX)|$ the canonical projection, and simply $p$ for $p_X: |C(\cX)| \ra X$. In the same way, if $Y \subset Y'$ are subcomplexes of $X$, we denote by $p_{Y,Y'}:|C_Y(\cX)| \rightarrow |C_{Y'}(\cX)|$ the canonical projection.

\begin{lem} Let $C(\cX)$ be a locally finite complex of spaces over $X$  with contractible fibres, and let $Y$ be a finite subcomplex of $X$. Then $p_Y: |C(\cX)| \rightarrow |C_Y(\cX)|$ is a homotopy equivalence.
\label{quotientequivalence}
\end{lem}

\begin{proof} It amounts to proving the result for $Y$ consisting of a single closed simplex $\sigma$. We have the following commutative diagram:
$$ \xymatrix{
     |C(\cX)|  \ar[d]_{\simeq} \ar[r]_{}^{p_Y} & |C_Y(\cX)| \ar[d]^{\simeq} \\
    |C(\cX)|/\left(\sigma \times C_\sigma\right) \ar[r]_{=} & |C_Y(\cX)|/\left(\sigma \times \star\right). \\
  }$$\\
The vertical arrows are homotopy equivalences, since we are quotienting by contractible subcomplexes, hence the result.  
\end{proof}

\begin{thm}[Dowker \cite{Dowker52}]
 The (continuous) identity map $X \ra X$ from $X$ with its CW topology to $X$ with its simplicial metric is a homotopy equivalence. \qed
\end{thm}

\begin{prop} Let $C(\cX)$ be a locally finite complex of space over $X$ with contractible fibres. If $X$ has finitely many type of simplices and is contractible, then $|C(\cX)|$ is contractible.
\label{homotopytype}
\end{prop}

\begin{proof} By the previous theorem, it is enough to show that the projection $p: |C(\cX)| \ra X$ induces isomorphisms on homotopy groups, when $X$ is endowed with its CW topology. For that topology, a continuous map from a compact space to $X$ has its image contained in a finite subcomplex, to which Lemma \ref{quotientequivalence} applies. 
\end{proof}

\section{Construction of $EG$ and $\partial G$ for developable complexes of groups.}
In this section, given a developable simple complex of groups $G(\cY)$ over a finite simplicial complex $Y$, we build a classifying space for its fundamental group. \\

\noindent \textbf{Notation}: We choose once and for all a non-positively curved complex of groups $G(\cY)$ over a finite simplicial complex endowed with a $M_\kappa$-structure, $\kappa \leq 0$. Recall that a complex of groups consists of the data $(G_\sigma, \psi_a, g_{a,b})$ of local groups $(G_\sigma)$, local maps $(\psi_a)$ and twisting elements $(g_{a,b})$. For the background on complexes of groups, we refer the reader to \cite{BridsonHaefliger}. We fix a maximal tree $T$ in the 1-skeleton of the first barycentric subdivision of $Y$, which allows us to define the fundamental group $G = \pi_1( G(\cY), T)$ and the canonical morphism $\iota_T: G(\cY) \ra G$ given by the collection of injections $G_\sigma \ra G$ (see \cite{BridsonHaefliger} p.553). Finally, we define $X$ as the universal covering of $G(\cY)$ associated to $\iota_T$. The simplicial complex $X$ naturally inheritates a $M_\kappa$-structure from that of $Y$ and the simplicial metric $d$ on $X$ makes it a complete geodesic metric space by work of Bridson \cite{BridsonPhD}. This space is CAT(0) by the curvature assumption on $G(\cY)$ (see \cite{BridsonHaefliger} p.562).

\subsection{Construction of $EG$ and $\partial_{Stab} G$.}

\begin{definition}[(cofinite and finite dimensional) classifying space for proper actions]
Let $\Gamma$ be a countable discrete group. A \textit{cofinite and finite dimensional classifying space for proper actions of $\Gamma$} (or briefly a \textit{classifying space for $\Gamma$}) is a contractible CW-complex $E\Gamma$ with a proper cocompact  and cellular action of $\Gamma$, and such that:
\begin{itemize}
\item for every finite subgroup $H$ of $\Gamma$, the fixed point set $E\Gamma^H$ is nonempty and contractible,
\item every infinite subgroup $H$ of $\Gamma$ has an empty fixed point set.
\end{itemize}
\end{definition}

\begin{definition}
A \textit{complex of spaces $D(\cY)$ compatible with the complex of groups $G(\cY)$} consists of the following:
\begin{itemize}
 \item For every vertex $\sigma$ of $\cY$, a space $D_\sigma$ that is a model of classifying space for proper actions $EG_\sigma$ of the local group $G_\sigma$,
 \item For every edge $a$ of $\cY$ with initial vertex $i(a)$ and terminal vertex $t(a)$, an embedding $\phi_a: EG_{t(a)} \hra EG_{i(a)}$ which is $G_{t(a)}$-equivariant, that is, for every $g \in G_{t(a)}$ and every $x \in EG_{t(a)}$, we have 
$$ \phi_{a}(g.x) = \psi_{a}(g).\phi_{a}(x),$$
and such that for every pair $(a,b)$ of composable edges of $\cY$, we have:
$$g_{a,b} \circ \phi_{ab} = \phi_a  \phi_b,$$

\end{itemize}
 We say that $D(\cY)$ extends to an \textit{$E\cZ$-complex of spaces} compatible with the complex of groups $G(\cY)$ if it satisfies the following extra conditions:
\begin{itemize}
 \item Each fibre $D_\sigma= EG_\sigma$ is endowed with an $E\cZ$-structure $(\overline{EG_ \sigma}, \partial G_\sigma),$
 \item The equivariant embeddings $(\phi_a)$ extend to equivariant embeddings $\phi_a:\overline{EG_{i(a)}} \ra \overline{EG_{t(a)}}$, such that for every pair $(a,b)$ of composable edges of $\cY$, we have:
$$g_{a,b} \circ \phi_{ab} = \phi_a  \phi_b.$$
\end{itemize}
\label{EZcomplexofspaces}
\end{definition}
Note that a complex of spaces compatible with the complex of groups $G(\cY)$ is \textit{not} a complex of spaces over $Y$ when the twisting elements $g_{a,b}$ are not tivial. Nonetheless, this data is used to build a complex of spaces over $X$, as explained in the following definition.
\begin{definition}
 We define the space 
$$E(\cY) =  \bigg( G \times \coprod_{\sigma \in V(\cY)} ( \sigma \times EG_\sigma) \bigg) / \simeq$$
where
$$(g, i_{\sigma, \sigma'}(x), s) \simeq \bigg(g \iota_T\big( [\sigma\sigma']\big)^{-1}, x, \phi_{(\sigma \sigma')}(s)\bigg) \mbox{ if } [\sigma\sigma'] \in E(\cY), x \in \sigma', g \in G,$$
$$ (gg', x, s) \simeq (g, x, g's)  \mbox{ if } x \in \sigma, s \in EG_\sigma, g' \in G_\sigma, g \in G.$$

The canonical projection $G \times \coprod_{\sigma \in V(\cY)} ( \sigma \times EG_\sigma) \ra G \times \coprod_{\sigma \in V(\cY)}  \sigma $ yields a map 
$$p: E(\cY) \ra X.$$

The action of $G$ on $G \times \coprod_{\sigma \in V(\cY)} ( \sigma \times EG_\sigma)$ on the first factor by left multiplication yields an action of $G$ on $E(\cY)$, making the projection $p: E(\cY) \ra X$ a $G$-equivariant map.

Note that $E(\cY)$ can be seen as a complex of spaces over $X$, the fibre of a simplex $[g, \sigma]$ being the classifying space $EG_\sigma$. Indeed, for en edge $[g,a]$ of the first barycentric subdivision of $X$, the gluing map $\phi_{[g\iota_T(a)^{-1}, i(a)], [g, t(a)]}: EG_{i(a)} \ra EG_{t(a)}$ is defined as $\phi_{i(a), t(a)}$.

 For every simplex $\sigma$ of $X$, we denote by $EG_\sigma$ the fibre over $\sigma$ of that complex of space. For simplices $\sigma, \sigma'$ of $X$ such that $\sigma' \subset \sigma$, we denote by $\phi_{\sigma', \sigma}: EG_\sigma \ra EG_{\sigma'}$ the associated gluing map.
\end{definition}

\begin{thm}
The space $E(\cY)$ is a classifying space for proper actions of $G$. 
\end{thm}
From now on, we denote by $EG$ this classifying space.

\begin{proof}
\textit{Local finiteness:} Let $\sigma$ be a simplex of $X$ and $U$ be an open set of $EG_\sigma$ that is relatively compact. It is enough to prove that for any injective sequence $(\sigma_n)$ of simplices of $X$ containing $\sigma$ there are only finitely many $n$ such that the image of $\phi_{\sigma, \sigma_n}$ meets $U$. By cocompacity of the action, we can assume that all the $\sigma_n$ are above a same simplex $\overline{\sigma}$ of $Y$. But since the action of $G_\sigma$ on $EG_\sigma$ is proper, it follows that for every simplex $\sigma'$ containing $\sigma$ and every compact subset $K$ of $EG_\sigma$, only finitely many distinct cosets $gEG_{\sigma'}$ in $EG_\sigma$ can meet $K$, hence the result.\\

\textit{Contractibility:} Since $E(\cY)$ is locally finite, $|E(\cY)|$ has the same homotopy groups as $X$ by \ref{homotopytype}, which is contractible by assumption. Thus $|E(\cY)|$ is a CW-complex whose homotopy groups vanish, hence it is contractible.\\

\textit{Cocompact action:} For every simplex $\sigma$ of $Y$, we choose a compact fundamental domain $K_\sigma$ for the action of $G_\sigma$ on $D_\sigma= EG_\sigma$. Now the image in $|E(\cY)|$ of $\bigcup_{\sigma \in S(Y)} \sigma \times K_\sigma$
clearly defines a compact subset of $|E(\cY)|$ meeting every $G$-orbit.\\

\textit{Proper action:} As $|E(\cY)|$ is a locally finite CW-complex, hence a locally compact space, it is enough to show that every finite subcomplex intersects only finitely many of its $G$-translates. \\
Let us first show that for every cell $\sigma$ of $|E(\cY)|$, there are only finitely many $g\in G$ such that $g\sigma = \sigma$. Indeed, let $g\in G$ such that  $g\sigma=\sigma$. The canonical projection $|E(\cY)| \rightarrow X$ is $G$-equivariant and sends a cell of $|E(\cY)|$ on a simplex of $X$, thus $g$ also stabilises the simplex $p(\sigma) \subset X$. Since $G$ acts without inversion on $X$, $g$ pointwise stabilises the vertices of $p(\sigma)$. Let $s$ be such a vertex. Then $g \in G_s$ and, by construction of $|C(\cY, \iota_T)|$, the restriction to $G_s$ of the action of $G$ on $|E(\cY)|$ is just the action of $G_s$ on $C_s = EG_s$. Thus, by definition of a classifying space for proper actions, this implies that there are only finitely many possibilites for $g$. 

Now, let $F$ be a finite subcomplex of $|E(\cY)|$. We have $ \left\{g \in G ~~|~~ gF \cap F \neq \varnothing \right\} \subset \underset{\sigma = \sigma' \mathrm{ mod }~G}{\bigcup_{\sigma, \sigma' \in S(F)}} \left\{g \in G ~~|~~ g\sigma = \sigma' \right\},$ and $\left\{g \in G ~~|~~ g\sigma = \sigma' \right\}$ has the same cardinal as\\ $\left\{g \in G ~~|~~ g\sigma = \sigma \right\}$, which is finite by the previous argument.\\

\textit{Fixed sets:} Let $H$ be a finite subgroup of $G$. As $G$ acts without inversion on the CAT(0) complex $X$, the subset $X^H$ is a nonempty convex subcomplex of $X$. Furthermore, for every simplex $\sigma$ of $X^H$, the subcomplex $(EG_\sigma)^H$ of $EG_\sigma$ is nonempty and contractible. Thus $|E(\cY)|^H$ is the realisation of a complex of spaces over the contractible complex $X^H$ and with contractible fibres, hence it is nonempty and contractible by \ref{homotopytype}.

If $H$ is an infinite subgroup of $G$, we prove by contradiction that $|E(\cY)|^H$ is empty. If this was not the case, there would exist a simplex $\sigma$ fixed pointwise under $H$ and an element $x $ of $EG_\sigma$ that is fixed under $H \subset G_\sigma$. But this is absurd as $(EG_\sigma)^H = \varnothing$ by assumption.
\end{proof}

We now turn to the construction of a boundary of $G$. As stated in \cite{FarrellLafontEZStructures}, the definition of an $E\cZ$-structure only applies to torsionfree groups. Here we use a notion of $\cZ$-structure suitable for groups with torsion, which was introduced by Dranishnikov \cite{DranishnikovBestvinaMessFormula}. 

\begin{definition}[$\cZ$-structures, $E\cZ$-structures]
 Let $\Gamma$ be a discrete group. A $\cZ$-\textit{structure} for $\Gamma$ is a pair $(Y,Z)$ of spaces such that:
\begin{itemize}
 \item $Y$ is a Euclidean retract, that is, a compact, contractible and locally contractible space with finite covering dimension,
 \item $Y \setminus Z$ is a classifying space for proper actions of $\Gamma$,
 \item $Z$ is a $\cZ$-set in $Y$, that is, $Z$ is a closed subpace of $Y$ such that for every open set $U$ of $Y$, the inclusion $U \setminus Z \hra U$ is a homotopy equivalence,
 \item Compact sets \textit{fade at infinity}, that is, for every compact set $K$ of $Y \setminus Z$, every element $z \in Z$ and every neighbourhood $U$ of $z$ in $Z$, there exists a subneighbourhood $V \subset U$ with the property that if a $\Gamma$-translate of $K$ intersects $V$, then it is contained in $U$.
\end{itemize}
The pair $(Y,Z)$ is called an $E\cZ$-\textit{structure} if in addition we have:
\begin{itemize}
 \item The action of $\Gamma$ on $Y \setminus Z$ continuously extends to $Y$.
\end{itemize}
\end{definition}

\begin{definition}
We define the space 

$$\Omega(\cY) =  \bigg( G \times \coprod_{\sigma \in V(\cY)} ( \left\lbrace \sigma \right\rbrace \times \partial G_\sigma) \bigg) / \simeq$$
where
$$ \big(gg', (x, s)\big) \simeq \big(g, (x, g's)\big)  \mbox{ if } x \in \sigma, s \in EG_\sigma, g' \in G_\sigma, g \in G.$$
Note that $\Omega(\cY)$ comes with a natural projection to the set of simplices of $X$. If $\sigma$ is a simplex of $X$, we denote by $\partial G_{\sigma}$ the preimage of $\left\lbrace \sigma \right\rbrace$ under that projection. We now define

$$\partial_{Stab} G = \Omega(\cY)/ \sim $$
where $\sim$ is the equivalence relation generated by the following identifications:
$$\bigg[g, \left\lbrace \sigma\right\rbrace , \xi\bigg] \sim \bigg[ g F( [\sigma\sigma'])^{-1}, \{ \sigma'\} , \phi_{[\sigma\sigma']}(\xi) \bigg] \mbox{ if } g \in G, [\sigma\sigma'] \in E(\cY) \mbox{ and } \xi \in \partial G_\sigma.$$

The action of $G$ on $G \times \coprod_{\sigma \in V(\cY)} ( \left\lbrace \sigma \right\rbrace \times \partial G_\sigma)$ by left multiplication on the first factor yields an action of $G$ on $\Omega(\cY)$ and on $\partial_{Stab} G$.
\label{defboundary}
\end{definition}

\begin{definition}
 We define the spaces $\partial G = \partial_{Stab} G \cup \partial X$ and $\overline{EG} = EG \cup \partial G$.

Moreover, we extend the projection $p: EG \ra X$ to a map $\bar{p}$ from $\overline{EG}$ to the set of subsets of $\overline{X}$ in the following way: 
\begin{itemize}
\item For $\widetilde{x} \in EG $, we define $\bar{p}(z)$ to be the singleton $ \left\{ p(\widetilde{x}) \right\}$.
\item For $\eta \in \partial X$, we define $\bar{p}(\eta )$ to be the singleton $ \left\{ \eta \right\}$.
\item For $\xi \in \partial_{Stab} G$, we set $\bar{p}(\xi) = D(\xi)$.
\item Finally, for  $K \subset \overline{EG}$, we set $\bar{p}(K) = \bigcup_{z \in K} \bar{p}(z)$.
\end{itemize}
\end{definition}

Our aim is to endow $\overline{EG}$ with a topology that makes $(\overline{EG}, \partial G)$ an $E\cZ$-structure for $G$.\\

\noindent \textbf{Notations}: Since the $\phi_{\sigma, \sigma'}$ are embeddings, we will identify $\phi_{\sigma, \sigma'}(\overline{EG_{\sigma'}})$ with $\overline{EG_{\sigma'}}$. For instance, if $U$ is an open subset of $EG_{\sigma'}$, we will simply write " we have $EG_\sigma \subset U$ in $EG_{\sigma'}$" instead of "we have $\phi_{\sigma, \sigma'}(EG_\sigma) \subset U$ in $EG_{\sigma'}$".\\

\begin{center}
\textit{From now on, we assume that there is a complex of spaces $D(\cY)$ that extends to an $E\cZ$-complex of spaces compatible with the complex of groups $G(\cY)$.}
\end{center}

\subsection{Further properties of $E\cZ$-complexes of spaces.}
In this paragraph, we define additional properties of $E\cZ$-complexes of spaces, which will enable us to study the properties of the equivalence relation $\sim$ previously defined.
\subsubsection{The limit set property.}
Recall that for a discrete group $\Gamma$ together with an $E\cZ$-structure $(\overline{E\Gamma}, \partial \Gamma)$ and a subgroup $H$, the \textit{limit set} $\Lambda H$ of $H$ in $\partial \Gamma$ is the set $ \overline{Hx} \cap \partial \Gamma,$ where $x$ is an arbitrary point of $E\Gamma$. 

\begin{definition}[Limit set property for an $E\cZ$-complex of spaces]
We say that the $E\cZ$-complex of spaces $D(\cY)$ compatible with the complex of groups $G(\cY)$ satifies the \textit{limit set property}  if the following conditions are satisfied: 
 \begin{itemize}
  \item For every pair of simplices $\sigma \subset \sigma'$ of $Y$, the map $\phi_{\sigma, \sigma'}$ realises an equivariant homeomorphism from $\partial G_{\sigma'}$ to the limit set $\Lambda G_{\sigma'} \subset \partial G_\sigma$.
  \item For every simplex $\sigma$ of $Y$, and every pair of subgroups $H_1, H_2$ in the family $\cF_\sigma = \left\lbrace  \bigcap_{i=1}^{n} g_i G_{\sigma_i} g_i^{-1}  | ~ g_1, \ldots, g_n \in G_\sigma, \sigma_1, \ldots,  \sigma_n \subset \mbox{st}(\sigma), n \in \bbN \right\rbrace,$ we have $ \Lambda H_1 \cap \Lambda H_2 = \Lambda(H_1 \cap H_2).$
\end{itemize}
\label{limitsetproperty}
\end{definition}

\noindent \textbf{Remarks.}  $(i)$ Let $\Gamma$ be a hyperbolic group, and $H$ a subgroup. Then $H$ is quasiconvex if and only if its limit set in $\partial \Gamma$ is equivariantly homeomorphic to $\partial H$, by a result of Bowditch \cite{BowditchConvergenceGroups}.\\
$(ii)$ Let $\Gamma$ be a hyperbolic group and $\partial \Gamma$ its Gromov boundary. Let $H_1$ and $H_2$ be two quasiconvex subgroups of $\Gamma$. Then $\Lambda H_1 \cap \Lambda H_2 = \Lambda(H_1 \cap H_2)$ by a result of \cite{GromovAsymptotic}.\\

\subsubsection{The finite height property.}

Recall that, for $\Gamma$ a discrete group and $H$ a subgroup, the \textit{height} of $H$ is the supremum of the set of integers $n \in \mathbb N$ such that there exist distinct cosets $\gamma_1 H ,\ldots,\gamma_n H \in G/H$ such that the intersection $\gamma_1 H \gamma_1^{-1} \cap \ldots \cap \gamma_n H \gamma_n^{-1}$ is infinite (see \cite{SageevWidth}). If such a supremum is infinite, we say that $H$ is of \textit{infinite height} in $\Gamma$. Otherwise, $H$ is said to be of \textit{finite height} in $\Gamma$. A quasiconvex subgroup of a hyperbolic group is of finite height, by a result of \cite{SageevWidth}.\\

\begin{definition}[finite height property]
We say that the $E\cZ$-complex of spaces $D(\cY)$ compatible with the complex of groups $G(\cY)$ satifies the \textit{finite height property} if for every pair of simplices $\sigma \subset \sigma'$ of $Y$, $G_{\sigma'}$ is of finite height in $G_\sigma$.
\label{finiteheightproperty}
\end{definition}

\section{The geometry of the action.}
In this section, we gather a few geometric tools that will be used to construct a topology on $\overline{EG} = EG \cup \partial G$. From now on, we assume that the $E\cZ$-complex of spaces $D(\cY)$ compatible with $G(\cY)$ satisfies the limit set property and the finite height property. We further assume that the action of $G$ on $X$ is acylindrical and we fix an \textit{acylindricity constant} $A>0$, that is, a constant such that every subcomplex of $X$ of diameter at least $A$ has a finite pointwise stabiliser.
\subsection{Some geometric results about $M_\kappa$-simplicial complexes.}
\begin{definition}
Let $d$ be the simplicial metric on $X$, and choose a base point $v_0 \in X$. For $x \in X$, we denote by $[v_0, x]$ the unique geodesic segment from $v_0$ to $x$, and by $\gamma_x: [0, d(v_0,x)] \ra X$ its parametrisation. For subsets $K, L$ of $X$, we define $\mbox{Geod}(K,L)$ as the set of points lying on a geodesic segment from a point of $K$ to a point of $L$.
\end{definition}

We recall the following proposition of Bridson.
\begin{prop}[containment lemma, Bridson \cite{BridsonPhD}]
For every $n$ there exists a constant $k$ such that for every finite subcomplex $K \subset X$ spanned by at most $n$ simplices, any geodesic path contained in the open star of $K$ meets at most $k$ simplices. \qed
\label{containment}
\end{prop}

\begin{lem}[finiteness lemma]
Let $X$ be as before. Then
\begin{itemize}
 \item[$(i)$]  for every $n \geq 1$, there exists a constant $k$ such that for every pair of finite connected subcomplexes $K, K'$ spanned by at most $n$ simplices, $\mbox{Geod}(v_0, K)$ meets at most $k$ simplices inside the open star of $K'$.
 \item[$(ii)$] for every finite subcomplex $K \subset X$, $\mbox{Geod}(v_0,K)$ meets only finitely many simplices.
\end{itemize}
\label{finite}
\end{lem}

\begin{proof}
 $(i)$ Since $X$ has finitely many isometry types of simplices, there is a lower bound $l$ on the set of distances from a closed simplex to its link (that is, the boundary of its open star), and an upper bound $L_n$ on the set of diameters of the closed star of finite connected subcomplexes spanned by at most $n$ simplices. Moreover, there exists an integer $m$ such that any finite subcomplex consisting of at mosts $n$ simplices is covered by $m$ balls of radius $\frac{l}{4}$. Now take $k$ to be an integer greater than $4(n+1)(L_n/l +1) +1$. 

We now show by contradiction that $\mbox{Geod}(v_0,K) $ meets at most $d_n$ simplices inside the open star of $K'$. If this was not the case, a pigeonhole argument shows that there exist distinct elements $x_0,\ldots,x_n$ in $K$ and distinct closed simplices $\sigma_0, \ldots, \sigma_n$  in the open star of $K'$ such that the geodesic segment from $v_0$ to $x_i$ meets $\sigma_i$ at a time $t_i$, with $|t_i - t_j| \leq \frac{l}{4}$ for distinct $i,j =0, \ldots, n$. Now since $X$ is a CAT(0) space, 
$$d(\gamma_{x_i}(t_i), \gamma_{x_j}(t_j)) \leq d(\gamma_{x_i}(t_i), \gamma_{x_i}(t_j)) + d(\gamma_{x_i}(t_j), \gamma_{x_j}(t_j)) \leq |t_i - t_j| + 2\frac{l}{4} \leq 3\frac{l}{4}.$$
Hence, by definition of $l$, distinct simplices $\sigma_i$, $\sigma_j$ have a nonempty intersection. In particular, there exists a vertex $v$ in the closed star of $K'$ containing each simplex $\sigma_i, i=0, \ldots, n$. But since $K'$ meets at most $n$ simplices, there can exist only $n$ simplices in the open star of $K'$ and contained in the open star of a single vertex. We thus have our contradiction.\\
$(ii)$ By the containment lemma \ref{containment}, for every finite sucomplex $K$ there exists an integer $n$ such that any geodesic from $v_0$ to $K$ meets at most $n$ simplices. The proof thus boils down to the following property, which we prove by induction on $n$.\\

$(H_n):$ Let $K$ be a finite subcomplex such that any geodesic from $v_0$ to $K$ meets at most $n$ simplices. Then $\mbox{Geod}(v_0,K)$ meets only finitely many simplices.\\

The result is clear for $n=1$. Suppose we have proven it up to rank $n$, and let $K$ be a finite subcomplex such that any geodesic from $v_0$ to $K$ meets at most $n+1$ simplices. Using the previous result, we know that $\mbox{Geod}(v_0,K) \cap \mbox{st}(K)$ meets finitely many simplices. For every $x \in K$, pick the last simplex touched by $[v_0,x]$ before entering the open star of $K$. These simplices span a finite subcomplex such that any geodesic from $v_0$ to $K'$ meets at most $n$ simplices. Now, a simplex meeting $\mbox{Geod}(v_0, K)$ either meets $ \mbox{Geod}\big(v_0, \mbox{st}(K) \setminus \mbox{st}(K) \big)$ or is contained in the open star of $K$. By the induction hypothesis, there are finitely many simplices meeting  $ \mbox{Geod}\big(v_0, \mbox{st}(K) \setminus \mbox{st}(K) \big)$. And by $(i)$, there are finitely many simplices in the open star of $K$ which meet $\mbox{Geod}(v_0,K)$, which concludes the induction step.
\end{proof}

\begin{definition}
For $\xi \in \partial_{Stab} G$, the finiteness lemma \ref{finite} yields an integer $m_\xi$ such that for every $x$ in $X$, the subset $\mbox{Geod}(x, D(\xi))$ meets at most $m_\xi$ simplices in $\mbox{st}(D(\xi))\setminus D(\xi)$.
\end{definition}

Up to rescaling the metric of $X$, we can assume that the constant $l$ defined in the previous proof is at least $1$, which will lighten notations in the rest of the article: 

\begin{center}
\noindent \textit{From now on, we will assume that the distance from any simplex to the boundary of its closed star is at least $1$}.
\end{center}

\subsection{Domains and their geometry.}
In this paragraph, we study the topological properties of the identifications made to build the boundary of $G$. 

\begin{definition}
 Let $\xi \in \partial_{Stab} G$. We define $D(\xi)$, called the \textit{domain} of $\xi$, as the subcomplex spanned by simplices $\sigma$ such that $\xi \in \partial G_\sigma$. We denote by $V(\xi)$ the set of vertices of $D(\xi)$.
\end{definition}

The aim of this paragraph is to prove the following:

\begin{prop}
Domains are finite convex subcomplexes of $X$ whose diameters are uniformly bounded above. 
\label{finitedomain}
\end{prop}

\begin{cor}
 For $\xi \in \partial_{Stab} G$, the containment lemma \ref{containment} yields an integer $d_\xi$ such that any geodesic path inside $\mbox{st}(D(\xi))$ meets at most $d_\xi$ simplices. Furthermore, let $D$ be an upper bound of all the $d_\xi, \xi \in \partial_{Stab} G$.
\end{cor}

Recall that $\Omega(\cY)$ is defined in \ref{defboundary} as the disjoint union of the $\partial G_v$'s ($v \in V(X)$) and that $\partial_{Stab} G$ is a quotient of $\Omega(\cY)$ defined by making identifications along edges of $X$. We start by proving the following proposition: 

\begin{prop}
 Let $v$ be a vertex of X. Then the projection $\pi : \partial G_v \ra \partial G$ is injective.
\label{xiloop}
\end{prop}

\begin{definition}
 Let $\xi \in \partial_{Stab} G$. A $\xi$-\textit{path} is the data $\{(v_i)_{0 \leq i \leq n}, (\xi_i)_{0 \leq i \leq n}, (x_i)_{1 \leq i \leq n }\}$ of: 
\begin{itemize}
 \item a sequence $v_0, \ldots, v_n$ of adjacent vertices of $X$,
 \item a sequence $\xi_0, \ldots, \xi_n$ of elements of $\Omega(\cY)$, such that $\xi \in \partial G_{v_i}$ for every $i$, and such that each $\xi_i$ is in the equivalence class $\xi$,
 \item a sequence $x_1, \ldots, x_n$ of elements of $\Omega(\cY)$, such that $x_i \in \partial G_{[v_{i-1}, v_i]}$ for every $i$, and such that $\phi_{v_{i-1},[v_{i-1},v_i]}(x_i) = \xi_{i-1}$ (resp $\phi_{v_{i},[v_{i-1},v_i]}(x_i) = \xi_{i}$). 
\end{itemize}
To lighten notations, a $\xi$-path will sometimes just be denoted $[v_0, \ldots, v_n]_\xi$. The path in the $1$-skeleton of $X$ induced by a $\xi$-path is called the \textit{support} of $[v_0, \ldots, v_n]_\xi$, and denoted $[v_0, \ldots, v_n]$. If $v_0=v_n$, a $\xi$-path will rather be called a $ \xi$-\textit{loop}.
\end{definition}

\begin{definition}
 A \textit{path of simplices} is a sequence of open simplices $\sigma_1, \ldots, \sigma_n$ such that $\overline{\sigma_i} \subset \overline{\sigma_{i+1}}$ or $\overline{\sigma_{i+1}} \subset \overline{\sigma_{i}}$ for every $i= 1, \ldots, n-1$. Equivalently, it is a finite path in the first barycentric subdivion of $X$. The integer $n$ is called the \textit{length} of the path of simplices.
\end{definition}

\begin{lem}
 Let $v_0, \ldots, v_n$ be vertices of $X$, $H = \cap _{0\leq i \leq n} G_{v_i}$, and $K$ be a connected subcomplex of $X$ pointwise fixed by $H$. Suppose that $H$ is infinite, and let $\xi \in \partial_{Stab} G$ such that, in $G_{v_0}$, we have 
$$\xi \in \Lambda H \subset \partial G_{v_0}.$$ 
Then $\xi \in \Lambda H \subset \partial G_{\sigma}$ for every simplex $\sigma$ of $K$, hence $K \subset D(\xi)$.
\label{stabiliserfix}
\end{lem}
\begin{proof}
 As $K$ is connected, it is enough to prove that for every path of simplices $\sigma_0=v_0, \ldots, \sigma_d$ contained in $K$, we have $\xi \in \partial H \subset \partial G_{\sigma}$. Now this follows from an easy induction on the number of simplices contained in such a path.
\end{proof}

\begin{lem}
Let $\xi \in \partial_{Stab} G$, $[v_0, \ldots, v_n]_\xi$ a $\xi$-path and $H= \cap _{0\leq i \leq n} G_{v_i}$. Then
\begin{itemize}
 \item $H$ is infinite,
 \item $\xi \in \Lambda H \subset \partial G_{v_i}$ for every $i= 0, \ldots, n$.
\end{itemize}
\label{xipath1}
\end{lem}
\begin{proof}
 We show the result by induction on $n \geq 1$. The result is immediate for $n=1$ by definition of $\sim$. Suppose the result true up to rank $n$ and let $\xi \in \partial_{Stab} G$ together with a $\xi$-path $[v_0, \ldots, v_{n+1}]_\xi$. By restriction, we get a $\xi$-path $[v_0 \ldots, v_n]_\xi$ for which the result is true by the induction hypothesis. Thus $\xi \in \Lambda (\cap_{0 \leq i \leq n} G_{v_i}) \subset \partial G_{v_n}$. But since $\xi$ is also in $ \partial G_{[v_n,v_{n+1}]} = \Lambda G_{[v_n,v_{n+1}]}$ by assumption, we get
$$ \xi \in \Lambda \big( \bigcap_{0 \leq i \leq n}  G_{v_i} \big) \cap \Lambda G_{[v_n, v_{n+1}]} = \Lambda \big( \bigcap_{0 \leq i \leq n+1} G_{v_i} \big) \subset \partial G_{v_n},$$
the previous equality following from the limit set property. Now, by \ref{stabiliserfix}, we get $\xi \in \Lambda (\cap _{0\leq i \leq n} G_{v_i}) \subset \partial G_{v_i}$ for every $i= 0, \ldots, n$, which concludes the induction.
\end{proof}

\begin{proof}[Proof of \ref{xiloop}]
 Let $\xi, \xi'$ two elements of $\Omega(\cY)$ being in the image of $\partial G_v$, that are equivalent for the equivalence relation $\sim$. Then there exists a $\xi$-loop $\{(v_i)_{0 \leq i \leq n}, (\xi_i)_{0 \leq i \leq n}, (x_i)_{1 \leq i \leq n }\}$ with $\xi_0=\xi$, $\xi_n=\xi'$. It is enough to prove the result when the support of that $\xi$-loop is injective. Let $Y$ be the set of all points on a geodesic between two elements of $[v_0, \ldots, v_n]$, and $Y'$ the set of all points on a geodesic between two points of $Y$. By the previous lemma, there is an infinite subgroup $H$ of $G$ stabilising pointwise $v_0, \ldots, v_n$. As $X$ is CAT(0), $H$ also stabilises pointwise every element in $Y$. As $[v_0, \ldots, v_n] \subset Y$ is contractible inside $Y'$ (the latter space being pointwise fixed by $H$), the finiteness lemma \ref{finite} implies that we can choose a finite contractible $2$-complex $F$ bordering $[v_0, \ldots, v_n]$ that is pointwise fixed by $H$. We call such a subcomplex a \textit{hull} of the loop $[v_0, \ldots, v_n]$. Hence the result will follow from the following fact, which we now prove by induction.\\

$(H_d)$: For every $\xi \in \partial_{Stab} G$ and every $\xi$-loop $\{(v_i)_{0 \leq i \leq n}, (\xi_i)_{0 \leq i \leq n}, (x_i)_{1 \leq i \leq n }\}$ admitting a hull containing at most $d$ simplices, then $\xi_0 = \xi_n$.\\

If $d=1$, then $n=2$, and the hull considered is just a single triangle $\sigma$. Since $H \subset G_\sigma$ because $H$ stabilises $\sigma$ pointwise, we can choose $x \in \partial G_\sigma$ such that $\phi_{v_1, \sigma}(x) = \xi_1$. From the commutativity of the diagram of embeddings for a simplex, it follows that $\phi_{ [v_0, v_1], \sigma}(x) = x_1$ and $\phi_{ [v_1, v_2], \sigma}(x) = x_2$. Hence $\xi_0 = \phi_{v_0, [v_0, v_1]}(x_1) = \phi_{v_0,\sigma}(x) =\phi_{v_0, [v_2, v_0]}(x_2) = \xi_2$.

Suppose the result true up to rank $d$, and let $\xi \in \partial_{Stab} G$, together with a $\xi$-loop $\{(v_i)_{0 \leq i \leq n}, (\xi_i)_{0 \leq i \leq n}, (x_i)_{1 \leq i \leq n }\}$ admitting a hull $F$ containing at most $d+1$ simplices. Choose any triangle $\sigma$ of $F$ containing the segment $[v_1,v_2]$. As $\sigma$ is stabilised by $H$, we can find $x \in \partial G_\sigma$ such that $\phi_{v_1, \sigma}(x) = \xi_1$. There are now two possible cases, depending of the nature of $\sigma$:
\begin{itemize}
 \item If another side of $\sigma$ is contained in the support of the $\xi$-loop, for example $[v_2,v_3]$, we set $x' = \phi_{ [v_1, v_3], \sigma}(x)$. 

Now the commutativity of the diagram of embeddings for $\sigma$ yields the following new $\xi$-loop
$$\{ (v_0, v_1, v_3, v_4, \ldots, v_n), (\xi_0, \xi_1, \xi_3, \ldots, \xi_n), (x_1,x', x_4, \ldots, x_n)\}.$$
A hull for that new loop is given by the closure of $F \setminus \sigma$, thus containing at most $d$ simplices, and we are done by induction.
\item If no other side of $\sigma$ is contained in the support of the $\xi$-loop, we set $a$ to be the remaining vertex of $\sigma$, $\alpha = \phi_{a, \sigma}(x)$, $x_2 = \phi_{ [v_1,a], \sigma}(x)$ and $x_2' = \phi_{ [a, v_2], \sigma}(x)$.

The commutativity of the diagram of embeddings for $\sigma$ yields the following new $\xi$-loop:
$$\{ (v_0, v_1, a, v_2, \ldots, v_n), (\xi_0, \xi_1, \alpha, \xi_2, \ldots, \xi_n), (x_1,x_2, x_2', x_3, \ldots, x_n)\}.$$
A hull for that new loop is given by the closure of $F \setminus \sigma$, thus containing at most $d$ simplices, and we are done by induction. \qedhere
\end{itemize}
\end{proof}

\begin{proof}[Proof of \ref{finitedomain}]
\textit{Convexity}: Let $x, x'$ be two points of $D(\xi)$. Let $v$ (resp. $v'$) be a vertex of $\sigma_x$ (resp. $\sigma_{x'}$). We can thus find a $\xi$-path $\{(v_i)_{0 \leq i \leq n}, (\xi_i)_{0 \leq i \leq n}, (x_i)_{1 \leq i \leq n }\}$ with $v_0=v$ and $v_n = v'$. As $\xi \in \partial G_{\sigma_{x}}$ and $\xi \in \partial G_{\sigma_{x'}}$, we can assume without loss of generality that its support $[v_0, \ldots, v_n]$ contains all the vertices of $\sigma_x$ and $\sigma_{x'}$. By \ref{xipath1}, this implies that the subgroup $H= \cap_{0 \leq i \leq p} G_{v_i}$ is infinite and that $\xi \in \Lambda H \subset \partial G_{v_0}$. Now since $H$ fixes pointwise all the vertices of $\sigma_x$ and $\sigma_{x'}$, and since $X$ is CAT(0), $H$ also fixes pointwise the geodesic segment $[x,x']$.  But by \ref{stabiliserfix}, the fixed-point set of $H$ is contained in $D(\xi)$, hence so is $[x, x']$. Thus $D(\xi)$ is convex.

\textit{Finiteness}: Let $\sigma$ be a simplex of $D(\xi)$ and $\sigma_1, \sigma_2,\ldots$ be a (possibly empty) sequence of simplices containing strictly $\sigma$ and contained in $D(\xi)$. It follows from the proof of \ref{xiloop} that $\xi \in \partial G_{\sigma_i} \subset \partial G_\sigma$ for every $i$. Thus, the limit set property \ref{limitsetproperty}, the finite height property \ref{finiteheightproperty} and the cocompacity of the action imply that there can be only finitely many such simplices. Thus $D(\xi)$ locally finite. To prove that it is also bounded, consider $x,x'$ two points of $D(\xi)$. By \ref{xipath1} the stabiliser of $\{x,x'\}$ is infinite. Thus the domain of $\xi$ has a diameter bounded above by the acylindricity constant. The complex $D(\xi)$ is locally finite and bounded, hence finite. Moreover, it is clear from the above argument that the bound can be chosen uniform on $\xi$.
\end{proof}

\subsection{Nestings and Families.}
\subsubsection{The convergence property and nestings.} 

\begin{definition}[the convergence property]
 We say that an $E\cZ$-complex of spaces compatible with $G(\cY)$ satisfies the \textit{convergence property} if, for every pair of simplices $\sigma \subset \sigma'$ in $Y$ and every injective sequence $(g_nG_{\sigma'})$ of cosets of $G_\sigma / G_{\sigma'}$, there exists a subsequence such that $(g_{\varphi(n)}\overline{EG_{\sigma'}})$ uniformly converges to a point in $\overline{EG_\sigma}$.
\end{definition}

\begin{center}
\textit{From now on, besides the limit set property, the finite height property and the acylindricity assumption, we assume that the $E\cZ$-complex of spaces $D(\cY)$ satisfies the convergence property.}
\end{center}

\begin{definition}
 Let $\xi \in \partial_{Stab} G$, $v$ a vertex of $D(\xi)$,  and $U$ a neighbourhood of $\xi$ in $\overline{EG_v}$. We say that a subneighbourhood $V \subset U$ containing $\xi$ is \textit{nested} in $U$ if for every simplex of $\mbox{st}(v)$ not contained in $D(\xi)$, we have
$$ \overline{EG_{\sigma}} \cap V \neq \varnothing \Rightarrow \overline{EG_{\sigma}} \subset U.$$
\end{definition}

\begin{lem}[nesting lemma]
 Let $\xi \in \partial_{Stab} G$, $v$ a vertex of $D(\xi)$ and $U$ a neighbourhood of $\xi$ in $\overline{EG_v}$. Then there exists a subneighbourhood of $\xi$ in $\overline{EG_v}$, $V \subset U$, which is nested in $U$.
\label{nesting}
\end{lem}
\begin{proof}
 We show this by contradiction. Consider a countable basis $(V_n)_n$ of neighbourhood of $\xi$ in $\overline{EG_v}$, and suppose that for every $n$, there exists a simplex $\sigma_n \in \mbox{st}(v) \setminus D(\xi)$ such that $\overline{EG_{\sigma_n}} \cap V_n \neq \varnothing$ and $\overline{EG_{\sigma_n}} \subsetneq U$. Up to a subsequence, we can assume that $(\sigma_n)_n$ is injective. By cocompacity of the action, we can also assume that all the $\sigma_n$ are over a unique simplex $\overline{\sigma}$ of $Y$. Now the convergence property implies that there should exist a subsequence $\sigma_{\lambda(n)}$ such that $\overline{EG_{\sigma_{\lambda(n)}}}$ uniformly converges to a point in $\overline{EG_v}$, a contradiction.
\end{proof}

\subsubsection{Families.}
Since, in $\partial G$, boundaries of stabilisers of vertices are glued together along boundaries of stabilisers of edges, we will construct neighbourhoods in $\overline{EG}$ of a point $\xi \in \partial_{Stab} G$ using neighbourhoods of the representatives of $\xi$ in the various $\overline{EG_v}$, where $v$ runs over the vertices of the domain of $\xi$. 

\begin{definition}[$\xi$-family]
 Let $\xi \in \partial_{Stab} G$. A collection $\cU$ of open sets $\left\lbrace U_v, v \in V(\xi) \right\rbrace $ is called a $\xi$-\textit{family} if for every pair of vertices $v,v'$ of $X$ that are joined by an edge $e$ and every $x \in \overline{EG_e}$,
$$\phi_{v,e}(x) \in U_v \Leftrightarrow \phi_{v',e}(x) \in U_{v'}.$$
\label{defxifamily}
\end{definition}

\begin{prop}
 Let $\xi \in \partial_{Stab} G$. For every vertex $v$ of $D(\xi)$, let $U_v$ be a neighbourhood of $\xi$ in $\overline{EG_v}$. Then there exists a $\xi$-family $\cU'$ such that $U_v' \subset U_v$ for every vertex $v$ of $D(\xi)$.
\label{xifamilies}
\end{prop}

\begin{proof}
 For every simplex $\sigma$ of $D(\xi)$, we construct open sets $U_{\sigma}'$ by induction on $\mbox{dim}(\sigma)$, starting with simplices of maximal dimension, that we denote $d$.

If $\mbox{dim}(\sigma)=d$, we set 
$$U_\sigma' = \bigcap_{v \in \sigma} \phi_{v, \sigma}^{-1}(U_v).$$
Assume the $U_\sigma'$ constructed for simplices of dimension at least $k \leq d$, and let $\sigma_0$ be of dimension $k-1$. If no simplex of dimension $\geq k$ contains $\sigma_0$, set 
$$U_{\sigma_0}' = \bigcap_{v \in \sigma} \phi_{v, \sigma_0}^{-1}(U_v).$$
Otherwise, since the $\phi_{\sigma, \sigma'}$ are embeddings, 
$$\displaystyle{\underset{\mathrm{dim}(\sigma)=k}{\bigcup_{\sigma_0  \subset  \sigma \subset D(\xi)}} \phi_{\sigma_0, \sigma}(U_\sigma')}$$
 is open in  
$$\displaystyle{\underset{\sigma_0  \subset  \sigma \subset D(\xi)}{\bigcup_{\mathrm{dim}(\sigma)=k}} \phi_{\sigma_0, \sigma}(\overline{EG_\sigma})}.$$ 
We can thus write it as the trace of an open set $U_{\sigma_0}'$ of $\overline{EG_{\sigma_0}}$. This yields for every vertex $v$ of $D(\xi)$ a new open set $U_v'$. By intersecting it with $U_v$, we can further assume that $U_v' \subset U_v$. This new collection of neighbourhoods clearly satisfies the desired property.
\end{proof}

\begin{definition}
 Let $\xi \in \partial_{Stab} G$, together with two $\xi$-families $\cU,  \cU'$. We say that $\cU'$ is \textit{nested} in $\cU$ if for every vertex $v$ of $D(\xi)$, $U_v'$ is nested in $U_v$. Furthermore we say that $\cU'$ is $n$-nested in $\cU$ if there exist $\xi$-families
$$\cU' = \cU^{[0]} \subset \ldots \subset \cU^{[n]} = \cU$$
with $\cU_i$ nested in $\cU_{i+1}$ for every $i=0,\ldots,n-1$.
\label{defnested}
\end{definition}

\subsection{A geometric toolbox.}
We now prove some results which will be the main tools in forthcoming proofs. Since the proofs of these lemmas relie heavily on the geometry of $X$, we start with a few definitions.
\begin{definition} Let $\xi \in \partial_{Stab} G$, $x \in X$, $\eta \in \partial X$ and $\varepsilon \in (0,1)$. 

 We denote by $[v_0, \eta)$ the unique geodesic ray from $v_0$ to $\eta$, and by $\gamma_\eta: [0, \infty) \ra X$ its parametrisation.

We denote by $D^\varepsilon(\xi)$ the open $\varepsilon$-neighbourhood of $D(\xi)$.

We say that a geodesic in $X$ parametrised by $\gamma$ \textit{goes through} (resp. \textit{enters}) $D^\varepsilon(\xi)$ if there exist $t_0$ such that $\gamma(t_0) \in D^\varepsilon(\xi)$ and $t_1 >t_0$ such that $\gamma(t_1) \notin D^\varepsilon(\xi)$ (resp. if there exists $t_0$ such that $\gamma(t_0) \in D^\varepsilon(\xi)$) .

If the geodesic $[v_0,x]$ goes through $D^\varepsilon(\xi)$, we define an \textit{exit simplex} $\sigma_{\xi, \varepsilon}(x)$ as the first simplex touched by $[v_0,x]$ after leaving $D^\varepsilon(\xi)$.
If $x \in D^\varepsilon(\xi)$, we set $\sigma_{\xi, \varepsilon}(x) = \sigma_x.$
\end{definition}
Note that, by the assumption on the distance from a simplex to the boundary of its closed star, we always have $D^\varepsilon(\xi) \subset \mbox{st}(D(\xi))$.

\begin{definition} Let $\xi \in \partial_{Stab} G$ , $\cU$ a $\xi$-family and $\varepsilon \in (0,1)$.
We define $\mbox{Cone}_{\cU, \varepsilon}(\xi)$ (resp. $\widetilde{\mbox{Cone}}_{\cU, \varepsilon}(\xi)$) as the set of points $x$ of $X$ such that the geodesic $[v_0,x]$ goes through (resp. enters) $D^\varepsilon(\xi)$ and such that for some vertex $v$ of $D(\xi)$ (hence for every by \ref{xifamilies}) contained in the exit simplex $\sigma_{\xi, \varepsilon}(x)$, we have, in $\overline{EG_v}$: 
$$\overline{EG_{\sigma_{\xi, \varepsilon}(x)}} \subset U_v.$$
\end{definition}

\begin{definition}
 For $\xi \in \partial_{Stab} G$ and $\cU$ a $\xi$-family (definition \ref{defxifamily}), we call $\mbox{st}_\cU(\xi)$ the subcomplex spanned by simplices $\sigma \subset \mbox{st}(D(\xi))$ such that for some (hence for every) vertex $v$ of $D(\xi) \cap \sigma$, we have, in $\overline{EG_v}$:
$$\overline{EG_\sigma} \cap U_v \neq \varnothing.$$
\end{definition}

\subsubsection{The crossing lemma.}

\begin{lem}[crossing lemma]
Let $\xi \in \partial_{Stab} G$, $\cU$, $\cU'$ two $\xi$-families, and $\sigma_1, \ldots, \sigma_n$ ($n \geq 1$) a path of open simplices contained in $\mbox{st}\big( D(\xi)\big) \setminus D(\xi)$. Suppose that $\cU'$ is $n$-nested in $\cU$ (definition \ref{defnested}), and that $\sigma_1 \subset \mbox{st}_{\cU'}\big(D(\xi)\big)$. Then for every $k \in \left\lbrace 1, \ldots, n \right\rbrace$ and every vertex $v$ of $D(\xi)$ contained in  $\sigma_k$, we have  $\overline{EG_{\sigma_k}} \subset U_v$ in $\overline{EG_v}$.
\label{passage}
\end{lem}

\begin{proof}
 We prove the result by induction on $n$, by using the properties of nested families. 

The result for $n=1$ follows from the definition of a nested family. Suppose the result true for $1, \ldots, n$, and let $\sigma_1, \ldots, \sigma_{n+1}$ be a path of simplices in $\mbox{st}\big( D(\xi)\big) \setminus D(\xi)$ and $\cU^{[0]} \subset \ldots \subset \cU^{[n+1]} = \cU$. By induction, the result is true for the path $\sigma_1, \ldots, \sigma_{n}$ and the filtration $\cU^{[0]} \subset \ldots \subset \cU^{[n]}$, so the only inclusions to be proved are the aforementionned ones for $\sigma_{n+1}$.

 If $\sigma_{n} \subset \sigma_{n+1}$, every vertex $v$ of $\sigma_{n}$ is also a vertex of $\sigma_{n+1}$, so the result is already true for vertices of $D(\xi)$ contained in $\sigma_n$. Now by the properties of $\xi$-families (see \ref{xifamilies}), this implies the result for every vertex of $D(\xi) \cap \sigma_{n+1}$.

 Suppose now that $\sigma_n \supset \sigma_{n+1}$, and let $v$ be a vertex of $D(\xi)$ contained in $\sigma_{n+1}$. Since $v$ is also in $\sigma_n$, $\overline{EG_{\sigma_{n}}} \subset U_{v_d}^{[n]}$ in $\overline{EG_{\sigma_n}}$, so we have $\overline{EG_{\sigma_{n+1}}} \cap  U_{v_n}^{[n]} \neq \varnothing$, which in turn implies $\overline{EG_{\sigma_{n+1}}} \subset  U_{v}^{[n+1]} $ since $\cU^{[n]}$ is nested in $\cU^{[n+1]}$. Now by the properties of $\xi$-families \ref{xifamilies}, the same result holds for every vertex $v$ of $D(\xi)$ contained in $\sigma_{n+1}$. \\
\end{proof}

\subsubsection{The geodesic reattachment lemma.}

\begin{lem}
Let $\xi \in \partial_{Stab} G$. There exists a $\xi$-family $\cV_\xi$ such that for every vertex $v$ of $D(\xi)$ and every simplex $\sigma$ of $\big(\mbox{st}(v) \setminus D(\xi)\big) \cap \mbox{Geod}\big( v_0, D(\xi) \big)$, we have $(V_\xi)_v \cap \overline{EG_\sigma} = \varnothing$.
\label{lemmesansnom5}
\end{lem}

\begin{proof}
Let $v$ be a vertex of $D(\xi)$. For every simplex $\sigma \subset \big(  \mbox{st}(v) \cap \mbox{Geod}(v_0, D(\xi)) \big)  \setminus D(\xi)$, let $V_{v, \sigma}$ be a neighbourhood of $\xi$ in  $\overline{EG_v}$  disjoint from $\overline{EG_{\sigma}}$. As $\big(  \mbox{st}(v) \cap \mbox{Geod}(v_0, D(\xi)) \big)  \setminus D(\xi)$ meets only finitely many simplices by \ref{finite}, set
$$ V_v = \bigcap_{\sigma \subset \big(  \mathrm{st}(v) \cap \mathrm{Geod}(v_0, D(\xi)) \big)  \setminus D(\xi)} V_{v, \sigma}, $$
and we can take $\cV_\xi$ to be a $\xi$-family contained in the collection of open sets $\left\lbrace V_v , v \in V(\xi) \right\rbrace$.
\end{proof}

Recall that by \ref{containment} and \ref{finite} we chose integers $d_\xi$ and $m_\xi$ such that a geodesic segment contained in the open star of $D(\xi)$ meets at most $d_\xi$ simplices, and such that for every $x$ in $X$ the subset $\mbox{Geod}(x, D(\xi))$ meets at most $m_\xi$ simplices in the open star of $D(\xi)$.

\begin{lem}
Let $\xi \in \partial_{Stab} G$. Let $\cU_\xi$ be a $\xi$-family that is $(m_\xi+d_\xi)$-nested in $\cV_\xi$. Let $x \in X \setminus D(\xi)$ such that there exists a simplex $\sigma \subset \bigg( \mbox{st}(D(\xi)) \setminus D(\xi) \bigg)$ that meets $\mbox{Geod}\big(x, D(\xi) \big)$ and such that for some (hence any) vertex $v$ of $\sigma \cap D(\xi)$ we have $ \overline{EG_\sigma} \subset (U_\xi)_v$. Then $x \notin \mbox{Geod}(v_0, D(\xi))$. 
\label{lemmesansnom4}
\end{lem}

\begin{proof}
We prove the lemma by contradiction. Let $x$ and $\sigma$ be as in the statement of the lemma. Let $z \in D(\xi)$ such that $x \in [v_0, z]$. Let $\sigma'$ be the last simplex touched by $[v_0, z]$ before entering $D(\xi)$, and $v'$ a vertex of $\sigma'$. By construction of $m_\xi$, there exists a path of simplices of length at most $m_\xi$ in $\mbox{st}(D(\xi)) \setminus D(\xi)$ between $\sigma$ and $\sigma'$. As $\overline{EG_\sigma}\subset (U_\xi)_v$, the crossing lemma \ref{passage} implies that  $\overline{EG_{\sigma'}} \subset (V_\xi)_{v'}$, which contradicts the definition of $\cV_\xi$.
\end{proof}

\begin{center}
\textit{From now on, every $\xi$-family will be assumed to be contained in $\cU_\xi$.}
\end{center}

The next lemma gives a useful criterion that ensures that a given path is a global geodesic. 
\begin{lem}[geodesic reattachment]
 Let $ \xi \in \partial_{Stab} G$ and $x \in X \setminus D(\xi)$. Suppose that there exists a simplex $\sigma \subset \bigg( \mbox{st}(D(\xi)) \setminus D(\xi) \bigg)$ that meets $\mbox{Geod}\big(x, D(\xi) \big)$ such that for some (hence any) vertex $v$ of $\sigma \cap D(\xi)$ we have $ \overline{EG_\sigma} \subset (U_\xi)_v.$ Then $[v_0, x]$ meets $D(\xi)$. Furthermore, if there exist a $\xi$-family $\cU$, and a $\xi$-family $\cU'$ that is $(m_\xi + d_\xi)$-nested in $\cU$ such that for some (hence every) vertex $v$ of $D(\xi) \cap \sigma$, we have $\overline{EG_\sigma} \subset U_v'$, then $x \in \widetilde{\mbox{Cone}}_{\cU, \varepsilon}(\xi) $ for every $\varepsilon \in (0,1)$.
\label{reattachment}
\end{lem}

In such a case, the geodesic from $v_0$ to $x$ meets $D(\xi)$, and is the concatenation of a geodesic segment in $\mbox{Geod}(v_0, D(\xi))$ and a geodesic in $\mbox{Geod}(D(\xi), x)$.

\begin{proof} 
Let $K = \mbox{Geod}(v_0, D(\xi) ) \cup \mbox{Geod}(D(\xi),x)$ and let $[v_0,x]_K$ be the geodesic from $v_0$ to $x$ in $K$ (which meets finitely many simplices by \ref{finite}). Our aim is to prove that $[v_0, x]_K = [v_0, x]$. By \ref{lemmesansnom4}, $x \notin \mbox{Geod}(v_0, D(\xi))$. As $D(\xi)$ is convex by \ref{finitedomain}, let $v_1, v_2 \in D(\xi)$ be such that $[v_0,x]_K = [v_0, v_1] \cup [v_1, v_2] \cup [v_2,x]$ and such that $[v_0, v_1)$ and $(v_2, x]$ do not meet $D(\xi)$. Let $\varepsilon \in (0,1)$. Let $a \in [v_0, v_1]$ such that $d(a, v_1)= \varepsilon$. If $x \notin D^\varepsilon(\xi)$ let $b \in [v_2, x]$ be such that $d(v_2,b)=\varepsilon$. Otherwise, let $b=x$.
Since $X$ is CAT(0), it is enough to prove that $[v_0,x]_K$ is a local geodesic at every point. We already have that $[v_0,v_1] \cup [v_1,v_2]$ and $[v_1, v_2] \cup [v_2,x]$ are geodesics, so it is sufficient to prove the result when $v_1 = v_2$. We thus have 
$$[v_0,x]_K = [v_0,v_1] \cup [v_1,x], $$
with $[v_0,v_1] \subset \mbox{Geod}(v_0, D(\xi))$ and $[v_1,x] \subset \mbox{Geod}(D(\xi), x)$. Assume by contradiction that $[v_0,x]_K$ is not a local geodesic at $v_1$. Then the geodesic segment $[a,b]$ does not meet $D(\xi)$. This geodesic segment yields a path of simplices between $\sigma_a $ and $\sigma_b$ of length at most $d_\xi$ in $\mbox{st}(D(\xi))   \setminus D(\xi)$.  
Furthermore,  there is a path of simplices between $\sigma$ and $\sigma_b$ of length at most $m_\xi$ in $\mbox{st}(D(\xi))   \setminus D(\xi)$. 
Thus, there is a path of simplices between $\sigma$ and $\sigma_a$ of length at most $m_\xi + d_\xi$ in $\mbox{st}(D(\xi))   \setminus D(\xi)$. But since $\overline{EG_b} \subset (U_\xi)_v$ and $\cU_\xi$ is $(m_\xi+d_\xi)$-nested in $\cV_\xi$, the crossing lemma \ref{passage} implies $\overline{EG_a} \subset V_v$, which is absurd since $\sigma_a \subset \big( \mbox{Geod}(v_0, D(\xi)) \cap \mbox{st}(D(\xi)) \big) \setminus D(\xi)$.

Thus $[v_0, x]_K = [v_0, x]$ and $\sigma_b = \sigma_{\xi, \varepsilon}(x)$. If there exist a $\xi$-family $\cU$, and a $\xi$-family $\cU'$ that is $(m_\xi + d_\xi)$-nested in $\cU$ such that for some (hence every vertex $v$ of $D(\xi) \cap \sigma$, we have $\overline{EG_\sigma} \subset U_v'$, then it follows from the above discussion that for some (hence every) vertex $v'$ of $\sigma_{\xi, \varepsilon}(x)$ we have $\overline{EG_{\sigma_{\xi, \varepsilon}(x)}} \subset U_{v'}$, hence $x \in \widetilde{\mbox{Cone}}_{\cU, \varepsilon}(\xi).$
\end{proof}

As a consequence, we get the following:
\begin{cor}
 Let $\xi \in \partial_{Stab} G$, $\cU$ a $\xi$-family and $\varepsilon \in (0,1)$. Then for every $x \in \widetilde{\mbox{Cone}}_{\cU, \varepsilon}(\xi)$, the geodesic segment $[v_0,x]$ meets $D(\xi)$.
\label{Uxi} 
\end{cor}

\begin{proof}
By \ref{lemmesansnom4}, we get $x \notin \mbox{Geod}(v_0, D(\xi))$. Let $y$ be a point of $\sigma_{\xi, \varepsilon}(x) \cap [v_0, x] \cap D^\varepsilon(\xi)$. It follows from the geodesic reattachment lemma \ref{reattachment} that $[v_0, y]$, hence $[v_0, x]$, meets $D(\xi)$. 
\end{proof}

\subsubsection{The refinement lemma.}

\begin{definition}
Let $\xi \in \partial_{Stab} G$, $\cU$ a $\xi$-family and $n \geq 1$. By \ref{finite}, let $k$ be such that for every connected finite subcomplex $F$ containing at most $n$ simplices, $\mbox{Geod}(v_0, F) \cap \mbox{st}(D(\xi))$ meets at most $k$ simplices. A $\xi$-family that is $k$-nested in $\cU$ is said to be $n$-\textit{refined} in $\cU$.
\end{definition}

\begin{lem}[refinement lemma]
Let $\xi \in \partial_{Stab} G$, $\cU$ a $\xi$-family and $n \geq 1$. Let $\cU'$ be a $\xi$-family which is $n$-refined in $\cU$. Then the following holds:

For every $\varepsilon \in (0,1)$ and every path of simplices $\sigma_1, \ldots, \sigma_n$ in $X \setminus D(\xi)$ such that there exists an element $x_1 \in \sigma_1$ such that $[v_0,x_1]$ enters $D^\varepsilon(\xi)$ and $\sigma_{\xi, \varepsilon}(x_1) \subset \mbox{st}_{\cU'}\big(D(\xi)\big)$, then we have 
$$\sigma_1, \ldots, \sigma_n \subset \widetilde{\mbox{Cone}}_{\cU, \varepsilon}(\xi).$$
\label{refinement}
\end{lem}
\begin{proof}
Let us prove that for every $x \in \cup_{1 \leq i \leq n} \sigma_i$, the geodesic segment $[v_0,x]$ meets $D(\xi)$. 
Let $x_1 \in \sigma_1$ such that $\sigma_{\xi, \varepsilon}(x_1) \subset \mbox{st}_{\cU'}(\xi)$. Note that the reattachment lemma \ref{reattachment} implies that $[v_0,x_1]$ meets $D(\xi)$. Let $y \in D(\xi)$ be the last element of $D(\xi)$ touched by $[v_0, y]$ before leaving $D(\xi)$.

Let $x \in \cup_{1 \leq i \leq n} \sigma_i$, and let $\sigma$ be a simplex of $\mbox{st}(D(\xi)) \setminus D(\xi)$ touched by $[y, x]$ after leaving $D(\xi)$. By retracting a path $\gamma$ between $x_1$ and $x$ in $\cup_{1 \leq i \leq n} \sigma_i$ along geodesic segments $[y,z]$ for $z$ in the support of $\gamma$ (see Figure 6), we get a path of simplices between $\sigma_{\xi, \varepsilon}(x_1)$ and $\sigma$ of length at most $k$. 

\begin{center}
\begingroup%
  \makeatletter%
  \providecommand\color[2][]{%
    \errmessage{(Inkscape) Color is used for the text in Inkscape, but the package 'color.sty' is not loaded}%
    \renewcommand\color[2][]{}%
  }%
  \providecommand\transparent[1]{%
    \errmessage{(Inkscape) Transparency is used (non-zero) for the text in Inkscape, but the package 'transparent.sty' is not loaded}%
    \renewcommand\transparent[1]{}%
  }%
  \providecommand\rotatebox[2]{#2}%
  \ifx\svgwidth\undefined%
    \setlength{\unitlength}{203.115439bp}%
    \ifx\svgscale\undefined%
      \relax%
    \else%
      \setlength{\unitlength}{\unitlength * \real{\svgscale}}%
    \fi%
  \else%
    \setlength{\unitlength}{\svgwidth}%
  \fi%
  \global\let\svgwidth\undefined%
  \global\let\svgscale\undefined%
  \makeatother%
  \begin{picture}(1,1.00981567)%
    \put(0,0){\includegraphics[width=\unitlength]{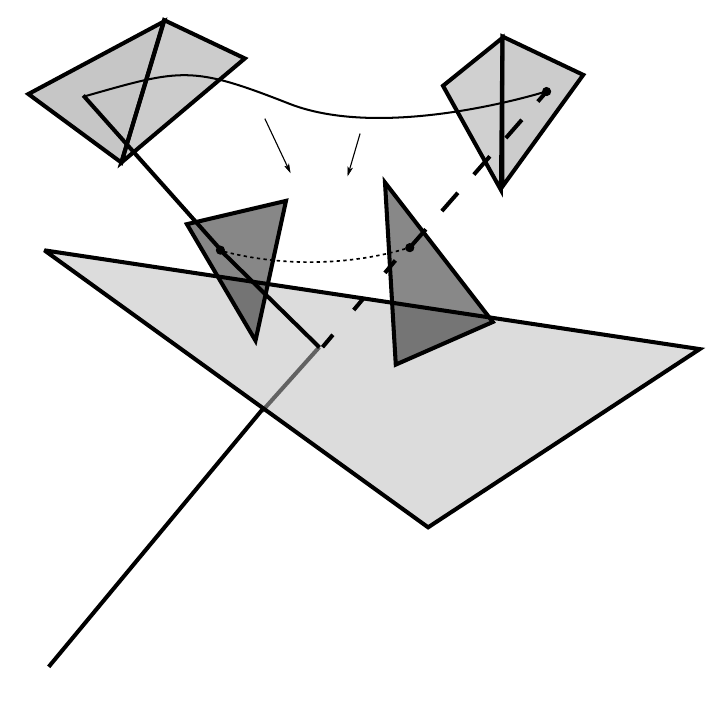}}%
    \put(-0.00556493,0.16292268){\color[rgb]{0,0,0}\makebox(0,0)[lt]{\begin{minipage}{0.35011979\unitlength}\raggedright $v_0$\end{minipage}}}%
    \put(0.05809325,0.97854268){\color[rgb]{0,0,0}\makebox(0,0)[lt]{\begin{minipage}{0.35011982\unitlength}\raggedright $x_1$\end{minipage}}}%
    \put(0.78618329,1.01832891){\color[rgb]{0,0,0}\makebox(0,0)[lt]{\begin{minipage}{0.21484623\unitlength}\raggedright $x$\end{minipage}}}%
    \put(0.77822603,0.34593978){\color[rgb]{0,0,0}\makebox(0,0)[lt]{\begin{minipage}{0.21086763\unitlength}\raggedright $D(\xi)$\end{minipage}}}%
    \put(0.65886699,0.66423048){\color[rgb]{0,0,0}\makebox(0,0)[lt]{\begin{minipage}{0.10742311\unitlength}\raggedright $\sigma$\end{minipage}}}%
    \put(0.39627719,0.71993136){\color[rgb]{0,0,0}\makebox(0,0)[lt]{\begin{minipage}{0.14323078\unitlength}\raggedright $\sigma_{\xi, \varepsilon}(x_1)$\end{minipage}}}%
  \end{picture}%
\endgroup%

Figure 1.
\end{center}

As $\cU'$ is $k$-nested in $\cU$, the crossing lemma \ref{passage} implies that for some (hence every) vertex $v$ of $\sigma$, we get $\overline{EG_\sigma} \subset U_v \subset (U_\xi)_v$. Thus the geodesic reattachment lemma \ref{reattachment} implies that $[v_0, x]$ meets $D(\xi)$.\\

Now let $x \in \cup_{1 \leq i \leq n} \sigma_i$, and let $ \gamma $ be a path in $\cup_{1 \leq i \leq n} \sigma_i$ from $x_1$ to $x$. By retracting this path along the geodesic segments $[v_0,y]_K$, for $ y $ on the support of $\gamma$, we get a path $\gamma' \subset D^\varepsilon(\xi)\setminus D(\xi)$ from $\sigma_{\xi, \varepsilon}(x_1)$ to $\sigma_{\xi, \varepsilon}(x)$. This path yields a path of simplices between $\sigma_{\xi, \varepsilon}(x_1)$ and $\sigma_{\xi, \varepsilon}(x)$ of length at most $k$ in $\mbox{st}(D(\xi)) \setminus D(\xi)$. Now since $\cU'$ is $k$-nested in $\cU$, the crossing lemma \ref{passage} implies that for every vertex $v$ of $\sigma_{\xi, \varepsilon}(x)$, we have $ \overline{EG_{\sigma_{\xi, \varepsilon}(x)}} \subset U_v,$ hence $x \in \widetilde{\mbox{Cone}}_{\cU, \varepsilon}(\xi)$.
\end{proof}

\subsubsection{The star lemma.}
\begin{lem}
 Let $x, y$ be points of $X$. As $[x,y]$ meets only finitely many simplices by \ref{finite}, there exists $\delta > 0$ such that the open $\delta$-neighbourhood $[x,y]^\delta$ of the geodesic segment $[x,y]$ satisfies
$$ [x,y]^\delta \subset \underset{\sigma \subset \mathrm{Span}([x,y])}{\bigcup} \mbox{st}(\sigma). $$ \qed
\label{lemmesansnom3}
\end{lem}

\begin{lem}[star lemma]
 Let $\xi \in \partial_{Stab} G$, $\varepsilon \in (0,1)$ and $x \in X \setminus D^{\varepsilon}(\xi)$ such that the geodesic segment $[v_0,x]$ goes through $D^{\varepsilon}(\xi)$. Then there exists $\delta > 0$ such that for every $y \in B(x, \delta) \setminus D^{\varepsilon}(\xi)$, the geodesic segment $[v_0,y]$ goes through $D^{\varepsilon}(\xi)$. Furthermore, for every $y \in B(x, \delta) \setminus D^{\varepsilon}(\xi)$, we have 
$$\sigma_{\xi, \varepsilon}(y) \subset \mbox{st}(\sigma_{\xi, \varepsilon}(x)).$$
\label{star}
\end{lem}
\begin{proof}
 Let $T= \mbox{dist}(v_0, x)$, and let $\gamma_x: [0, T] \ra X$ be the parametrisation of the geodesic segment $[v_0,x]$. Let $t_0 >0$ such that $[v_0,x]$ leaves $D^\varepsilon(\xi)$ at time $t_0$. Since $D(\xi)$ is convex by \ref{finitedomain}, the map $z \mapsto \mbox{dist}(z,D(\xi))$ is convex, hence there exists $r>0$ such that 
\begin{align*}
 \gamma_x\bigg([t_0-r,t_0) \bigg)  &\subset D^\varepsilon(\xi),
\\ \gamma_x  \big( \left[ t_0-r, t_0 +r \right] \cap [0,T] \big) &\subset \mbox{st}\left(\sigma_{\xi, \varepsilon}(x) \right). 
\end{align*}
Using the third inclusion and \ref{lemmesansnom3}, we can choose $\tau>0$ such that for every $y_-, y_+$ in the $\tau$-neighbourhood of $\gamma_x  \big( \left[ t_0-r, t_0 +r \right] \cap [0,T] \big)$, the geodesic segment $[y_-,y_+]$ is contained in $\mbox{st}(\sigma_{\xi, \varepsilon}(x))$.\\
Let 
$$k = \frac{1}{4}\bigg( \varepsilon - \mbox{dist}(\gamma_x(t_0-r), D(\xi))\bigg). $$
We finally set $\delta =\displaystyle{ \frac{1}{100}\mbox{min}( k, \tau, r)}$. Furthermore, in the case $x \notin \overline{D^\varepsilon(\xi)}$, we can assume without loss of generality that $\delta < \frac{1}{100}(T-t_0)$.
Let $y \in B(x, \delta) \setminus D^\varepsilon(\xi)$, and let $\gamma_y$ be its parametrisation. \\
Since $\delta \leq \frac{r}{2}$, we have $d(v_0,y) \geq t_0 - r$. By convexity of $z \mapsto d(z,D(\xi))$, the inequality $\delta < k$ now implies
\begin{align*}
 d\bigg( \gamma_y(t_0-r), D(\xi) \bigg) &\leq d\bigg( \gamma_x(t_0-r), D(\xi) \bigg)  + d\bigg( \gamma_x(t_0-r), \gamma_y(t_0 - r) \bigg)
\\ &< (\varepsilon - 400\delta ) + 3\delta
\\ &< \varepsilon,
\end{align*} 
so $\gamma_y(t_0-r) \in D^\varepsilon(\xi)$. Since $y \notin D^\varepsilon(\xi)$, it follows that the geodesic segment $[v_0,y]$ goes through $D^\varepsilon(\xi)$ and leaves it for some $t_1 \geq t_0-r$. Now, $\gamma_x$ and $\gamma_y$ parametrise geodesics starting at $v_0$ and such that $d(x,y) \ < \delta$, so since $X$ is a CAT(0)-space, we get  $d(\gamma_x(t_0-r), \gamma_y(t_0-r)) \leq 3 \delta \leq \tau $. 

Moreover,  after leaving $D^\varepsilon(\xi)$ the geodesic $[v_0,y]$ meets the $10\delta$-ball centered at $\gamma_x(t_0)$ for some $t_2 \geq t_0 - r$. Indeed, this is clear if $x \in \overline{D^\varepsilon(\xi)}$; If $x \notin \overline{D^\varepsilon(\xi)}$, then $[v_0,y]$ meets the $3\delta$-ball centered at $\gamma_x(t_0 + 5\delta)$, which is contained in $(X \setminus D^\varepsilon(\xi)) \cap B(\gamma_x(t_0), 10\delta)$. Hence, $[v_0,y]$ meets the $B(\gamma_x(t_0), \tau) \setminus D^\varepsilon(\xi)$. By definition of $\tau$, it thus follows from \ref{lemmesansnom3} that 
$$\gamma_y\bigg(  \left[ t_0-r, t_2  \right]  \bigg) \subset \mbox{st}\left(\sigma_{\xi, \varepsilon}(x) \right),$$ which implies $\sigma_{\xi, \varepsilon}(y) \subset \mbox{st}\big( \sigma_{\xi, \varepsilon}(x) \big)$.
\end{proof}

The star lemma \ref{star} immediately implies the following:

\begin{cor}
 Let $\xi \in \partial_{Stab} G$, $\cU$ a $\xi$-family and $\varepsilon \in (0,1)$. Then the set $\mbox{Cone}_{\cU, \varepsilon}(\xi)$ is open in $X$. \qed
\label{opencones}
\end{cor}

\subsection{Definition of the topology.}
In this paragraph, we define a topology on $\overline{EG}$, by defining a basis of open neighbourhoods at every point. Since points of $\overline{EG}$ are of three different kinds ($EG$, $\partial X$ and $\partial_{Stab} G$), we treat these cases separately.
\begin{definition}
 Let $\widetilde{x} \in EG$. We define a basis of neighbourhoods of $\widetilde{x}$ in $\overline{EG}$, denoted $\cO_{\overline{EG}}(\widetilde{x})$, as the set of open sets of $EG$ containing $\widetilde{x}$.
\end{definition}

We now turn to the case of elements in the boundary of $X$. Recall that since $X$ is a simplicial CAT(0) space with countably many simplices, the bordification $\overline{X} = X \cup \partial X$ has a natural metrisable topology, though not necessarily compact if $X$ is not locally finite. For every $\eta \in \partial X$, a basis of neighbourhoods of $\eta$ in that bordification is given by the family of
$$V_{r,\delta}(\eta) = \left\lbrace x  \in X \bigg|  d(v_0,x)>r ~~\mathrm{and}~~ \gamma_x(r) \in B(\gamma_\eta(r), \delta)     \right\rbrace, ~ r, \delta >0. $$
\textbf{Remark:} For $r, \delta>0$, $\eta \in \partial X$ and if $\gamma$ is the parametrisation of a geodesic such that there exists $T\geq 0$ with $\gamma(T) \in V_{r, \delta}(\eta)$, then $\gamma(t) \in V_{r, \delta}(\eta)$ for every $t \geq T$.\\

We denote by $\cO_{\overline{X}}(\eta)$ this basis of neighbourhoods of $\eta$ in $\overline{X}$. Endowed with that topology, $\overline{X}$ is a secound countable metrisable space (see \cite{BridsonHaefliger}).

Note that the topology of $\overline{X}$ satisfies the following properties:

\begin{lem}
Let $\eta \in \partial X$. Then there exists a basis of neighbourhoods $(U_n)$ of $\eta$ in $\overline{X}$ such that $U_n$ and $U_n \setminus \partial X$ are contractible for every $n \geq 0$. 
\label{contractibleneighbourhoods}
\end{lem}

\begin{proof}
For $r, \delta >0$, let $U_{r, \delta}(\eta) = V_{r, \delta}(\eta) \cup B(\gamma_\eta(r), \delta)$. This defines a basis of neighbourhood of $\eta$ in $\overline{X}$. As $U_{r, \delta}(\eta)$ can be retracted by strong deformation along geodesics starting at $v_0$ onto $B(\gamma_{\eta}(r), \delta)$, it is contractible. Furthermore, as $U_{r, \delta}(\eta)$ can be retracted by strong deformation onto $U_{r, \delta}(\eta) \setminus \partial X$, the same holds for $U_{r, \delta}(\eta) \setminus \partial X$. 
\end{proof}

\begin{lem}
Let $\eta \in \partial X$, $U$ a neighbourhood of $\eta$ in $\overline{X}$ and $k \geq 0$. Then there exists a neighbourhood $U'$ of $\eta$ in $\overline{X}$ that is contained in $U$ and such that $d(U' \cap X, X \setminus U) >k.$
\label{CAT(0)nesting}
\end{lem}

\begin{proof}
The definition of the topology of $\overline{X}$ implies the following: if $(x_n)$ and $(y_n)$ are two sequences of $X$ such that $d(x_n,y_n)$ is bounded, then $(x_n)$ converges to a point of $\partial X$ if and only if $(y_n)$ converges to the same point. Reasoning by contradiction thus implies the lemma.
\end{proof}

\begin{definition}
 Let $\eta \in \partial X$, and let $U$ be a neighbourhood of $\eta$ in $\overline{X}$. We set
$$V_U(\eta) = p^{-1}(U \cap  X) \cup (U \cap \partial X) \cup \left\lbrace  \xi \in \partial_{Stab} G | D(\xi) \subset U \right\rbrace.$$
When $U$ runs over the basis $\cO_{\overline{X}}(\eta)$ of neighbourhoods of $\eta$ in $\overline{X}$, the above formula defines a collection of neighbourhoods for $\eta$ in $\overline{EG}$, denoted $\cO_{\overline{EG}}(\eta)$.
\end{definition}
We finally define open neighbourhoods for points in $\partial_{Stab} G$.
\begin{definition}
Let $\xi \in \partial_{Stab} G$, $\cU \subset \cU_\xi$ be a $\xi$-family, and $\varepsilon \in (0,1)$. A neighbourhood $V_{\cU, \varepsilon}(\xi)$ is defined in four steps as follows:
\begin{itemize}
\item Let $W_{\cU, \varepsilon}(\xi)$ be the set of points $\widetilde{x} \in EG$ whose projection $x\in X$ belongs to $D^\varepsilon(\xi)$ and is such that for some (hence every) vertex $v$ of $\sigma_x$, we have $\phi_{v, \sigma_x}(\widetilde{x}) \in U_v.$
\item Let $W_1$ be the set of points of $EG$ whose projection in $X$ belongs to $\mbox{Cone}_{\cU, \varepsilon}(\xi)$.
\item Let  $W_2$ be the set of points of $\partial X$ that belong to $\mbox{Cone}_{\cU, \varepsilon}(\xi)$.
\item Let $W_3$ be the set of points of $\partial_{Stab} G$ whose domain is contained in $\widetilde{\mbox{Cone}}_{\cU, \varepsilon}(\xi)$.
\end{itemize}
Now set
$$V_{\cU, \varepsilon}(\xi) = W_{\cU, \varepsilon}(\xi) \cup W_1 \cup W_2  \cup W_3.$$
This collection of neighbourhoods of $\xi$ in $\overline{EG}$ is denoted $\cO_{\overline{EG}}(\xi)$.
\end{definition}

Note that for $\xi$-families $\cU' \subset  \cU$ and $\varepsilon' < \varepsilon$, we do not necessarily have the inclusion $V_{\cU', \varepsilon'}(\xi) \subset V_{\cU, \varepsilon}(\xi)$ since these two neighbourhoods are defined by looking at the way geodesics leave two (a priori non related) different neighbourhoods of the domain $D(\xi)$. However, the crossing lemma immediately implies the following:
\begin{lem}
 Let $\xi \in \partial_{Stab} G$, $\cU, \cU'$ two $\xi$-families, and $0<\varepsilon' < \varepsilon$. If $\cU'$ is $d_\xi$-nested in $\cU$, then $V_{\cU', \varepsilon'}(\xi) \subset V_{\cU, \varepsilon}(\xi).$ \qed
\label{subneighbourhood}
\end{lem}

\begin{definition}
 We define a topology on $\overline{EG}$ by taking the topology generated by the elements of $\cO_{\overline{EG}}(x)$, for every $x \in \overline{EG}$. We denote by $\cO_{\overline{EG}}$ the set of elements of $\cO_{\overline{EG}}(x)$ when $x$ runs over $\overline{EG}$. Thus, any an open set in $\overline{EG}$ is a union of finite intersections of elements of $\cO_{\overline{EG}}$.
\end{definition}
We will show in the next section that $\cO_{\overline{EG}}$ is actually a \textit{basis} for the topology of $\overline{EG}$.

\section{The topology of $\overline{EG}$.}
In this section, we study the topology of $\overline{EG}$. More precisely, we prove that $\overline{EG}$ is a compact metrisable space. Recall that by the classical metrization theorem, it is enough to prove that $\overline{EG}$ is a second countable Hausdorff regular space (see below for definitions) which is sequentially compact.
\subsection{Filtration.}
Here we prove that the set of neighbourhoods we just defined is a basis for the topology of $\overline{EG}$. In order to do that, we need the following:\\

\noindent \textbf{Filtration Lemma}. Let $z,z' \in \overline{EG}$ and $U \in \cO_{\overline{EG}}(z)$ an open neighbourhood of $z$. If $z' \in U$, then there exists an open neighbourhood of $z'$, $U' \in \cO_{\overline{EG}}(z')$, such that $U' \subset U$.\\

Since points of $\overline{EG}$ are of three different natures ($EG$, $\partial X$, and $\partial_{Stab} G$), the proof breaks into six distinct cases.

\begin{prop} Let $\widetilde{x},\widetilde{y} \in EG$ and  $U \in \cO_{EG}(\widetilde{x})$ an open neighbourhood of $\widetilde{x}$ in $EG$. If $ \widetilde{y} \in U$, then there exists an open neighbourhood of $\widetilde{y}$ in $EG$, $U' \in \cO_{\overline{EG}}(\widetilde{y})$ such that $U' \subset U$.
\end{prop}
\begin{proof}
By definition of the topology, we can take $U'=U$.
\end{proof}
\begin{prop} Let $\eta, \eta' \in \partial X$ and $U \in \cO_{\overline{X}}(\eta)$ an open neighbourhood of $\eta$ in $X$. If $\eta' \in V_U(\eta)$, then there exists an open neighbourhood $U'$  of $\eta'$ in $\partial X$, such that $V_{U'}(\eta') \subset V_U(\eta)$.
\end{prop}
\begin{proof} Since $\cO_{\overline{X}}$ is a basis of neighbourhoods for $\overline{X}$, there exists a neighbourhood $U' \in \cO_{\overline{X}}(\eta')$ such that $ U' \subset U$. Now one clearly has $ \eta \in V_{U'}(\eta') \subset V_U(\eta)$.
\end{proof}
\begin{prop} Let $\widetilde{x} \in EG, \eta \in \partial X$ and $U$ an open neighbourhood of $\eta$ in $X$. If $\widetilde{x} \in V_U(\eta)$, then there exists an open neighbourhood $U'$ of $\widetilde{x}$ in $\overline{EG}$, $U' \in \cO_{EG}(\widetilde{x})$, such that $U' \subset V_U(\eta)$.
 \end{prop}
\begin{proof} It is enough to choose an arbitratry open neighbourhood $U'$ of $\widetilde{x}$ contained in $p^{-1}(U \cap X)$.
\end{proof}
\begin{prop} Let $\xi \in \partial_{Stab} G, \eta \in \partial X $ and $U \in \cO_{\overline{X}}(\eta)$ an open neighbourhood of $\eta$ in $\overline{X}$. If $\xi \in V_U(\eta)$, then there exist $\varepsilon \in (0,1)$ and a $\xi$-family $\cU$ such that $V_{\cU, \varepsilon}(\xi) \subset V_U(\eta)$.
\label{filtrationxieta}
 \end{prop}
\begin{proof} The subcomplex $D(\xi) \subset U$ is finite, hence compact, so choose $\varepsilon \in (0,1)$ such that $D^{\varepsilon}(\xi) \subset U$. Let $\cU$ be any $\xi$-family. For every $x \in \widetilde{\mbox{Cone}}_{\cU, \varepsilon}(\xi)$, the geodesic segment $[v_0,x]$ meets $D(\xi)$ by \ref{Uxi}. As $D(\xi)$ is contained in $U$, the same holds for $x$. It then follows that $V_{\cU_\xi, \varepsilon}(\xi) \subset V_U(\eta)$.
\end{proof}
\begin{prop} Let $\eta \in \partial X, \xi \in \partial_{Stab} G$, $\cU$ a $\xi$-family and $\varepsilon \in (0,1)$. If $\eta \in V_{\cU, \varepsilon}(\xi)$, then there exists an open neighbourhood $U$ of $\eta$ in $\overline{X}$ such that $V_U(\eta) \subset V_{\cU, \varepsilon}(\xi)$.
\label{filtrationetaxi} 
\end{prop}
\begin{proof} Let $\gamma_\eta: [0, \infty) \ra X$ be a parametrisation of the geodesic ray $[v_0, \eta)$. The subcomplex $D(\xi)$ being finite by \ref{finitedomain}, choose $R>0$ such that $D(\xi) \subset B(v_0,R)$, and let $x= \gamma_{\eta}(R+1)$. Since $\eta \in V_{\cU, \varepsilon}(\xi)$, we have $ x \in \mbox{Cone}_{\cU, \varepsilon}(\xi)$, which is open in $X$ by \ref{opencones}. Let $\delta>0$ such that $B(x, \delta) \subset$ Cone$_{\cU, \varepsilon}(\xi)$. Now if we set $U' = V_{R+1 , \delta}(\eta) \in \cO_{\overline{X}}(\eta)$, it follows that $V_{U'}(\eta) \subset V_{\cU, \varepsilon}(\xi)$. 
\end{proof}

\begin{prop} Let $\widetilde{x} \in EG, \xi \in \partial_{Stab} G$, $\cU$ a $\xi$-family and $\varepsilon \in (0,1)$. If $\widetilde{x} \in V_{\cU, \varepsilon}(\xi)$, then there exists a $U \in \cO_{\overline{EG}}(\widetilde{x})$ such that $U \subset V_{ \cU, \varepsilon}(\xi)$. 
\label{filtrationxxi}
\end{prop}
\begin{proof} It is enough to prove that $V_{\cU, \varepsilon}(\xi) \cap EG$ is open in $EG$. First, since the maps $\phi_{\sigma, \sigma'}$ are embeddings, it is clear that $W_{\cU, \varepsilon}(\xi)$ is open in $EG$. Let $\widetilde{y} \in V_{\cU, \varepsilon}(\xi) \cap EG$ with $y = p(\widetilde{y}) \notin D^\varepsilon(\xi)$. The star lemma \ref{star} yields a $\delta >0$ such that for every $z \in B(y, \delta) \setminus D^{\varepsilon}(\xi)$, the geodesic segment $[v_0,z]$ goes through $D^{\varepsilon}(\xi)$ and $\sigma_{\xi, \varepsilon}(z) \subset \mbox{st}(\sigma_{\xi, \varepsilon}(y))$. We can further assume that $B(y, \delta) \subset \mbox{st}(\sigma_y)$. It now follows immediately from the construction of $V_{\cU, \varepsilon}(\xi)$ that $p^{-1}\left( B(x, \delta) \right)$ is an open neighbourhood of $x$ contained in $V_{\cU, \varepsilon}(\xi)$, which concludes the proof.
\end{proof}

\begin{prop} Let $\xi, \xi' \in \partial_{Stab} G$, $\cU$ a $\xi$-family and $\varepsilon \in (0,1)$. If $ \xi' \in  V_{\cU, \varepsilon}(\xi)$, then there exists a $\xi'$-family $\cU'$ and $\varepsilon' \in (0,1)$ such that $V_{\cU', \varepsilon}(\xi') \subset V_{\cU, \varepsilon}(\xi)$.
\label{filtrationxixi}
 \end{prop}

By \ref{star}, let $\delta \in (0, \varepsilon)$ be such that for all $y \in D^{\delta}(\xi') \setminus D^{\varepsilon}(\xi)$, the geodesic segment $[v_0,y]$ goes through $D^{\varepsilon}(\xi)$ and is such that $\sigma_{\xi, \varepsilon}(y) \subset st\left(\sigma_{\xi, \varepsilon}(x)\right)$, for some $x \in D(\xi')$. We now define a $\xi'$-family using the following lemma.\\

\begin{lem}
There exist nested $\xi'$-family $\cU^{[d_\xi]} \supset \ldots \supset \cU^{[0]}= \cU'$ such that the following holds: Let $x$ be a point of $\mbox{Cone}_{\cU', \delta}(\xi')$ such that the geodesic from $v_0$ to $x$ leaves $D^\delta(\xi')$ at a point which is still inside $D^\varepsilon(\xi)$. Let $\sigma_1 = \sigma_{\xi', \delta}(x), \ldots, \sigma_n = \sigma_{\xi, \varepsilon}(x)$ ($n \leq d_\xi$) be the path of simplices met by the geodesic segment $[v_0,x]$ inside $\overline{D^\varepsilon}(\xi)$ after leaving $D^\delta(\xi')$ (cf Figure $2$). 

\begin{center}
\begingroup%
  \makeatletter%
  \providecommand\color[2][]{%
    \errmessage{(Inkscape) Color is used for the text in Inkscape, but the package 'color.sty' is not loaded}%
    \renewcommand\color[2][]{}%
  }%
  \providecommand\transparent[1]{%
    \errmessage{(Inkscape) Transparency is used (non-zero) for the text in Inkscape, but the package 'transparent.sty' is not loaded}%
    \renewcommand\transparent[1]{}%
  }%
  \providecommand\rotatebox[2]{#2}%
  \ifx\svgwidth\undefined%
    \setlength{\unitlength}{350.83585038bp}%
    \ifx\svgscale\undefined%
      \relax%
    \else%
      \setlength{\unitlength}{\unitlength * \real{\svgscale}}%
    \fi%
  \else%
    \setlength{\unitlength}{\svgwidth}%
  \fi%
  \global\let\svgwidth\undefined%
  \global\let\svgscale\undefined%
  \makeatother%
  \begin{picture}(1,0.55094899)%
    \put(0,0){\includegraphics[width=\unitlength]{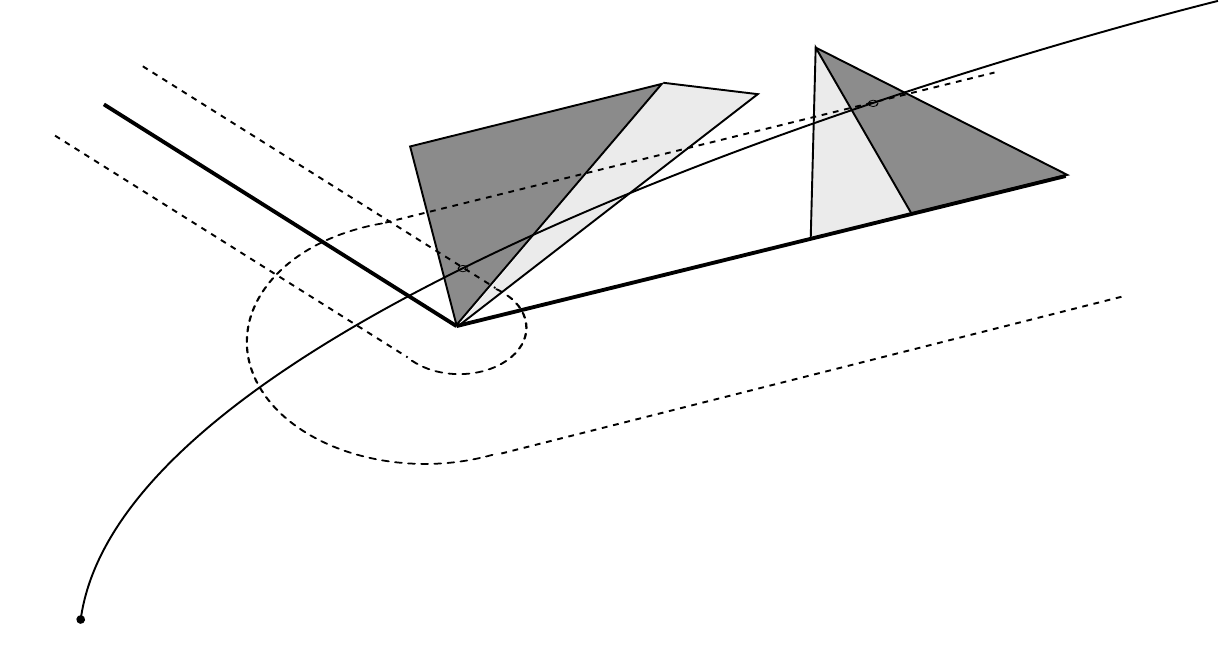}}%
    \put(0.6733901,0.22680163){\color[rgb]{0,0,0}\makebox(0,0)[lt]{\begin{minipage}{0.40548294\unitlength}\raggedright $D^\varepsilon(\xi)$\end{minipage}}}%
    \put(0.09523821,0.05405825){\color[rgb]{0,0,0}\makebox(0,0)[lt]{\begin{minipage}{0.33245759\unitlength}\raggedright $v_0$\end{minipage}}}%
    \put(-0.0026995,0.37770555){\color[rgb]{0,0,0}\makebox(0,0)[lt]{\begin{minipage}{0.63993282\unitlength}\raggedright $D^\delta(\xi')$\end{minipage}}}%
    \put(0.40333435,0.43822243){\color[rgb]{0,0,0}\makebox(0,0)[lt]{\begin{minipage}{0.06149503\unitlength}\raggedright $\sigma_{\xi', \delta}(x)$\end{minipage}}}%
    \put(0.72624228,0.45116829){\color[rgb]{0,0,0}\makebox(0,0)[lt]{\begin{minipage}{0.05765158\unitlength}\raggedright $\sigma_{\xi, \varepsilon}(x)$\end{minipage}}}%
  \end{picture}%
\endgroup%

Figure $2$.
\end{center}

We then have the following, for every $1 \leq k \leq n$: 
\begin{itemize}
 \item[$(i)$] The simplex $\sigma_k$ is contained in  $\bigcup_{v' \in V(\xi') \cap V(\xi')}\mbox{st}(v')$ but not in $\bigcup_{v \in V(\xi) \setminus V(\xi')} \mbox{st}(v)$.
 \item[$(ii)$] For every vertex $v'$ of $\sigma_k$ contained in $D(\xi')$, the inclusion $\overline{EG_{\sigma_k}} \subset U_{v'}^{[k]}$ holds in $\overline{EG_{v'}}$.
\end{itemize}
\end{lem}
\begin{proof}
If $v'$ is a vertex of $ D(\xi) \cap D(\xi')$, then for every vertex $ v$ of $\overline{st}(v') \cap \left(D(\xi) \setminus D(\xi')\right)$, choose a neighbourhood $W_{v,v'}$ of $\xi'$ in $\overline{EG_{v'}}$ missing $\overline{EG_{\left[v,v'\right]}}$, and set \\
$$W_{v'} = \left( \bigcap_{v \in \overline{st}(v') \cap \left(V(\xi) \setminus V(\xi')\right)} W_{v,v'} \right) \cap U_{v'}.$$

If $v'$ is a vertex not in $D(\xi)$, set $W_v = \overline{EG_{v'}}$.\\
We now define $\cU'$ to be a $\xi'$-family that is $d_\xi$-nested in the family of $W_{v'}, v' \in D(\xi')$, that is, $\cU'$ is a $\xi'$-family such that there exists a sequence of nested $\xi'$-families $\cU^{[d_\xi]} \supset \ldots \supset \cU^{[0]} = \cU'$ satisfying $W_{v'} \supset U_{v'}^{[d_\xi]} \supset \ldots \supset U_{v'}^{[0]} = U_{v'}'$ for every vertex $v'$ of $D(\xi')$.

We now prove $(i)$ and $(ii)$ by induction on $k$. Since the geodesic segment $[v_0,x]$ leaves $D^\delta(\xi')$ while inside $D^\varepsilon(\xi)$, we have $\sigma_1 = \sigma_{\xi', \delta}(x) \subset \bigcup_{v' \in V(\xi) \cap V(\xi')}\mbox{st}(v').$ To prove $(i)$ for $k=1$, we reason by contradiction. Suppose there exists a vertex $v'$ of $D(\xi) \cap D(\xi')$ and a vertex $v$ of $D(\xi) \setminus D(\xi')$ such that $\sigma_1 \subset \mbox{st}\big([v,v']\big)$, then we have $\overline{EG_{\sigma_1}} \subset \overline{EG_{[v,v']}}$ in $\overline{EG_{v'}}$. But the former set is contained in $U_{v'}$ since $\widetilde{x} \in V_{\cU', \delta}(\xi')$, and the latter is disjoint from $U_{v'}$ by construction of $\cU'$, which is absurd.

Suppose the result has been proved up to rank $k$. If $\sigma_{k+1} \subset \sigma_k$, the result is straightforward, so we suppose that $\sigma_k \subset \sigma_{k+1}$. We prove $(i)$  by contradiction. Suppose there exists a vertex $v'$ of $D(\xi) \cap D(\xi')$ and a vertex $v$ of $D(\xi) \setminus D(\xi')$ such that $\sigma_{k+1} \subset \mbox{st}\big([v,v']\big)$. Then by the induction hypothesis, we have $\overline{EG_{[v,v']}} \cap U_{v'}^{[k]} \neq \varnothing$ in $\overline{EG_{v'}}$, hence $\overline{EG_{[v,v']}} \subset U_{v'}^{[k+1]} \subset W_{v'}$ since $\cU^{[k]}$ is nested in $\cU^{[k+1]}$,and the last inclusion contradicts the definition of $\cU'$. 

We now prove $(ii)$. Let $v_k$ a vertex of $D(\xi)\cap D(\xi')$ contained in $\sigma_{k}$ (hence in $\sigma_{k+1}$). Thus we have $\overline{EG_{\sigma_{k+1}}}   \subset \overline{EG_{\sigma_k}}  \subset U_{v_k}^{[k]} \subset U_{v_k}^{[k+1]}$ in $\overline{EG_{v_k}}$. Now let $v'$ be another vertex of $D(\xi') \cap D(\xi)$ contained in $\sigma_{k+1}$ (if any). We thus have $\overline{EG_{[v_k,v']}} \cap  U_{v_k}^{[k]} \neq \varnothing$ in $\overline{EG_{v_k}}$, so $\overline{EG_{[v_k,v']}} \subset  U_{v_k}^{[k+1]} $ in  $\overline{EG_{v_k}}$. But by \ref{xifamilies}, this implies $\overline{EG_{[v_k,v']}} \subset  U_{v'}^{[k+1]} $, which proves $(ii)$.
\end{proof}

\begin{proof}[Proof of \ref{filtrationxixi}] Let us show now that $V_{\cU', \delta}(\xi') \subset V_{\cU, \varepsilon}(\xi).$ Let $z \in V_{\cU', \delta}(\xi')$ and $x \in \bar{p}(z)$. The geodesic $[v_0,x]$ meets $D^\delta(\xi')$, hence $D^\varepsilon(\xi)$. To prove that $z \in V_{\cU, \varepsilon}(\xi)$, it is now enough to prove that $x \in \widetilde{\mbox{Cone}}_{\cU, \varepsilon}(\xi)$.

If $x \in W_{\cU', \delta}(\xi') \cap D^\varepsilon(\xi)$, it follows from the definition of $\cU'$ that $z \in W_{\cU, \varepsilon}(\xi)$. 

If the geodesic segment $[v_0,x]$ meets $D^\delta(\xi')$ outside $D^\varepsilon(\xi)$, it follows from the definition of $\delta$ that there exists $x' \in D(\xi') \setminus D(\xi)$ such that $\sigma_{\xi, \varepsilon}(x) \subset st\left(\sigma_{\xi, \varepsilon}(x')\right)$. But since $x' \in \widetilde{\mbox{Cone}}_{\cU, \varepsilon}(\xi)$, the same holds for $x$.

Thus the only case left to consider is when the geodesic segment $[v_0,x]$ leaves $D^\delta(\xi')$ while still being inside $D^\varepsilon(\xi)$. But by the previous lemma, we get that for every vertex $v'$ of $\sigma_{\xi, \varepsilon}(x)$ contained in $D(\xi)$, $\overline{EG_{\sigma_{\xi, \varepsilon}(x)}}  \subset U_{v}^{[n]} \subset U_{v}$ in $\overline{EG_{v}}$, which now implies $x \in \widetilde{\mbox{Cone}}_{\cU, \varepsilon}(\xi)$. This concludes the proof.
\end{proof}

\begin{prop} 
$\cO_{\overline{EG}}$ is a basis for the topology of $\overline{EG}$, which makes it a second countable space. For this topology, $EG$ embeds as a dense open subset.
\label{basis}
\end{prop}
\begin{proof}
To prove that $\cO_{\overline{EG}}$ is a basis for the topology of $\overline{EG}$, it is enough to show that for every open sets $U_1, U_2 $ of $\overline{EG}$ and every $z \in U_1 \cap U_2$, there exists an open neighbourhood $W \in \cO_{\overline{EG}}$ such that $z \in W \subset U_1 \cap U_2$. 

If $z \in EG$: By the results from the previous paragraph, there exists $V_1,V_2 \in \cO_{EG}(z)$ such that $V_1 \subset U_1$ and $V_2 \subset U_2$. Then take $W$ to be any element of $\cO_{EG}(z) =  \cO_{\overline{EG}}(z)$ contained in $V_1 \cap V_2$.

If $z = \eta \in \partial X$: By the results from the previous paragraph, let $O_1, O_2 \in \cO_{\overline{X}}(\eta)$ such that $V_{O_1}(\eta) \subset U_1$ and $V_{O_2}(\eta) \subset U_2$. Choosing a neighbourhood $W \in \cV_{\overline{X}}(\eta)$ contained in $O_1 \cap O_2$, it follows that $V_W(\eta) \subset U_1 \cap U_2$.

If $z = \xi \in \partial_{Stab} G$: By the results from the previous paragraph, let $V_{\cU_1, \varepsilon_1}(\xi), V_{\cU_2, \varepsilon_2}(\xi)$ such that $V_{\cU_1, \varepsilon_1}(\xi) \subset U_1$ and $V_{\cU_2, \varepsilon_2}(\xi) \subset U_2$. Let $\cU$ be a $\xi$-family which is $d_\xi$-nested in $ \left\lbrace (U_1)_v \cap (U_2)_v, v \in V(\xi)\right\rbrace$, and let $\varepsilon = \mbox{min}(\varepsilon_1, \varepsilon_2)$. It follows from \ref{subneighbourhood} that $V_{\cU, \varepsilon}(\xi) \subset  V_{\cU_1, \varepsilon_1}(\xi) \cap V_{\cU_2, \varepsilon_2}(\xi) \subset U_1 \cap U_2$.\\

To prove that this topology is second countable, we define countable many open sets $(U_n)_{n \geq 0}$ such that for every open set $U$ in $\cO_{\overline{EG}}$ and every $x$ in $U$, there exist an integer $m$ such that $x \in U_m \subset U$.

Since $EG$ is the realisation of a complex of spaces over a simplicial complex with countably many simplices, and with fibres that have a CW-structure with countably many cells, it inheritates a CW-complex structure with countably many cells. Thus its topology is second countable, and we can choose a countable basis of neighbourhoods $(U_n),n \geq 0$ of $EG$.

Since $X$ is a simplicial complex with countably many cells, it is a separable space, hence so is the set $\Lambda$ of points lying on a geodesic from $v_0$ to a point of $\partial X$. Let us denote $Y$ a dense countable subset of $\Lambda$. Now the family of open sets on $V_{r, \varepsilon}(\eta)$ for $\eta \in \partial X$, $\gamma_\eta(r) \in D$ and $\varepsilon \in \bbQ$ is a countable family, yielding a countable family of open neighbourhoods of $\overline{EG}$, denoted $(V_n)_{n \geq 0}$. Note that $(V_n)_{n \geq 0}$ contains a basis of neighbourhoods for every point of $\overline{EG}$ belonging to $\partial X$.

A neighbourhood of a point $\xi$ of $\partial_{Stab} G$ can be thought as defined by choosing a constant $\varepsilon \in (0,1)$, a finite subcomplex of $X$ (the domain of $\xi$), and for every vertex $v$ of that subcomplex an open set of $\overline{EG_v}$. Since domains of points of $\partial_{Stab} G$ are finite by \ref{finite}, there are only countably many such subcomplexes. Furthermore, for every vertex $v$ of $X$, $\overline{EG_v}$ has a countable basis of neighbourhoods. It is now clear that we can define a countable family $(W_n)_{n \geq 0}$ of open neighbourhoods, containing a basis of neighbourhoods of every element of $\partial_{Stab} G$. 

The family consisting of all the $U_n, V_n, W_n$ is now a countable basis of neighbourhoods of $\overline{EG}$.\\

Finally, the subset $EG$, which is open by construction of the topology, is dense in $\overline{EG}$ since every open set in that basis of neighbourhoods meets $EG$ by construction.
\end{proof}

\begin{lem} The topology of $\overline{EG}$ does not depend on the choice of a base point. Moreover, the action of $G$ on $EG$ continuously extends to $\partial G$.
\label{basepoint}
\end{lem}

\begin{proof} Choose $x_0$ and $x_1$ two points of $X$ (note that we do not assume these points to be vertices). Throughout this proof, we will indicate the dependance on the base point by indicating it as an exponent. It is a well know fact that the topology of $\overline{X}$ does not depend on the base point, so it is enough to consider neighbourhoods of points in $\partial_{Stab} G$. 

The number of simplices in a domain $D(\xi)$, $\xi \in \partial_{Stab} G$ is uniformly bounded above by \ref{finitedomain}. Thus, the finiteness lemma \ref{finite} yields an integer $k$ such that for every $x$ of $X$ and every $\xi \in \partial_{Stab} G$, $\mbox{Geod}(D(\xi), x)$ meets at most $k$ simplices inside the open star of $D(\xi)$. Let $\xi \in \partial_{Stab} G$, $\cU_0$ a $\xi$-family for the topology centered at $x_0$ and $\varepsilon > 0$. Now let $\cU_1$ be a $\xi$-family for the topology centered at $x_1$, which is $k$-nested in $\cU_0$. Let $x$ be an element of $\widetilde{\mbox{Cone}}_{\cU_1, \varepsilon}^{v_1}(\xi)$. Then the geodesic reattachment lemma \ref{reattachment} implies that $[x_0, x]$ meets $D(\xi)$. Since $D(\xi)$ is connected, there exists a paths of simplices from $\sigma_{\xi, \varepsilon}^{v_0}(x)$ to $\sigma_{\xi, \varepsilon}^{v_1}(x)$ of length at most $k$ in $ \big( \mbox{Geod}(x, D(\xi)) \cap \mbox{st}(D(\xi)) \big) \setminus D(\xi)$. Since $\cU_1$ is $k$-nested in $\cU_0$, it now follows from the crossing lemma that $x \in \widetilde{\mbox{Cone}}_{\cU_0, \varepsilon}^{v_0}(\xi)$, hence $\widetilde{\mbox{Cone}}_{\cU_1, \varepsilon}^{v_1}(\xi) \subset \widetilde{\mbox{Cone}}_{\cU_0, \varepsilon}^{v_0}(\xi)$. Moreover, since $\cU_1$ is contained in $\cU_0$, we get $V_{\cU_1, \varepsilon}^{v_1}(\xi) \subset V_{\cU_0, \varepsilon}^{v_0}(\xi)$.\\

We extend the $G$-action on $EG$ to $\partial G$ as follows. First note that the action naturally extends to $\partial X$. Indeed, $G$ acts on the CAT(0) space $X$ by isometries, and those isometries naturally extend to homeomorphisms of the visual boundary $\partial X$. Furthermore, we defined in Section 2 a $G$-action on $\partial_{Stab} G$. Thus we have an action of $G$ on $\overline{EG}$, which we now prove to be continuous.

Let $g \in G$. Since $EG$ is open in $\overline{EG}$ and the action of $G$ on $EG$ is continous, it is enough to check the continuity at points of $\partial G$.

Let $\eta \in \partial X$, and let $U$ be a neighbourhood of $g\eta$ in $\overline{X}$. The action of $G$ on $\overline{X}$ being continuous, let $U'$ be a neighbourhood of $\eta$ in $\overline{X}$ such that $g.U' \subset U$. It then follows from the definition of our basis of neighbourhoods that $g. V_{U'}(\eta) \subset V_U(g.\eta)$.

Let $\xi \in \partial_{Stab} G$, $\cU$ a $g\xi$-family and $\varepsilon \in (0,1)$. For every vertex $v$ of $D(\xi)$, the action of $G_v$ on $EG_v$ continuously extends to $\overline{EG_v}$ by the $\cZ$-boundary assumption, so $g^{-1}\cU$ is a collection of open sets for $\xi$. Since the topology does not depend on the basepoint, let $\cU'$ be a $\xi$-family contained in $\xi$ and such that $\widetilde{Cone}_{\cU', \varepsilon}(\xi) \subset g^{-1}. \widetilde{Cone}_{\cU, \varepsilon}(\xi)$. We thus have $g.\widetilde{Cone}_{\cU', \varepsilon}(\xi) \subset  \widetilde{Cone}_{\cU, \varepsilon}(\xi)$, and also 
 $g.W_{\cU', \varepsilon}(\xi) \subset  W_{\cU, \varepsilon}(g.\xi)$, hence $g.V_{\cU', \varepsilon}(\xi) \subset V_{\cU, \varepsilon}(\xi)$.
\end{proof}

\subsection{Induced topologies.}

\begin{prop}
 The topology of $\overline{EG}$ induces the natural topologies on $EG$, $\partial X$ and $\overline{EG_v}$ for every vertex $v$ of $X$.
\label{inducedtopology}
\end{prop}
We first prove that for any open set $U$ in the basis of neighbourhoods $\cO_{\overline{EG}}$ previously defined, $U \cap EG$ is open in $EG$. For $x \in EG$, the result is obvious for elements in $\cO_{\overline{EG}}(x)$ since open sets in $\cO_{\overline{EG}}(x)$ are open sets of $EG$ by definition. For $\eta \in \partial X$ and $U$ a neighbourhood of $ \eta$ in $\overline{X}$, we have $V_U(\eta) \cap EG = p^{-1}(U \cap X)$ which is open in $EG$. For $\xi \in \partial_{Stab} G$, $\varepsilon \in (0,1)$ and $\cU$ a $\xi$-family, it was proven in \ref{filtrationxxi} that $V_{\cU, \varepsilon}(\xi) \cap EG$ is open in $EG$.

We now prove that for any open set $U$ in the basis of neighbourhoods $\cO_{\overline{EG}}$, $U \cap \partial X$ is open in $\partial X$. For an element $\eta \in \partial X$ and $U$ a neighbourhood of $\eta$ in $\overline{X}$, we have $ V_U(\eta) \cap \partial X = U \cap \partial X$, which is open in $\partial X$. Now consider $\xi \in \partial_{Stab} G$ , $\varepsilon \in (0,1)$ and $\cU$ a $\xi$-family. If $V_{\cU ,\varepsilon}(\xi) \cap \partial X$ is empty there is nothing to prove, otherwise let $\eta \in V_{\cU, \varepsilon}(\xi) \cap \partial X$. By \ref{filtrationetaxi}, let $U'$ be a neighbourhood of $\eta$ in $\overline{X}$ such that $\cV_{U'}(\eta) \subset V_{\cU, \varepsilon}(\xi)$. Thus, $\eta \in U' \cap \partial X \subset V_{\cU, \varepsilon}(\xi) \cap \partial X$, and $V_{\cU, \varepsilon}(\xi) \cap \partial X$ is open in $\partial X$.

Before proving the analogous result for $\overline{EG_v}$, with $v$ a vertex of $X$, we need the following lemma.

\begin{lem}
 Let $\xi \in \partial_{Stab} G$, $\cU$ a $\xi$-family and $\varepsilon \in (0,1)$. Then there exists a $\xi$-family $\cU' \subset \cU$ such that $ \bigcup_{v \in V(\xi)} U_v' \cap  \partial G_v \subset V_{\cU, \varepsilon}(\xi)$.
\label{lemmesansnom}
\end{lem}
\begin{proof}
 By \ref{finitedomain}, let $m$ be an integer such that domains of points of $\partial_{Stab} G$ meet at most $m$ simplices, and let $\cU'$ be a $\xi$-family which is $m$-refined in $\cU$. Let $\xi' \in \bigcup_{v \in V(\xi)} U_v^{'} \cap \partial G_v$ and $x \in D(\xi')$. Since $D(\xi')$ is convex by \ref{finitedomain}, let $\gamma$ be a geodesic path in $D(\xi')$ from $x$ to $D(\xi)$  and meeting $D(\xi)$ at a single point. This yields a path of open simplices from $\sigma_1 \subset \mbox{st}(D(\xi)) \setminus D(\xi)$ to $\sigma_x $ of length at most $m$ in $X \setminus D(\xi)$ with $n \leq d_\xi$. Now since $\xi' \in \partial G_{\sigma_1}$, we have $\sigma_1 \subset \mbox{st}_{\cU'}(\xi) $. By definition of $\cU^{'}$, we get  $\sigma_n  \subset \widetilde{\mbox{Cone}}_{\cU, \varepsilon}(\xi) $, hence the result.
\end{proof}

\begin{proof}[proof of \ref{inducedtopology}]
Let $v$ be a vertex of $X$. We now prove that for every open set $U$ in the basis of neighbourhood $\cO_{\overline{EG}}$, $U \cap \overline{EG_v}$ is open in $\overline{EG_v}$. 

We proved already that the topology of $\overline{EG}$ induces the natural topology on $EG$. Now using the filtration lemmas \ref{filtrationxieta} and \ref{filtrationxixi}, it is enough to show, for every $\xi \in \partial G_v$, every $\varepsilon \in (0,1)$ and every $\xi$-family $\cU$, that $V_{\cU, \varepsilon}(\xi) \cap \overline{EG_{v}}$ contains a neighbourhood of $\xi$ in $\overline{EG_v}$. By \ref{lemmesansnom}, let $\cU'$ be a $\xi$-family contained in $\cU$ and such that every element of $U_v' \cap \partial G_v$ belongs to $V_{\cU, \varepsilon}(\xi)$. Then we have $\xi \in U_v' \subset V_{\cU, \varepsilon}(\xi) \cap \overline{EG_v}$, and so $V_{\cU, \varepsilon}(\xi) \cap \overline{EG_v}$ is open in $\overline{EG_v}$. Thus the topology of $\overline{EG}$ induces the natural topology on $\overline{EG_v}$.

Finally, note that the map $\overline{EG_v} \ra \overline{EG}$ is injective by \ref{xiloop}. As $\overline{EG_v}$ is a compact space, that map is an embedding.
\end{proof}

In the exact same way, we can prove the following:

\begin{lem}
 Let $\sigma$ be a closed cell of $X$. Then the quotient map $\sigma \times \overline{EG_\sigma} \ra \overline{EG}$ is continuous. \qed
\label{topologystab2}
\end{lem}

\subsection{Weak separation}
In this paragraph, we prove the following:

\begin{prop}
 The space $\overline{EG}$ satisfies the $T_0$ condition, that is, for every pair of distinct points, there is an open set of $\overline{EG}$ containing one but not the other.
\label{weakseparation}
\end{prop}
 Note that this property does not imply that the space is Hausdorff. However, we will prove in the next subsection that $\overline{EG}$ is also \textit{regular}, and it is a common result of point-set topology that a space that is $T_0$ and regular is also Hausdorff. As usual, the proof of \ref{weakseparation} splits in many cases.

\begin{prop}
 Let $\widetilde{x}, \widetilde{y}$ be two distinct elements of $ EG \subset \overline{EG}$. Then $\widetilde{x}$ and $\widetilde{y}$ admit disjoint neighbourhoods.
\end{prop}
\begin{proof}
 Open sets in $EG$ are open in $\overline{EG}$ by definition. The result thus follows from the fact that $EG$ is a Hausdorff space.
\end{proof}
\begin{prop}
 Let $\eta, \eta'$ be two distinct elements of $\partial X \subset \overline{EG}$. Then $\eta$ and $\eta'$ admit disjoint neighbourhoods.
\end{prop}
\begin{proof}
 The space $\overline{X}$ is metrisable, hence Hausdorff. Choosing disjoint neighbourhoods $U$ of $\eta$ in $\overline{X}$ (resp. $U'$ of $\eta'$ in $\overline{X}$ ) yield disjoint neighbourhoods $V_U(\eta), V_{U'}(\eta')$.
\end{proof}
\begin{prop}
 Let $\widetilde{x} \in EG$ and $\eta \in \partial X$. Then $\widetilde{x}$ and $\eta$ admit disjoint neighbourhoods.
\end{prop}
\begin{proof}
 Let $x = p(\widetilde{x}) \in X$. Since $\overline{X}$ is a Hausdorff space, let $U$ be a neighbourhood of $x$ in $\overline{X}$ and $U'$ be a neighbourhood of $\eta'$ in $\overline{X}$ that are disjoint. Then $p^{-1}(U) $ is a neighbourhood of $\widetilde{x}$ in $\overline{EG}$ and $V_{U'}(\eta)$ is a neighbourhood of $\eta$ in $\overline{EG}$ that is disjoint from $p^{-1}(U)$.
\end{proof}
\begin{prop}
 Let $\xi \in \partial_{Stab} G$ and $\eta \in \partial X$. Then there exists a neighbourhood of $\eta$ in $\overline{EG}$ that does not contain $\xi$.
\end{prop}
\begin{proof}
 Since $D(\xi)$ is bounded, let $R>0$ such that the $D(\xi)$ is contained in the $R$-ball centered at $v_0$. Now take a neighbourhood $U$ of $\eta$ in $\overline{X}$ that does not meet that $R$-ball. The subset $V_U(\eta)$ is a neighbourhood of $\eta$ in $\overline{EG}$ to which $\xi$ does not belong.
\end{proof}

\begin{prop}
 Let $\widetilde{x} \in EG$ and $\xi \in \partial_{Stab} G$. Then there exists a neighbourhood of $\widetilde{x}$ in $\overline{EG}$ that does not contain $\xi$.
\end{prop}
\begin{proof}
 Choose any neighbourhood of $\widetilde{x}$ in $EG$. This is by definition a neighbourhood of $\widetilde{x}$ in $\overline{EG}$, to which $\xi$ does not belong. 
\end{proof}

\begin{prop}
 Let $\xi, \xi'$ be two different elements of $\partial_{Stab} G$. Then there exists a neighbourhood of $\xi$ in $\overline{EG}$ that does not contain $\xi'$.
\end{prop}
\begin{proof}
If $D(\xi) \cap D(\xi') \neq \varnothing$, let $v$ be a vertex in that intersection and let $U_v$ be a neighbourhood of $\xi$ in $\overline{EG_v}$ that does not contain $\xi'$. Now we can take a $\xi$-family $\cU'$ small enough so that $U_v' \subset U_v$ and thus $\xi' \notin V_{\cU', \frac{1}{2}}(\xi)$.

If $D(\xi) \cap D(\xi') = \varnothing$, let $x \in D(\xi')$. There are two cases to consider:
\begin{itemize}
 \item If $[v_0,x]$ does not meet $D(\xi)$, then $V_{\cU_\xi, \frac{1}{2}}(\xi)$ does not contain $\xi'$ by \ref{Uxi}. 
 \item Otherwise, $[v_0,x]$ meets $D(\xi)$ and leaves it. Let $\sigma$ be the first simplex touched by $[v_0,x]$ after leaving $D(\xi)$, $v$ a vertex of $\sigma \cap D(\xi)$ and $U_v$ a neighbourhood of $\xi$ in $\overline{EG_v}$ that does not contain $\overline{EG_\sigma}$. Now let $\cU'$ be $\xi$-family such that $U_v' \subset U_v$ and $\cU''$ a $\xi$-family that is $d_\xi$-nested in $\cU'$. It then follows from the crossing lemma that $\xi' \notin V_{\cU'', \frac{1}{2}}(\xi)$. \qedhere
\end{itemize}
\end{proof}

\subsection{Regularity}

In this paragraph, we prove the following:

\begin{prop} 
 The space $\overline{EG}$ is regular, that is, for every open set $U$ in $\overline{EG}$ and every element $x \in U$, there exists another open set $V$ containing $x$ and contained in $U$, and such that every element not in $U$ admits a neighbourhood that does not meet $V$.
\label{regular}
\end{prop}

Since we previously defined a basis of neighbourhoods for $\overline{EG}$, it is enough to prove such a proposition for open sets $U$ in that basis. As usual, the proof of \ref{regular} splits in many cases, depending on the nature of the open sets $U$ and elements of $U$ involved.
\begin{prop}
 Let $\widetilde{x} \in EG$ and $U$ an open neighbourhood of $\widetilde{x}$ in $\overline{EG}$. Then there exists a subneighbourhood $U'$ of $\overline{EG}$ containing $\widetilde{x}$ and such that every element not in $U$ admits a neighbourhood that does not meet $U'$.
\end{prop}
\begin{proof}
 The space $EG$ being a CW-complex, its topology is regular, so we can choose a neighbourhood $U'$ of $\widetilde{x}$ in $EG$ whose closure (in $EG$) is contained in $U$. Let us call $V$ that closure, and let $x = p(\widetilde{x})$. Since $EG$ is locally finite, we can further assume that $p(V)$ meets only finitely many simplices and that it is contained in $\mbox{st}(\sigma_x)$. We now show that $V$ is closed in $\overline{EG}$, which implies the proposition.

A point of $EG \setminus V$ clearly admits a neighbourhood in $\overline{EG}$ that does not meet $V$, since open subsets of $EG$ are open in $\overline{EG}$. For a point $\eta \in \partial X$, choosing any neighbourhood of $\eta$ in $\overline{X}$ that does not meet $p(V)$ yields a neighbourhood of $\eta$ in $\overline{EG}$ not meeting $V$. Thus the only case left is that of an element $\xi \in \partial_{Stab} G$. There are two cases to consider:

If $x \in D(\xi)$, then since $p(V)$ meets only finitely many simplices, it is easy to find a $\xi$-family $\cU$ such that $W_{\cU, \frac{1}{2}}(\xi)$ misses $V$, which implies that the whole $V_{\cU, \frac{1}{2}}(\xi)$ misses $V$.

If $x \notin D(\xi)$, then lemma \ref{finite} ensures the existence  of a finite subcomplex $K \subset X$ containing $\mbox{Geod}( v_0, p(V)) \cap \mbox{st}(D(\xi))$. We define a $\xi$-family $\cU$ and a constant $\varepsilon$ as follows. Let $v$ be a vertex of $D(\xi)$. For every $\sigma \subset \left(  \mbox{st}(v) \cap K \right)  \setminus D(\xi)$, let $U_{v, \sigma}$ be a neighbourhood of $\xi$ in $\overline{EG_v}$ which is disjoint from $ \overline{EG_{\sigma}}$. We now set\\
$$ U_v = \underset{\sigma \subset \left(  \mathrm{st}(v) \cap K \right)  \setminus D(\xi)}{\bigcap} U_{v, \sigma}. $$
Let $\cU$ be a $\xi$-family which is contained in $ \left\lbrace U_v, v \in V(\xi) \right\rbrace $, and choose 
$$\varepsilon = \mbox{min} \big(\frac{1}{3} \mbox{dist}(p(V), D(\xi)) ,1 \big),$$
which is positive since $p(V) \subset \mbox{st}(\sigma_x)$. \\
We now show by contradiction that $ V_{\cU, \varepsilon}(\xi) \cap V = \varnothing$. Suppose there exists an element $\widetilde{y}$ in that intersection and let $y = p(\widetilde{y})$. By \ref{Uxi}, $[v_0,y]$ goes through $D(\xi)$. But since $\widetilde{y} \in V$, we have $\sigma_{\xi, \varepsilon}(y) \subset K$, which contradicts the construction of $\cU$.

Thus every element of $\overline{EG} \setminus V$ admits a neighbourhood missing $V$, so $V$ is closed in $\overline{EG}$.
\end{proof}

\begin{prop}
Let $\eta \in \partial X$ and $U$ be an open neighbourhood of $\eta$ in $\overline{X}$. Then there exists an open neighbourhood $U'$ of $\eta$ in $\overline{X}$ such that every element not in $V_U(\eta)$ admits a neighbourhood that does not meet $V_{U'}(\eta)$.
\end{prop}
\begin{proof}
 By \ref{CAT(0)nesting}, we first choose a neighbourhood $W$ of $\eta$ in $\overline{X}$ contained in $U$ and such that $d(W \cap X, X \setminus U) > A+1$, where $A$ is the acylindricity constant. Since $\overline{X}$ is metrisable, hence regular, we can further assume that $\overline{W} \subset U$. Finally, we can choose $R>0$ and $\delta>0$ such that $U'=V_{R,\delta}(\eta)$ is contained in $W$ and $B(\gamma_\eta(R), \delta)$ is contained in the open star of $\gamma_\eta(R)$ (recall that $\gamma_\eta$ is a parametrization of the geodesic ray $[v_0, \eta)$). We now show that every element not in $V_U(\eta)$ admits a neighbourhood that does not meet $V_{U'}(\eta)$.\\

Let $z \in EG \setminus V_{U}(\eta)$. Then $p(z)$ is not in $U$, hence not in  $\overline{U'}$. Since $U'$ is open in $\overline{X}$, there exist an open set $U''$ of $\overline{X}$ containing $p(z)$ and such that $U'' \subset X \setminus \overline{U'}$. Then $p^{-1}(U'')$ is open in $\overline{EG}$ and $p^{-1}(U'')$ does not meet $V_U(\eta)$.

Let $\eta' \in \partial X \setminus V_{U}(\eta)$. Then $\eta' \notin U \cap \partial X$ hence $\eta' \notin \overline{U'} $. Since $\overline{U'} $ is closed in $\overline{X}$, we choose an open set $U''$ in $\cO_{\overline{X}}(\eta)$ missing $U'$. It is now clear that $V_{U''}(\eta')$ does not meet $V_\cU'( \eta)$.

Let $\xi \in (\partial_{Stab} G) \setminus V_U(\eta)$. To find a neighbourhood of $\xi$ that does not meet $V_{U'}(\eta)$, is enough to find a $\xi$-family $\cU'$ such that $U' \cap \widetilde{\mbox{Cone}}_{\cU', \frac{1}{2}}(\xi) = \varnothing$. We define such a $\xi$-family as follows:\\
 Let $x = \gamma_\eta(R)$. By \ref{finite}, let $K$ be the finite subcomplex of $X$ spanned by simplices meeting $\mbox{Geod}(D(\xi), x)$. Let $v $ be a vertex of $D(\xi)$. For every simplex $\sigma$ contained in $\left(  \mbox{st}(v) \cap K \right)  \setminus D(\xi)$, let $U_{v, \sigma}$ be an open neighbourhood of $\xi$ in $\overline{EG_v}$ disjoint from $\overline{EG_\sigma}$. We then set\\
$$ V_v = \bigcap_{\sigma \subset \left(  \mathrm{st}(v) \cap K \right)  \setminus D(\xi)} U_{v, \sigma}. $$
Now take $\cU$ to be a $\xi$-family contained in $\left\lbrace V_v, v \in V(\xi) \right\rbrace$, and let $\cU'$ be a $\xi$-family that is $2$-refined in $\cU$.

We now show by contradiction that $U' \cap \widetilde{\mbox{Cone}}_{\cU', \frac{1}{2}}(\xi) \neq \varnothing$ . Let $y$ be an element in this intersection. Then $[v_0,y]$ meets $D(\xi)$ (by \ref{Uxi}) and $B(x, \delta) \cap S(v_0,R)$ (by construction of $U'$). 

Since $d(U', X \setminus U) \geq A+1$ and $D(\xi)$ meets $U$, it follows that $\mbox{st}(D(\xi)) \cap U' = \varnothing$. Hence the geodesic segment $[v_0,y]$ enters $D(\xi)$ before meeting $B(x, \delta) \cap S(v_0,R)$.
Let $y'$ be the point of $[v_0,y]$ inside $B(x, \delta) \cap S(v_0,R)$. By construction of $R$ and $\delta$, it follows that $\overline{\sigma_{y'}}$ is a face of $\overline{\sigma_x}$. Now since $x \in \mbox{Cone}_{\cU', \frac{1}{2}}(\xi)$, the refinement lemma \ref{refinement} implies that $\sigma_{y'} \subset \mbox{Cone}_{\cU, \frac{1}{2}}(\xi)$, which contradicts the definition of $\cU$.
\end{proof}

\begin{prop}
Let $\xi \in \partial_{Stab} G$, $\varepsilon \in (0,1)$ and $\cU$ a $\xi$-family. Then there exists a $\xi$-family $\cU'$ such that every element not in $V_{\cU, \varepsilon}(\xi)$ admits a neighbourhood that misses $V_{\cU', \varepsilon}(\xi) $.
 \end{prop}
\begin{proof}
 Let $m$ be an integer such that domains of points of $\partial_{Stab} G$ meet at most $m$ simplices. Choose a $\xi$-family $\cU'$ which is $m$-refined and nested in $\cU$. We now show that every element not in $V_{\cU, \varepsilon}(\xi)$ admits a neighbourhood that misses $V_{\cU', \varepsilon}(\xi) $.

Let $\widetilde{x} \in EG \setminus V_{\cU, \varepsilon}(\xi)$, and $x = p(\widetilde{x})$.
\begin{itemize}
 \item If $x \in \overline{D^{\varepsilon}(\xi)}$, then $\widetilde{x} \notin \left\{x\right\} \times U_{\sigma_x}$, hence $\widetilde{x} \notin \left\{x\right\} \times \overline{U_{\sigma_x}'}$. Let $W_x$ be a neighbourhood of $\widetilde{x}$ in $\left\{x\right\} \times \overline{EG_{\sigma_x}}$ that does not meet $\overline{U_{\sigma_x}}$, and $V$ be an open neighbourhood of $x$ in $X$ contained in $\mbox{st}(\sigma_x)$. $V \times W_x$ does not meet $V_{\cU', \varepsilon}(\xi)$.
 \item If $x \notin \overline{D^{\varepsilon}(\xi)}$, let $V$ be an open neighbourhood of $x$ in $X$ contained in $\mbox{st}(\sigma_x)$. Then $p^{-1}(V)$ does not meet $V_{\cU', \varepsilon}(\xi)$ as $\cU'$ is refined in $\cU$.
In particular, the closure of  $\mbox{Cone}_{\cU', \varepsilon}(\xi)$ is contained in $\widetilde{\mbox{Cone}}_{\cU, \varepsilon}(\xi)$, so $x \notin \overline{\mbox{Cone}_{\cU', \varepsilon}(\xi)}$. Now choose a neighbourhood $V$ of $x$ in $X$ contained in \\$X \setminus \bigg( \overline{ \mbox{Cone}_{\cU', \varepsilon}(\xi)} \cup \overline{D^{\varepsilon}(\xi)} \bigg)$, and it follows that $p^{-1}(V)$ is a neighbourhood of $\widetilde{x}$ in $\overline{EG}$ that does not meet $V_{\cU', \varepsilon}(\xi)$.
\end{itemize}
Let $\eta \in  \partial X \setminus V_{\cU, \varepsilon}(\xi).$ We construct a neighbourhood $V$ of $\eta$ in $\overline{X}$ that does not meet $\widetilde{\mbox{Cone}}_{\cU', \varepsilon}(\xi)$. First, since $D(\xi)$ is bounded, let $R > 0$ such that $D(\xi)$ is contained in the $R$-ball centered at $v_0$, and let $x = \gamma_\eta(R+1)$. 
\begin{itemize}
 \item If $[v_0,\eta)$ does not meet $D(\xi)$, let $\delta = \frac{1}{2}\mbox{dist}\bigg( \gamma_\eta\big([0,R+1] \big), D(\xi)\bigg) >0$, and let $V$ be a neighbourhood of $\eta$ in $\overline{X}$ that is contained in $V_{R+1, \delta}(\eta)$. For every $y$ in $V$, $[v_0,y]$ does not meet $D(\xi)$, hence $V \cap \widetilde{\mbox{Cone}}_{\cU', \varepsilon}(\xi) = \varnothing$.
 \item If $[v_0,\gamma)$ goes through $D(\xi)$, then since $x$ does not belong to $ \widetilde{\mbox{Cone}}_{\cU', \varepsilon}(\xi)$, let  $v$ be a vertex of $D(\xi)$ in  $\sigma_{\xi, \varepsilon}(x)$ such that $\overline{EG_{\sigma_{\xi, \varepsilon}}(x)} \nsubseteq U_v$ in $\overline{EG_v}$. Lemma \ref{star} yields a constant $\delta>0$ such that for every $y \in B(x, \delta)$, $[v_0,y]$ goes through $D(\xi)$ and $\sigma_{\xi, \varepsilon}(y) \subset \mbox{st}\big(\sigma_{\xi, \varepsilon}(x)\big)$ inside $D^{\varepsilon}(\xi)$. As $\cU'$ is nested in $\cU$, it then follows that $V_{R+1, \delta}(\eta)$ does not meet $V_{\cU', \varepsilon}(\xi)$. Then $[v_0,y]$ goes through $B(x, \delta)$, so it also goes through $\mbox{st}(\sigma_{\xi, \varepsilon}(x))$. As $\cU'$ is nested in $\cU$ and $\overline{EG_\sigma} \nsubseteq U_v$ in $\overline{EG_v}$, it follows that $\overline{EG_{\sigma_{\xi, \varepsilon}(y)}} \nsubseteq U_{v}'$, hence $y \notin \widetilde{\mbox{Cone}}_{\cU', \varepsilon'}(\xi)$ and $V \cap \widetilde{\mbox{Cone}}_{\cU', \varepsilon'}(\xi) = \varnothing$.
\end{itemize} 
Let $\xi' \in (\partial_{Stab} G) \setminus V_{\cU, \varepsilon}(\xi)$. To find a neighbourhood of $\xi'$ that misses $V_{\cU, \varepsilon}(\xi)$, it is enough to find a $\xi'$-family $\cU''$ such that $\widetilde{\mbox{Cone}}_{\cU^{''}, \varepsilon}(\xi') \cap \widetilde{\mbox{Cone}}_{\cU', \varepsilon}(\xi) = \varnothing$. We define such a $\xi'$-family as follows. By \ref{finite}, let $K$ be a finite subcomplex containing $\mbox{Geod}(v_0, D(\xi) ) \cap \mbox{st}(D(\xi'))$. Let $v $ be a vertex of $D(\xi')$. For every $\sigma \subset \left(  \mbox{st}(v) \cap K \right)  \setminus D(\xi')$, let $U_{v, \sigma}^{''}$ be a neighbourhood of $\xi'$ in $\overline{EG_v}$ which is disjoint from $\overline{EG_{\sigma}}$, and set \\
$$ U_{v}^{''} =  \underset{\sigma \subset \left(  \mbox{st}(v) \cap K \right)  \setminus D(\xi')}{\bigcap} U_{v, \sigma}^{''}. $$
If $v$ is also in $D(\xi)$, we can further assume by the convergence property that the only $EG_\sigma$ inside $EG_v$ meeting both $U_v$ and $U_v''$ contain $\xi$ and $\xi'$. Now let $\cU^{''}$  be a $\xi'$-family which is $m$-refined in $ \left\lbrace U_v^{''}, v \in V(\xi') \right\rbrace $, and let us prove by contradiction that $\widetilde{\mbox{Cone}}_{\cU^{''}, \varepsilon^{'}}(\xi') \cap \widetilde{\mbox{Cone}}_{\cU', \varepsilon}(\xi) = \varnothing$. Let $x$ be in such an intersection. Then, by \ref{Uxi}, the geodesic $[v_0,x]$ meets both $D(\xi)$ and $D(\xi')$.

Note that $[v_0,x]$ cannot meet $D(\xi')$ before $D(\xi)$ or leave both $D(\xi)$ and $D(\xi')$ at the same time, by construction of the $U_v''$. Thus $[v_0,x]$ meets $D(\xi)$ before $D(\xi')$. Let $x'$ be the first point of $D(\xi')$ met by $[v_0,x]$. If $D(\xi) \cap D(\xi') = \varnothing$, it follows from the fact that $\cU'$ is $m$-nested in $\cU$ that $D(\xi') \subset \widetilde{\mbox{Cone}}_{\cU, \varepsilon}(\xi)$, hence $\xi' \in V_{\cU, \varepsilon}(\xi)$, which is absurd. Otherwise, let $\gamma$ be a geodesic path in $D(\xi')$ from $x'$ to a point of $D(\xi)$, such that $\gamma$ meets $D(\xi)$ in exactly one point. Let $\sigma$ be the last simplex touched by $\gamma$ before touching $D(\xi)$. The fact that $\cU'$ is $m$-nested in $ \cU$ implies that $\overline{EG_\sigma} \subset U_v$ for some (hence every) vertex $v$ of $\sigma \cap D(\xi)$, hence $\xi' \in U_v \subset V_{\cU, \varepsilon}(\xi)$, a contradiction. 
\end{proof}

\begin{thm}
The space $\overline{EG}$ is separable and metrisable.
\label{metrisable}
\end{thm}

\begin{proof}
It is secound countable by \ref{basis}, regular by \ref{regular} and satisfies the $T_0$ condition. Thus it is Hausdorff and the result follows from Urysohn's metrization theorem.
\end{proof}

\subsection{Sequential Compactness.}
In this subsection, we prove the following: 

\begin{thm}
The metrisable space $\overline{EG}$ is compact.
\label{compact}
\end{thm}

First of all, note that since $EG$ is dense in $\overline{EG}$ by \ref{basis}, it is enough to prove that any sequence in $EG$ admits a subsequence converging in $\overline{EG}$. Let $\left(\widetilde{x_n}\right)_{n \geq 0} \in \left(EG\right)^{\bbN}$. For every $n \geq 0$, let $x_n = p(\widetilde{x_n})$. Furthermore, to every $x_n$ we associate the finite sequence $\sigma_{0}^{(n)} = v_0,  \sigma_{1}^{(n)}, \ldots,$ of simplices met by $[v_0,x_n]$. Finally, let $l_n \geq 1$  be the number of simplices of such a sequence.\\

\begin{lem} Suppose that for all $k \geq 0$, $\left\lbrace \sigma_k^{(n)}, n \geq 0\right\rbrace $ is finite.
\begin{itemize}
 \item If $(l_n)$ admits a bounded subsequence, then $(x_n)$ admits a subsequence that converges to a point of $EG \cup \partial_{Stab} G$.
 \item Otherwise, $(x_n)$ admits a subsequence that converges to a point of $\partial X$.
\end{itemize}
\label{compact2}
\end{lem}

\begin{proof} Up to a subsequence, we can assume that there exist open simplices $\sigma_0, \sigma_1, \ldots$ such that for all $k  \geq 0, (\sigma_k^{(n)})_{n \geq 0}$ is eventually constant at $\sigma_k$. There are two cases to consider:
\begin{enumerate}
 \item[$(i)$] Up to a subsequence, there exists a constant $m \geq 0$ such that each geodesic $[v_0,x_n]$ meets at most $m$ simplices. This implies that the $x_n$ live in a finite subcomplex. Up to a subsequence, we can now assume that there exists a (closed) simplex $\sigma$ of $X$ such that $x_n$ is in the interior of $\sigma$ for all $n \geq 0$. This in turn implies that $\widetilde{x_n}$ in $\sigma \times \overline{EG_{\sigma}}$ (or more precisely in the image of $\sigma \times \overline{EG_{\sigma}}$ in $\overline{EG}$) for all $n \geq 0$. This space is compact since the canonical map $\sigma \times \overline{EG_{\sigma}} \hookrightarrow \overline{EG}$ is continuous by \ref{topologystab2}, hence we can take a convergent subsequence.
 \item[$(ii)$] Up to a subsequence, we can assume that $l_n \xrightarrow[n \infty]{} \infty$. For $r>0$, let $\pi_r:X \ra \overline{B}(v_0,r)$ the retraction on $\overline{B}(v_0,r)$ along geodesics starting at $v_0$. By assumption, we have that for every $r>0$, the sequence of projections $(\pi_r(x_n))_{n \geq 0}$ lies in a finite subcomplex of $X$. A diagonal argument then shows that, up to a subsequence, we can assume that all the sequences of projections $(\pi_m(x_n))_{n \geq 0}$ converge in $X$ for every $m \geq 0$. As the topology of $\overline{X}$ is the topology of the projective limit
$$ \overline{B}(v_0,1) \xleftarrow{\pi_1} \overline{B}(v_0,2) \xleftarrow{\pi_2} \ldots,$$ 
it then follows that $(x_n)$ converges in $\overline{X}$. As $l_n \ra \infty$, $(x_n)$ converges to a point $\eta$ of $\partial X$. The definition of the topology of $\overline{EG}$ now implies that $(\widetilde{x_n})$ converges to $\eta$ in $\overline{EG}$. \qedhere
\end{enumerate}
\end{proof}

\begin{lem} Suppose that there exists  $k \geq 0$ such that $\left\lbrace \sigma_k^{(n)}, n \geq 0\right\rbrace $ is infinite. Then $(x_n)$ admits a subsequence that converges to a point of $\partial_{Stab} G$.
\label{compact1}
\end{lem}

\begin{proof}
Without loss of generality, we can assume that such a $k$ is minimal. Up to a subsequence, we can assume that there exist open simplices $\sigma_1, \ldots, \sigma_{k-1}$ such that for all $n \geq 0$, $\sigma_0^{(n)} = \sigma_0, \ldots, \sigma_{k-1}^{(n)} = \sigma_{k-1}$, and $\left( \sigma_k^{(n)} \right)_{n \geq 0} $ is injective. By cocompacity of the action, we can furthermore assume (up to a subsequence) that the $\sigma_k^{(n)}$ are above a unique simplex of $Y$. This corresponds to embeddings $ \overline{EG_{\sigma_k^{(n)}}} \hookrightarrow \overline{EG_{\sigma_{k-1}}}$. By the convergence property, we can assume, up to a subsequence, that in $\overline{EG_{\sigma_{k-1}}}$ the sequence of subspaces $\overline{EG_{\sigma_k^{(n)}}}  $ uniformly converges to an element $ \xi \in \partial G_{\sigma_{k-1}}$. Let us prove that $(\widetilde{x_n})_{n \geq 0}$ converges to $\xi$ in $\overline{EG}$.\\

Since $\overline{EG}$ has a countable basis of neighbourhoods, it is enough to prove that for every $\varepsilon \in (0,1)$ and every $\xi$-family $\cU$ there exists a subsequence of $(\widetilde{x_n})$ lying in $V_{\cU, \varepsilon}(\xi)$.
By construction of $\xi$, we have $\sigma_{k-1} \subset D(\xi)$, and there exists a vertex $v_k$ of $D(\xi)$ such that $\sigma_k^{(n)} \subset$ $\mbox{st}(v_k)$ for all $n \geq 0$. Two cases may occur:\\
\begin{itemize}
 \item Up to a subsequence, all the $[v_0,x_n]$ leave $D^{\varepsilon}(\xi)$ inside $\sigma_k^{(n)}$. Since $\overline{EG_{\sigma_k^{(n)}}}$ uniformly converges to $\xi$ in $\overline{EG_{\sigma_{k-1}}}$ and thus in $\overline{EG_{v_k}}$, we can assume, up to a subsequence, that $\overline{EG_{\sigma_k^{(n)}}}  \subset U_{v_k}$ inside $\overline{EG_{\sigma_k}}$. This implies that $\widetilde{x_n} \in V_{\cU, \varepsilon}(\xi),$ which is what we wanted.\\
 \item Up to a subsequence, all the $[v_0,x_n]$ remain inside $D^{\varepsilon}(\xi)$ when inside $\sigma_k^{(n)}$. Up to a subsequence, we can further assume that all the $\sigma_{k+1}^{(n)}, n\geq 0$ are above a unique simplex of $Y$. Thus there exists a vertex $v_{k+1}$ of $D(\xi) \cap \overline{\mbox{st}}(v_k)$ such that $\sigma_{k+1}^{(n)} \subset st(v_{k+1})$ for all $n \geq 0$. 

 In particular we have $\sigma_k^{(n)} \subset $ st$(v_k) \cap $ st$(v_{k+1})$ and thus $\xi \in \partial G_{v_{k+1}}$. Since $\cU$ is a $\xi$-family, the fact that $\overline{EG_{\sigma_k^{(n)}}}$ uniformly converges to $\xi$ in $\overline{EG_{v_k}}$ implies that that $\overline{EG_{\sigma_k^{(n)}}}$ uniformly converges to $\xi$ in $\overline{EG_{v_{k+1}}}$. Note that since the sequence $(\sigma_k^{(n)})_{n \geq 0}$ takes infinitely many values, the finiteness lemma \ref{finite} implies that $(\sigma_{k+1}^{(n)})_{n \geq 0}$ also takes infinitely many values. Up to a subsequence, we can thus assume by the convergence property that $\overline{EG_{\sigma_{k+1}^{(n)}}}$ uniformly converges in $\overline{EG_{v_{k+1}}}$. As $\overline{EG_{\sigma_{k}^{(n)}}}$ uniformly converges to $\xi$ in $\overline{EG_{v_{k+1}}}$ , the same holds for $\overline{EG_{\sigma_{k+1}^{(n)}}}$, and we are back to the previous situation.
\end{itemize}
By iterating this algorithm, two cases may occur: 
\begin{itemize}
\item There is a rank $k' \geq k$ such that, up to a subsequence, all the $[v_0,x_n]$ leave $D^{\varepsilon}(\xi)$ while being inside $\sigma_{k'}^{(n)}$ and the same argument as before shows that we can take a subsequence satisfying $\widetilde{x_n} \in V_{\cU, \varepsilon}( \xi)$.
\item  Up to a subsequence, at every stage $k' \geq k$ all the $[v_0,x_n]$ remain within $D^{\epsilon}(\xi).$ In the latter case, the containment lemma \ref{containment} implies that there exists an integer $m \geq 0$ such that each geodesic segment $[v_0,x_n]$ meets at most $m$ simplices. Up to a subsequence, we can further assume that all the $[v_0,x_n]$ meet exactly $m$ simplices. Thus we can iterate our algorithm up to rank $m$, which yields the existence of a vertex $v_m$ of $D^{\varepsilon}(\xi)$ such that $\sigma_m^{(n)} \subset $ st$(v_m)$ for all $n \geq 0$ and such that  $\overline{EG_{\sigma_m^{(n)}}} $ uniformly converges to $\xi$ in $\overline{EG_{v_{m}}}$. Up to a subsequence, we can furthermore assume that $ \overline{EG_{\sigma_m^{(n)}}} \subset U_m$ in $\overline{EG_{v_{k+1}}}$ for all $n \geq 0$. This in turn implies $\widetilde{x_n} \in W_{\cU, \varepsilon}(\xi)$, hence $\widetilde{x_n} \in V_{\cU, \varepsilon}(\xi)$ and we are done. \qedhere
\end{itemize}
\end{proof}

\begin{proof}[Proof of \ref{compact}]
 This follows immediately from \ref{metrisable}, \ref{compact1} and \ref{compact2}.
\end{proof}

As a direct consequence, we get the following convergence criterion. 

\begin{cor}
Let $(z_n)$ be a sequence in $\overline{EG}$.
\begin{itemize}
\item $z_n$ converges to a point $\eta \in \partial X$ if and only if the sequence of coarse projections $\bar{p}(z_n)$ uniformly converges to $\eta$.
\item Suppose that for $n$ large enough, geodesics from $v_0$ to $\bar{p}(z_n)$ meet the domain of some element $\xi \in \partial_{Stab} G$. For every $n$, choose $x_n \in \bar{p}(z_n)$ and let $\sigma_n$ be the first simplex touched by the geodesic $[v_0, x_n]$ after leaving $D(\xi)$. If the sequence of subsets $\partial G_{\sigma_n}$ uniformly converges to $\xi$ in $\partial G_v$, for some vertex $v \in D(\xi)$, then $(z_n)$ converges to $\xi$. \qedhere
\end{itemize}
Furthermore, the convergence can be made uniform on a subset $K \subset \overline{EG}$ if the aforementionned conditions are realised uniformly on $K$. \qed
\label{convergencecriterion}
\end{cor}

\section{The properties of $\partial G$.}
In this section we prove the following: 

\begin{thm}
$(\overline{EG}, \partial G)$ is an $E\cZ$-structure in the sense of Farrell-Lafont.
\label{EZstructure}
\end{thm}

\subsection{The $\cZ$-set property}

Here we prove the following:

\begin{prop}
 $\partial G$ is a $\cZ$-set in $\overline{EG}$.
\label{Zset}
\end{prop}

Proving this property is generally technical. However, Bestvina and Mess proved in \cite{BestvinaMessBoundaryHyperbolic} a useful lemma ensuring that a given set is a $\cZ$-set in a bigger set, which we now recall.
\begin{lem}[Bestvina-Mess \cite{BestvinaMessBoundaryHyperbolic}]  Let $(\widetilde{X}, Z)$ be a pair of finite-dimensional metrisable compact spaces with $Z$ nowhere dense in $\widetilde{X}$, and such that $X= \widetilde{X} \setminus Z$ is contractible and locally contractible, with the following condition holding:
\begin{itemize}
 \item [(*)] For every $z \in Z$ and every neighbourhood $\widetilde{U}$ of $z$ in $\widetilde{X}$, there exists a neighbourhood $\widetilde{V}$ contained in $\widetilde{U}$ and such that 
$$ \widetilde{V} \setminus Z \hookrightarrow \widetilde{U} \setminus Z$$
is null-homotopic.
\end{itemize}
Then $\widetilde{X}$ is an Euclidian retract and $Z$ is a $\cZ$-set in $\widetilde{X}$.
\label{BestvinaMess}
\end{lem}
We now use this lemma to prove that the boundary $\partial G$ is a $\cZ$-boundary in $\overline{EG}$. 
\begin{lem} 
 $\overline{EG}$ and $\partial G$ are finite-dimensional.
\label{finitedimensional}
\end{lem}
\begin{proof}
We have 
$$ \partial G = \bigg( \bigcup_{v \in V(X)} \partial G_v \bigg) \cup \partial X.$$
Each vertex stabiliser boundary is a $\cZ$-boundary in the sense of Bestvina, hence finite-dimensinal, and they are closed subspaces of $\partial G$ by \ref{inducedtopology}. As the action of $G$ on $X$ is cocompact, their dimension is uniformly bounded above, so the countable union theorem implies that $\bigcup_{v \in V(X)} \partial G_v$ is finite-dimensional. Furthermore, $X$ is a CAT(0) space of finite geometric dimension, so its boundary has finite dimension by a result of Caprace \cite{CapraceInfinityCAT(0)}. Thus, the classical union theorem implies that $\partial G$ is finite-dimensional. Now $\overline{EG} = EG \cup \partial G$. $EG$ is a CW-complex that can be decomposed as the countable union of its closed cells, all of which having a dimension bounded above by $\mbox{dim}(X). \mbox{ sup}_\sigma (\mbox{dim }EG_\sigma)$. It follows from the countable union theorem in covering dimension theory that $EG$ is finite dimensional, and the same holds for $\overline{EG}$ by the classical union theorem.
\end{proof}

We now turn to the proof of the $\cZ$-set property, using the lemma of Bestvina-Mess recalled above. As usual, the proof splits in two cases, depending on the nature of the point of $\partial G$ that we consider.
\begin{prop}
Let $\eta \in \partial X$ and $U$ be a neighbourhood of $\eta$ in $\overline{X}$. Then there exists a subneighbourhood $U' \subset U$ of $\eta$ in $\overline{X}$ such that the inclusion
$$V_{U'}(\eta) \setminus \partial G \hra V_U(\eta) \setminus \partial G$$
is null-homotopic.
\label{Zset1}
\end{prop}
\begin{proof}
The lemma \ref{CAT(0)nesting} yields a neighbourhood $U'$ of $\eta$ in $\overline{X}$ such that $d(U' \cap X, X \setminus U) > 1$. In particular, $\mbox{Span}(U'\setminus \partial X) \subset U$, and $p^{-1}(\mbox{Span}(U'\setminus \partial X))$ can be seen as the realisation of a complex of spaces over $\mbox{Span}(U'\setminus \partial X)$ the fibres of which are contractible. Thus \ref{homotopytype} implies that the projection $p^{-1}(\mbox{Span}(U'\setminus \partial X)) \ra \mbox{Span}(U'\setminus \partial X)$ is a homotopy equivalence. Now \ref{contractibleneighbourhoods} yields another neighbourhood $U'' \subset U'$ of $\eta$ in $\overline{X}$ such that $U'' \setminus \partial X$ is contractible. We thus have the following commutative diagram:
\[ \xymatrix{
    V_{U}(\eta) \setminus  \partial G  & ~~ p^{-1}(\mbox{Span} (U' \setminus \partial X)) \ar[d]^\simeq \ar@{_{(}->}[l] & V_{U''}(\eta) \setminus \partial G  \ar[d] \ar@{_{(}->}[l]  \\
    & \mbox{Span}( U' \setminus \partial X)  & U'' \setminus \partial X\ar@{_{(}->}[l]^0  .
  } \]
Now since $U''\setminus \partial X$ is contractible, the inclusion $V_U'(\eta) \setminus \partial G \hra V_U(\eta) \setminus \partial G$ is null-homotopic.
\end{proof}

\begin{prop}
Let $\xi \in \partial_{Stab} G$, $\varepsilon \in (0,1)$ and $\cU$ a $\xi$-family. Then there exists a $\xi$-family $\cU'$ such that $V_{\cU', \varepsilon}(\xi)$ is a subneighbourhood of $V_{\cU, \varepsilon}(\xi)$ and such that the inclusion 
$$ V_{\cU', \varepsilon}(\xi) \setminus \partial G \hra V_{\cU, \varepsilon}(\xi) \setminus \partial G$$
is null-homotopic.
\label{Zset2}
\end{prop}

\begin{lem} 
There exist $\xi$-families $\cU', \cU''$, a subcomplex $X'$ of $X$ with $\widetilde{\mbox{Cone}}_{\cU'', \varepsilon}(\xi) \subset X' \subset \widetilde{\mbox{Cone}}_{\cU, \varepsilon}(\xi)$, and a subset $C'$ of $EG$ with $V_{\cU'', \varepsilon}(\xi) \setminus  \partial G \subset C' \subset V_{\cU, \varepsilon}(\xi) \setminus  \partial G$, such that $p(C') \subset X'$ and the projection map $C' \ra X'$ is a homotopy equivalence. 
\end{lem}
\begin{proof}
Let $\cU'$ be a $\xi$-family that is $2$-refined in $\cU$ and $d_\xi$-nested in $\cU$. It follows from the refinement lemma \ref{refinement} that $\mbox{Span}\bigg(\mbox{Cone}_{\cU', \varepsilon}(\xi) \bigg) \subset \widetilde{\mbox{Cone}}_{\cU, \varepsilon}(\xi).$ By \ref{subneighbourhood}, we have $V_{\cU', \varepsilon}(\xi)  \subset V_{\cU, \varepsilon}(\xi)$. Let 
$$X' = \mbox{Span}( \mbox{Cone}_{\cU', \varepsilon}(\xi) ) \cup \bigg(\overline{D^{\varepsilon}(\xi)} \cap \widetilde{\mbox{Cone}}_{\cU, \varepsilon}(\xi) \bigg).$$
Note that it is possible to give $\overline{D^{\varepsilon}(\xi)} \cap \widetilde{\mbox{Cone}}_{\cU, \varepsilon}(\xi)$ a simplicial structure from that of $X$ such that a vertex of $\overline{D^{\varepsilon}(\xi)} \cap \widetilde{\mbox{Cone}}_{\cU, \varepsilon}(\xi)$ for that structure either is a vertex of $D(\xi)$ or belongs to an edge in $X$ between a vertex of $D(\xi)$ and a vertex of $X \setminus D(\xi)$. Furthermore, we can give $\mbox{Span}( \mbox{Cone}_{\cU', \varepsilon}(\xi) )$ a simplicial structure that is finer that that of $X$, whose vertices are the vertices of $\mbox{Span}(\mbox{Cone}_{\cU', \varepsilon})$ and vertices of $\overline{D^{\varepsilon}(\xi)} \cap \widetilde{\mbox{Cone}}_{\cU, \varepsilon}(\xi)$ (for its given simplicial structure), that is compatible with that of $\overline{D^{\varepsilon}(\xi)}$, and which turns $X'$ into a simplicial complex such that an open simplex is completely contained either in $\overline{D^{\varepsilon}(\xi)}$ or in $X \setminus D^{\varepsilon}(\xi)$ (see Figure $3$). Thus $X'$ is endowed with a simplicial structure.

\begin{center}
\begingroup%
  \makeatletter%
  \providecommand\color[2][]{%
    \errmessage{(Inkscape) Color is used for the text in Inkscape, but the package 'color.sty' is not loaded}%
    \renewcommand\color[2][]{}%
  }%
  \providecommand\transparent[1]{%
    \errmessage{(Inkscape) Transparency is used (non-zero) for the text in Inkscape, but the package 'transparent.sty' is not loaded}%
    \renewcommand\transparent[1]{}%
  }%
  \providecommand\rotatebox[2]{#2}%
  \ifx\svgwidth\undefined%
    \setlength{\unitlength}{357.275bp}%
    \ifx\svgscale\undefined%
      \relax%
    \else%
      \setlength{\unitlength}{\unitlength * \real{\svgscale}}%
    \fi%
  \else%
    \setlength{\unitlength}{\svgwidth}%
  \fi%
  \global\let\svgwidth\undefined%
  \global\let\svgscale\undefined%
  \makeatother%
  \begin{picture}(1,0.43799657)%
    \put(0,0){\includegraphics[width=\unitlength]{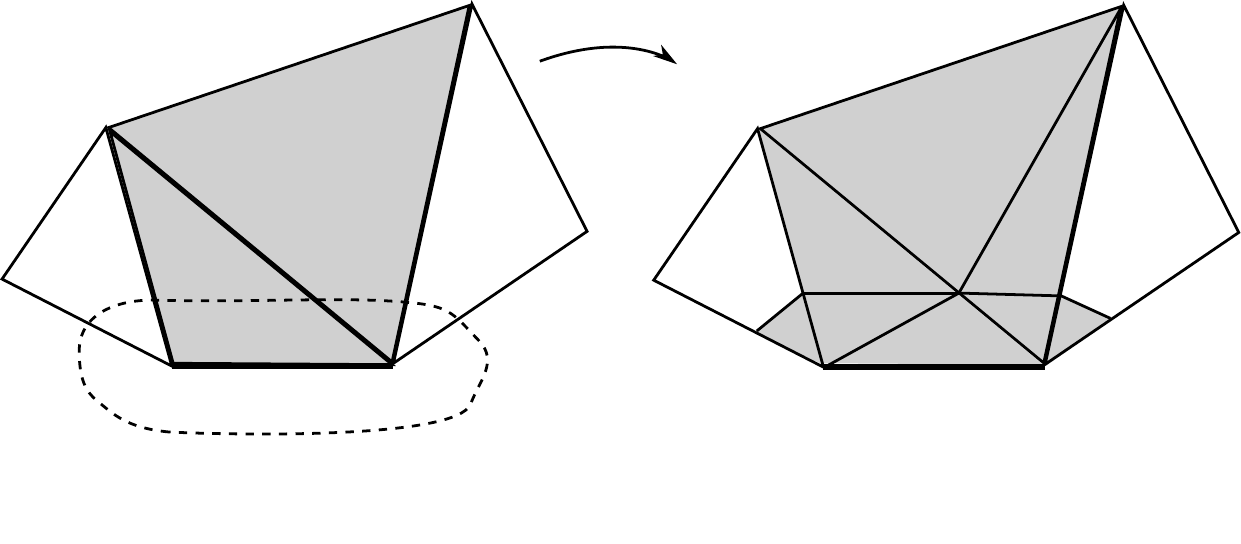}}%
    \put(0.09211234,0.08409978){\color[rgb]{0,0,0}\makebox(0,0)[lt]{\begin{minipage}{0.15720237\unitlength}\raggedright $D^{\varepsilon}(\xi)$\end{minipage}}}%
    \put(-0.04586386,0.5014212){\color[rgb]{0,0,0}\makebox(0,0)[lt]{\begin{minipage}{0.28839283\unitlength}\raggedright $ $\end{minipage}}}%
    \put(0.01752225,0.41635863){\color[rgb]{0,0,0}\makebox(0,0)[lt]{\begin{minipage}{0.31069277\unitlength}\raggedright $\mbox{Span}(\mbox{Cone}_{\cU', \varepsilon}(\xi))$\end{minipage}}}%
    \put(0.62931473,0.40415928){\color[rgb]{0,0,0}\makebox(0,0)[lt]{\begin{minipage}{0.19565477\unitlength}\raggedright $X'$\end{minipage}}}%
  \end{picture}%
\endgroup%

Figure $3$.
\end{center}

\noindent We now define a contratible open subset $C'_\sigma$ of $EG_\sigma$ for every open simplex $\sigma$ of $X'$. This will allow us to define the following subset of $EG$:
$$C' =  \bigcup_{\sigma \in S(X')} \sigma \times C_\sigma'.$$
Note that although $C'$ is not naturally the realisation of a complex of spaces in the sense of the first section, it is nonetheless possible to endow it with one, so as to use \ref{homotopytype} to show that the natural projection $C' \ra X'$ is a homotopy equivalence. 

We first define these spaces $C'_\sigma$ for vertices of $X'$. Let $v$ be such a vertex.
\begin{itemize}
\item If $v$ is a vertex of $D(\xi)$, the compactification $\overline{EG_v}$ is locally contractibe so we can choose a contractible open set $U_v'$ of $\overline{EG_v}$ contained in $\overline{U_v}$ and containing $\xi$, and set $C'_v = U_v' \cap EG_v$. As $\partial G_v$ is a $\cZ$-boundary, $C'_v$ is a contractible open subset. 
 \item If $v$ does not belong to $D^\varepsilon(\xi)$, set $C_v' = \overline{EG_v}$.
 \item If $v$ is a vertex of $\overline{D^{\varepsilon}(\xi)} \setminus D(\xi)$ (for the chosen simplicial structure of $\overline{D^{\varepsilon}(\xi)} \subset X'$), then either $v$ belongs to $\mbox{Span}(\mbox{Cone}_{\cU', \varepsilon}(\xi))$, in which case we set $C_v'=EG_v$, or it does not, in which case $v$ belongs to a unique edge $e$ (for the simplicial structure of $X$) between a vertex $v'$ of $D(\xi)$ and a vertex of $X \setminus D(\xi)$. In that case, $\overline{EG_e}$ is contained in $U_{v'}$ since $\cU'$ is nested in $\cU$ and we set $C_v'= EG_e$.
\end{itemize}
We now define the subsets $C_\sigma'$ for simplices $\sigma \subset X'$. Let $\sigma$ be such a simplex, and let $\sigma'$ be the unique open simplex of $X$ such that $\sigma \subset \sigma'$ as subsets of $X$. We set $C_\sigma' = EG_{\sigma'}$. 

We define the space $C' =  \bigcup_{\sigma \in S(X')} \sigma \times C'_\sigma$. As explained above, the projection $C' \ra X'$ is a homotopy equivalence. Furthermore, we can choose a $\xi$-family $\cU''$ small enough so that the subset  $V_{\cU'', \varepsilon}(\xi) \setminus \partial G$ is contained in $C'$.  
\end{proof}

\begin{proof}[Proof of \ref{Zset2}] We apply the previous lemma twice to get the following commutative diagram:
\small \[ \xymatrix{
    V_{\cU, \varepsilon}(\xi) \setminus  \partial G  & ~C' \ar[d]^\simeq \ar@{_{(}->}[l] & ~V_{\cU'', \varepsilon}(\xi) \setminus \partial G  \ar[d] \ar@{_{(}->}[l] & ~C^{(3)}  \ar[d]^\simeq \ar@{_{(}->}[l] \\
    & ~X'  & ~\widetilde{\mbox{Cone}}_{\cU'', \varepsilon}(\xi)  \ar@{_{(}->}[l]  &  ~X^{(3)} \ar@{_{(}->}[l]^{~~~~~0} .
  } \]
\normalsize Since $X^{(3)}$ retracts by strong deformation (along geodesics starting at $v_0$) inside $\widetilde{\mbox{Cone}}_{\cU', \varepsilon}(\xi) $ on the contractible subcomplex  $D(\xi)$ (relatively to $D(\xi)$), the inclusion $X^{(3)} \hra \widetilde{\mbox{Cone}}_{\cU', \varepsilon}(\xi)$ is nullhomotopic, hence $ C^{(3)} \hra V_{\cU, \varepsilon}(\xi) \setminus \partial G$ is null-homotopic. As there exists a $\xi$-family $\cU^{(4)}$ such that $V_{\cU^{(4)}, \varepsilon}(\xi) \setminus \partial G \hra C^{(3)}$, this concludes the proof.
\end{proof}

\begin{proof}[Proof of \ref{Zset}:]
Thus, \ref{compact} and \ref{finitedimensional} together with \ref{Zset1} and \ref{Zset2} yield the desired result.
\end{proof}

\subsection{Compact sets fade at infinity}
Here we prove the following: 

\begin{prop}
Compacts subsets of $EG$ \textit{fade at infinity} in $\overline{EG}$, that is, for every $x \in \partial G$, every neighbourhood $U$ of $x$ in $\overline{EG}$ and every compact $K \subset EG$, there exists a subneighbourhood $V \subset U$ of $x$ such that any $G$-translate of $K$ meeting $V$ is contained in $U$.
\label{fade}
\end{prop}

 As usual, we split the proof in two parts, depending on the nature of the points considered.

\begin{prop}
 Let $\eta \in \partial X$. For every neighbourhood $U$ of $\eta$ in $\overline{X}$ and every compact subset $K \subset EG$, there exists a neighbourhood $U'$ of $\eta$ contained in $U$ and such that any $G$-translate of $K$ meeting $V_{U'}(\eta)$ is contained in $V_U(\eta)$.
\label{fade1}
\end{prop}
\begin{proof}
By \ref{CAT(0)nesting}, let $U'$ be a neighbourhood of $\eta$ in $\overline{X}$ which is contained in $U$ and such that 
$$d(U', X \setminus U) > \mbox{diam}(p(K)).$$
Let $g \in G$ such that $ gK$ meets $ V_{U'}(\eta) $. Since $G$ acts on $X$ by isometries, we have
$$\mbox{diam}\left( p(g.K) \right)  = \mbox{diam}\left( g.p(K)  \right) = \mbox{diam}\left( p(K) \right),$$
which implies that $gK \subset V_U(\eta)$.
\end{proof}
The proof for points of $\partial_{Stab} G$ is slightly more technical. We start by defining a class of compact sets of $EG$ which are easy to handle.
\begin{definition}
Let $F$ be a finite subcomplex of $X$, together with a collection $(K_\sigma)_{\sigma \in S(F)}$ of non empty compact subsets of $EG_\sigma$ for every simplex $\sigma$ of $F$. Suppose that for every simplex $\sigma$ of $F$ and every face $\sigma'$ of $\sigma$, we have $\phi_{\sigma', \sigma}(K_{\sigma}) \subset K_{\sigma'}.$
Then the set 
$$\bigcup_{\sigma \in \cS(F)} \sigma \times K_\sigma.$$
is called a  \textit{standard compact subset of $EG$ over $F$}. Every compact subset of $EG$ obtained in such a way is called a standard compact of $EG$.
\end{definition}
Note that the projection in $X$ of any compact subset of $EG$ meets finitely many simplices of $X$, so every compact subset of $EG$ may be seen as a subset of a standard compact subset of $EG$.

\begin{definition}
Let $\xi \in \partial_{Stab} G$ and $\cU$ a $\xi$-family. We define $W_{\cU}(\xi)$ as the set of elements $\widetilde{x}$ of $EG$ whose projection $x \in X$ belongs to the domain of $\xi$ and is such that for some (hence any) vertex $v$ of $\sigma_x \cap D(\xi)$ we have 
$$\phi_{v, \sigma_x}(\widetilde{x}) \in U_v.$$
\end{definition}

Before proving that compact sets fade near points of $\partial_{Stab} G$, we prove the following lemma.

\begin{lem}
 Let $\xi \in \partial_{Stab} G$, $\varepsilon \in ~(0,1)$ and $\cU$ a $\xi$-family. Let $K$ be a compact subspace of $EG$. Then there exists a $\xi$-family $\cU'$ contained in $\cU$ such that for every element $g \in G$, the following holds:

If $gK$ meets $W_{\cU'}(\xi)$, then $gK \cap p^{-1}(D(\xi))$ is contained in $W_{\cU}(\xi).$
\label{lemmesansnom2}
\end{lem}
\begin{proof}
Let $L$ be a standard compact subset of $EG$ over the (finite) flag complex defined by $\mbox{Span } p(K)$. By choosing the $L_\sigma$ big enough, we can assume that $L$ contains $K$. Let $N \geq 0$ be such that any two vertices of $L$ can be joined by a sequence of at most $N$ adjacent vertices. 

Since $D(\xi)$ and $p(L)$ meet finitely many vertices of $X$, there are only finitely many elements of $G$ such that $g.p(L)$ meets $D(\xi)$ up to left multiplication by an element of $G_v, v \in V(\xi)$. Let $(g_\lambda)_{\lambda \in \Lambda}$ be such a finite family. For every vertex $v$ of $V(\xi)$, $\left\lbrace g_\lambda L \cap EG_v , \lambda \in \Lambda \right\rbrace $ is a finite (possibly empty) collection of compact subsets of $EG_v$. Since $\partial G_v$ is a Bestvina boundary for $G_v$, compact subsets fade at infinity in $\overline{EG_v}$, so there exists a subneighbourhood $U_v'$ of $U_v$ such that any $G_v$-translate of one of these $g_\lambda L$ meeting $U_v'$ is contained in $U_v$. Repeating this procedure $N+1$ times, we get a sequence of $\xi$-families denoted 
$$ \left\lbrace U_v, v  \in V(\xi)\right\rbrace \supset \cU^{[N]} \supset \cU^{[N-1]} \supset \ldots \supset \cU^{[0]}.$$ 
Let $g \in G$ such that $gK$ meets $ W_{\cU'}(\xi)$, and let $w$ be a vertex of $D(\xi)$ such that $gK$, hence $gL$, meets $U_{w}^{[0]}$. In order to prove the lemma, it is enough to show by induction on $k = 0, \ldots, N$ the following: \\

$(H_k):$ For every chain of adjacent vertices $w_0=w, w_1, \ldots, w_k$ of $D(\xi)$ such that $gL$ meets $EG_{w_0}, \ldots, EG_{w_k}$, then $gL \cap EG_{w_k} \subset U_w^{[k+1]}$.\\

Since $gL$ meets $D(\xi)$, let $\lambda \in \Lambda$ such that $gL=g_\lambda L$ pointwise. The result is true for $k=0$  by definition of $\cU^{[0]}$ and $\cU^{[1]}$. Suppose we have proven it up to rank $k$, and let $w_0=w, w_1,\ldots, w_{k+1}$ a chain of vertices of $D(\xi)$ such that $gL$ meets $EG_{w_0}, \ldots, EG_{w_k}$. By induction hypothesis, we already have $gL \cap EG_{w_k} \subset U_{w_k}^{[k+1]}$. Since $p(L)$ is a flag complex, it follows from the fact that $gL$ meets $EG_{w_k}$ and $EG_{w_{k+1}}$ that $gL$ also meets $EG_{[w_k, w_{k+1}]}$. In particular, since $gL \cap EG_{w_{k}} \subset U_{w_k}^{[k+1]}$, it follows from the properties of $\xi$-families that $gL \cap EG_{w_{k+1}}$ meets $U_{w_{k+1}}^{[k+1]}$. This in turn implies that $gL \cap EG_{w_{k+1}} \subset U_{v_{k+1}}^{[k+2]}$, which concludes the induction.
\end{proof}
\begin{prop}
 Let $\xi \in \partial_{Stab} G$, $\varepsilon \in (0,1)$ and $\cU$ a $\xi$-family. Let $K$ be a connected compact subset of $EG$. Then there exists a $\xi$-family $\cU'$ contained in $\cU$ and such that every $G$-translate of $K$ meeting $V_{\cU', \varepsilon}(\xi)$ is contained in $V_{\cU, \varepsilon}(\xi).$
\label{fade2}
\end{prop}
\begin{proof}
Let $k$ be the number of simplices met by $p(K)$, and let $\cU'$ be a $\xi$-family that is $k$-refined in $\cU$. Applying the previous proposition to $V_{\cU', \varepsilon}(\xi)$ yields a $\xi$-family $\cU''$. Finally, let $\cU'''$ be a $\xi$-family that is $k$-refined in $\cU''$.

Suppose that $gK$ meets $ V_{\cU''', \varepsilon}(\xi)$, and let $\widetilde{x_0} \in gK \cap V_{\cU''', \varepsilon}(\xi)$. Let $\widetilde{x} \in gK$, and let us prove that $\widetilde{x} \in V_{\cU, \varepsilon}(\xi)$. Since $p(K)$ is connected, let $\gamma$ be a path from $x_0=p(\widetilde{x_0})$ to $x=p(\widetilde{x})$ in $p(gK)$. This yields a path of open simplices $\sigma_1, \ldots, \sigma_n$, with $n \leq k$. If $gK$ does not meet $D(\xi)$, the refinement lemma \ref{refinement} implies that $\sigma_n \subset \widetilde{\mbox{Cone}}_{\cU, \varepsilon}(\xi)$, and $\widetilde{x} \in V_{\cU, \varepsilon}(\xi)$.

Otherwise, let $n_0$ (resp. $n_1$)  be such that $\sigma_{n_0}$ (resp. $\sigma_{n_1}$) is the first (resp. the last) simplex contained in $D(\xi)$. If $x_0$ is not in $D(\xi)$, we can apply the refinement lemma \ref{refinement} to the path $\sigma_1, \ldots, \sigma_{n_0-1}$, which implies $\sigma_{n_0-1} \subset \mbox{st}_{\cU''}(D(\xi))$. In particular, we see that $gK$ meets $W_{\cU''}(\xi)$, which is also true if $x_0$ is in $D(\xi)$. Now by definition of $\cU''$, we have that $gK \cap p^{-1}(D(\xi)) \subset W_{\cU' }(\xi)$. If $\gamma$ goes out of $D(\xi)$ after $\sigma_{n_1}$, then $\sigma_{n_1+1} \subset \mbox{st}_{\cU'}(\xi)$, and we can apply the refinement lemma \ref{refinement} to the path of simplices  $\sigma_{n_1+1}, \ldots, \sigma_{n}$. In any case, we get in the end $\widetilde{x} \in  V_{\cU, \varepsilon}(\xi)$, which concludes the proof.
\end{proof}

\begin{proof}[Proof of \ref{fade}:]
This follows from \ref{fade1} and \ref{fade2}
\end{proof}

\begin{proof}[Proof of \ref{EZstructure}:] 
This follows from \ref{compact}, \ref{basepoint}, \ref{Zset}, and \ref{fade}.
\end{proof}

\subsection{Proof of the main theorem.}
We are now ready to conclude the proof the main theorem. 

\begin{lem}
 Let $X, Y$ and $G$ as in the statement of the main theorem. Then for every simplex $\sigma$ of $Y$, the embedding $\overline{EG_\sigma} \hra \overline{EG}$ realises an equivariant homeomorphism from $\partial G_\sigma$ to $\Lambda G_\sigma \subset \Lambda G$. Moreover, for every pair $H_1, H_2$ of subgroups in the family $\cF = \left\lbrace \bigcap_{i=1}^{n} g_i G_{\sigma_i} g_i^{-1} | ~ g_1, \ldots, g_n \in G, ~ \sigma_1, \ldots, \sigma_n \in \mbox{S}(Y), ~ n\in \bbN \right\rbrace,$ we have $\Lambda H_1 \cap \Lambda H_2 = \Lambda(H_1 \cap H_2) \subset  \partial G.$
\label{hereditary1}
\end{lem}

\begin{proof}
The equivariant embedding $\overline{EG_\sigma} \hra  \overline{EG}$ induces an equivariant embedding $\partial G_\sigma \hra \Lambda G_\sigma \subset \partial G$. But since $\overline{EG_\sigma}$ is a closed subspace of $\overline{EG}$ by \ref{inducedtopology}, and which is stable under the action of $G_\sigma$, the reverse inclusion $\Lambda G_\sigma \subset \partial G_\sigma$ follows.

Now let $\sigma_1, \ldots, \sigma_n$ be simplices of $X$. The inclusion 
$$\Lambda( \bigcap_{1 \leq i \leq n} G_{\sigma_i}) \subset \bigcap_{1\leq i \leq n} \Lambda G_{\sigma_i} $$
is clear, and the reverse inclusion follows directly from \ref{xipath1}.
\end{proof}

\begin{lem}
 Let $X$ and $G$ be as in the satement of the main theorem. Then for every simplex $\sigma$ of $X$, the embedding $\overline{EG_\sigma} \hra \overline{EG}$ satisfies the convergence property.
\label{hereditary2}
\end{lem}

\begin{proof}
 Let $(g_nG_\sigma)$ be a sequence of distinct $G$-cosets. This yields an injective sequence of simplices $(g_n \sigma)$ of $X$. Let $\widetilde{x}$ be any point of $EG_\sigma$. By compacity of $\overline{EG}$, we can assume up to a subsequence that $g_n\widetilde{x}$ converges to an element $l \in \overline{EG}$. But it follows immediately from \ref{compact1} and \ref{compact2} that $l \in \partial G$ and that $g_n \overline{EG_\sigma}$ uniformly converges to $l$.
\end{proof}

\begin{lem}
 Let $X$ and $G$ be as in the statement of the main theorem. Then for every simplex $\sigma$ of $X$, the group $G_\sigma$ is of finite height in $G$. 
\label{hereditary3}
\end{lem}

\begin{proof}
 Let $g_1G_\sigma, \ldots, g_nG_\sigma$ be distinct $G$-cosets such that $g_1G_\sigma g_1^{-1} \cap \ldots \cap g_n G_\sigma g_n^{-1}$ is infinite. Thus the simplices $g_1 \sigma, \ldots, g_n\sigma$ of $X$ are distinct and such that their stabilisers have an infinite intersection. But as there is a uniform bound on the number of simplices contained in the domain of an element of $\partial_{Stab} G$ by \ref{finitedomain}, lemma \ref{stabiliserfix} implies that there is a uniform bound on the number of simplices whose stabilisers have an infinite intersection, hence the result.
\end{proof}

\begin{proof}[Proof of \ref{maintheorem1}:]
This follows from \ref{EZstructure}, \ref{hereditary1}, \ref{hereditary2} and \ref{hereditary3}
\end{proof}

\subsection{Boundaries in the sense of Carlsson-Pedersen.}
So far we have been concerned with the notion of an $E\cZ$-structure in the sense of Farrell-Lafont. We now turn to a slightly stronger notion of boundary, which also has stronger implications for the Novikov conjecture.
\begin{definition}
 Let $G$ be a group endowed with an $E\cZ$-structure in the sense of Farrell-Lafont $(\overline{EG}, \partial G)$. We say that $(\overline{EG}, \partial G)$ is an $E\cZ$-\textit{structure in the sense of Carlsson-Pedersen} if in addition we have: \\
For every finite group $H$ of $G$, the fixed point set $\overline{EG}^H$ is nonempty and admits $EG^H$ as a dense subset.
\end{definition}
The importance of such finer structures comes from the following implication.
\begin{thm}[\cite{CarlssonPedersenEZBoundaries} ]
 If $G$ admits an $E\cZ$-structure in the sense of Carlsson-Pedersen, then $G$ satisfies the integral Novikov conjecture.
\end{thm}
In our context, we will need an additional assumption on these $E\cZ$-structures. As explained below, this is by no mean a restrictive assumption.
\begin{definition} 
We say that an $E\cZ$-structure in the sense of Carlsson-Pedersen $(\overline{EG}, \partial G)$  is \textit{strong} if in addition we have the following:
\begin{center}For every finite group $H$ of $G$, $(\partial G)^H$ is either empty or a $\cZ$-set in $\overline{EG}^H$.\end{center}
\end{definition}

Without any assumption of a strong $E\cZ$-structure, it is still possible to prove the following partial result.
\begin{lem}
 Let $H \subset G$ be a finite subgroup. Then the closure of $EG^H$ in $\overline{EG}$ is exactly $\overline{EG}^H$.
\label{CPdense}
\end{lem}
\begin{proof}
As $EG$ is a classifying space for proper actions of $G$, $EG^H$ is nonempty. We now prove that it is dense in $\overline{EG}^H$.

Let $\xi \in \partial_{Stab} G \cap \overline{EG}^H$. The domain $D(\xi)$ is thus stable under the action of $H$. As $D(\xi)$ is a finite convex subcomplex of $X$, the fixed point theorem for CAT(0) spaces implies that there is a point of $D(\xi)$ fixed by $H$. Since the action is without inversion, we can further assume that $H$ fixes a vertex $v$ of $D(\xi)$. Moreover, $EG_v^H$ is dense in $\overline{EG_v^H}$. Thus, by definition of a basis of neighbourhoods at $\xi$, any neighbourhood of $\xi$ in $\overline{EG}$ meets $EG^H$. 

Now let $\eta \in \partial X \cap \overline{EG}^H$. Let $\gamma$ be a geodesic from a point of $X^H$ to $\eta$. Then $\gamma$ is fixed pointwise by $H$. Let $U$ be a neighbourhood of $\eta$ in $\overline{X}$. Since the path $\gamma$ eventually meets $U$, let $\sigma$ be a simplex of $X$ contained in $U$ and met by $\gamma$. Thus $\sigma$ is fixed pointwise by $H$. Now since $EG_\sigma ^H$ is nonempty by assumption, it follows that $EG ^H$ meets $V_U(\eta)$, and the result follows.
\end{proof}

However, the previous reasoning does not show the contractibility of $\overline{EG}^H$. We now reformulate our main theorem in the setting of $E\cZ$-structures in the sense of Carlsson-Pedersen.

\begin{definition}
  An $E\cZ$-complex of spaces \textit{in the sense of Carlsson-Pedersen} (compatible with the complex of groups $G(\cY)$) is a complex of spaces over a fundamental domain for the action satisfying the axioms of a compatible $E\cZ$-complex of spaces, with strong $E\cZ$-structures in the sense of Carlsson-Pedersen instead of $E\cZ$-structures in the sense of Farrell-Lafont.
\end{definition}

\begin{thm}
The combination theorem for boundaries of groups \ref{maintheorem1} remains true if ones replaces ``$E\cZ$-complexes of spaces'' with ``$E\cZ$-complexes of spaces in the sense of Carlsson-Pedersen''.
\end{thm}

\begin{proof}
 The only thing to prove is that $(\overline{EG}, \partial G)$ is an $E\cZ$-structure in the sense of Carlsson-Pedersen. We already know that it is an $E\cZ$-structure in the sense of Farrell-Lafont by Theorem 0.1 in the case of $E\cZ$-structures in the sense of Farrell-Lafont. Let $H$ be a finite subgroup of $G$. To prove that $\overline{EG}^H$ is contractible, we want to apply the lemma \ref{BestvinaMess} of Bestvina-Mess to the pair $(\overline{EG}^H, \overline{EG}^H \setminus EG^H)$. 

In order to do this, first notice that $EG^H$ is nothing but the complex of spaces over $X^H$ with fibres the subcomplexes $EG_\sigma^H$ of $EG_\sigma$. Thus, it is possible to apply the exact same reasoning with $X^H$ in place of $X$ and the $EG_\sigma^H$ in place of the $EG_\sigma$. As $X^H$ is a convex, hence contractible sucomplex of $X$, this is enough to recover the fact that $EG^H$ is contractible. 

Now, notice that, because of \ref{CPdense},  $\overline{EG}^H$ is obtained from $EG^H$ by the same procedure as before, compatifying every $EG_\sigma^H$ (for $\sigma$ a simplex fixed under $H$) by $\overline{EG_\sigma}^H$ and adding the visual boundary of the CAT(0) subcomplex $X^H$, $\partial (X^H) = (\partial X)^H$. We now briefly indicate why this is enough to prove the $\cZ$-set property for $(\overline{EG}^H, \overline{EG}^H \setminus EG^H)$. The only properties that were required are the fact that $X$ is a CAT(0) space, the convergence properties of the embeddings between the various classifying spaces, and the fact that $\partial G_\sigma$ is a $\cZ$-set in $\overline{EG_\sigma}$. But since $X^H$ is convex in a CAT(0) space, it is itself CAT(0). Moreover, the convergence properties of the embeddings are clearly still satisfied for simplices that are fixed under $H$. Finally, by assumption,  $(\partial G_\sigma)^H$ is a $\cZ$-set in $\overline{EG_\sigma}^H$. Thus, the same reasoning as in \ref{Zset1} and \ref{Zset2} shows that the lemma \ref{BestvinaMess} of Bestvina-Mess applies, thus implying that $(\overline{EG}^H, \overline{EG}^H \setminus EG^H)$ is a $\cZ$-compactification, and we are done. \end{proof}

\section{A high-dimensional combination theorem for hyperbolic groups.}  
In this section, we apply our construction of boundaries to get a generalisation of a combination theorem of Bestvina-Feighn to complexes of groups of arbitrary dimension. 

This will be done by constructing an $E\cZ$-structure for $G$ and proving that $G$ is a uniform convergence group on its boundary. Note that this proof has the advantage of yielding a construction of the Gromov boundary of $G$.

In the following, $G(\cY)$ will be a complex of groups over a simplicial complex $Y$ satisfying the conditions of \ref{maintheorem2}. We will denote by $G$ the fundamental group of $G(\cY)$ and by $X$ a universal covering.

\subsection{A few facts about hyperbolic groups and quasiconvex subgroups.}

We start by recalling here a few elementary facts about hyperbolic groups. There is an extensive litterature about such groups, and we refer the reader to \cite{CoornaertDelzantPapadopoulos}, \cite{GromovHyperbolicGroups} for more details.

\begin{lem}
\begin{itemize}
\item Let $H_1 \leq H_2 \leq H$ be three hyperbolic groups. If $H_1$ is quasiconvex in $H_2$, and $H_2$ is quasiconvex in $H$, then $H_1$ is quasiconvex in $H$. If both $H_1$ and $H_2$ are quasiconvex in $H$, then $H_1$ is quasiconvex in $H_2$. 
\item (Gromov \cite{GromovAsymptotic}) Let $H$ be a hyperbolic group, and $H_1,H_2$ two quasiconvex subgroups. Then $H_1\cap H_2$ is quasiconvex in $H$, and $\Lambda(H_1 \cap H_2) = \Lambda H_1 \cap \Lambda H_2$. \qedhere
\end{itemize}
\label{lemmaquasiconvex}
\end{lem}

\begin{cor}
Let $\Gamma$ be a finite connected graph contained in the $1$-skeleton of $X$, and $ \Gamma' \subset  \Gamma$ a connected subgraph. Then $\cap_{v \in \Gamma} G_v$ is hyperbolic and quasiconvex in $\cap_{v \in \Gamma'} G_v$. 
\label{quasiconvexgraph}
\end{cor}
\begin{proof}
 This follows from an easy induction on the number of vertices of $\Gamma$, together with lemma \ref{lemmaquasiconvex}.
\end{proof}
Recall that in the case of a hyperbolic group $H$, there is a very explicit example of classifying space for proper actions, namely the Rips complex. Moreover, there is a natural notion of boundary, namely the Gromov boundary of $H$ (see \cite{CoornaertDelzantPapadopoulos}). 
\begin{thm}[ \cite{BestvinaMessBoundaryHyperbolic}, \cite{MeintrupSchickEGhyperbolic} ]
 Let $H$ be a finitely generated hyperbolic group, $H'$ a finitely generated subgroup, and $S$ a finite generating set of $H$ that contains a finite generating set of $H'$. For $n >> 0$, the Rips complex $P_n(H)$ is contractible and there is a topology on $P_n(H) \cup  \partial H$ such that $(P_n(H) \cup  \partial H, \partial H)$ is an $E\cZ$-structure for $H$. Furthermore, if $H'$ is quasiconvex in $H$, the equivariant embedding $P_n(H') \hra P_n(H)$ naturally extends to an equivariant embedding $P_n(H') \cup \partial H' \hra P_n(H) \cup \partial H$. \qed
\end{thm}

\subsection{Construction of an $E\cZ$-complex of space compatible with $G(\cY)$.}
We now define an $E\cZ$-complex of spaces over $Y$ as follows:
\begin{itemize}
 \item  We define inductively sets of generators for the local groups of the complex of groups $G(\cY)$ induced over $Y$ in the following way: Start with simplices $\sigma$ of $Y$ of maximal dimension, and choose for each of them a finite symmetric set of generators for $G_\sigma$. Suppose we have defined a set of generators for local groups over simplices of dimension at most $k$. If $\sigma$ is a simplex of dimension $k-1$, choose a finite set of generators which contains all the generators of local groups of simplices strictly containing $\sigma$. This allows us to define for every simplex $\sigma$ of $Y$ a set of generator such that $S_\sigma \subset S_{\sigma'}$ whenever $\sigma' \subset \sigma$.
 \item Let $n \geq 1$ be an integer. Define $D_\sigma$ as the Rips complex $P_n(G_\sigma)$ associated to the set of generators $S_\sigma$. Moreover, if $\sigma \subset \sigma'$, let $\phi_{\sigma, \sigma'}$ be the equivariant embedding $P_n(G_{\sigma'}) \hra P_n(G_\sigma)$.
 \item Since there are only finitely many hyperbolic groups involved, choose $n \geq 0$ such that all the previously defined Rips complexes are contractible.
\end{itemize}
It follows from the above discussion that 
\begin{prop}
 The complex of spaces $D(\cY)$ is compatible with the complex of groups $G(\cY)$. \qed
\end{prop}

\begin{lem}
 The $E\cZ$-complex of spaces $D(\cY)$ satisfies the limit set property.
\end{lem}
\begin{proof}
For every pair of simplices $\sigma \subset \sigma'$ of $Y$, $G_{\sigma'}$ is a quasiconvex subgroup of $G_\sigma$, so the map $\phi_{\sigma, \sigma'}: \partial G_{\sigma'} \ra \partial G_\sigma$ realises a $G_{\sigma'}$-equivariant homeomorphism $\partial G_{\sigma'} \ra \Lambda G_{\sigma'} \subset \partial G_\sigma$ by a result of Bowditch \cite{BowditchConvergenceGroups}.

For every simplex $\sigma$ of $Y$, the family $\cF_\sigma = \left\lbrace  \bigcap_{i=1}^{n} g_i G_{\sigma_i} g_i^{-1} | ~ g_0, \ldots, g_n \in G_\sigma, \sigma_1, \ldots,  \sigma_n \in \mbox{st}(\sigma), n \in \bbN \right\rbrace,$ is contained in the family of quasiconvex subgroups of $G_\sigma$. Indeed, let $g_0, \ldots, g_n$ be elements of $G$. Then, as $X$ is CAT(0), $ \bigcap_{0 \leq i \leq n} g_i G_\sigma g_i^{-1} = \bigcap_{v \in \Gamma} g_i G_v g_i^{-1},$
where $\Gamma$ is a graph containing all the vertices of the simplices $g_0 \sigma, \ldots, g_n\sigma$ and contained in the convex hull of the $g_0 \sigma, \ldots, g_n \sigma$. For such subgroups, the equality $ \Lambda H_1 \cap \Lambda H_2 = \Lambda(H_1 \cap H_2)$ holds by \ref{lemmaquasiconvex}.
\end{proof}

\begin{lem}
 The $E\cZ$-complex of spaces $D(\cY)$ satisfies the convergence property.
\end{lem}
\begin{proof}
 This is proposition 1.8 of \cite{DahmaniCombination}.
\end{proof}

\begin{lem}
  The $E\cZ$-complex of spaces $D(\cY)$ satisfies the finite height property.
\end{lem}
\begin{proof}
 A quasiconvex subgroup of a hyperbolic group has finite height by a result of \cite{SageevWidth}.
\end{proof}

The main theorem \ref{maintheorem1} now implies the following:

\begin{cor}
 The fundamental group of $G(\cY)$ admits a classifying space for proper actions and a strong boundary in the sense of Carlsson-Pedersen. \qed
\end{cor}

Note that this corollary does not use the hyperbolicity of $X$. 

\subsection{Background on convergence groups and hyperbolicity.}

\begin{definition}[convergence group] 
 A group $\Gamma$ acting on a compact metrisable space $M$ with more than two points is called a \textit{convergence group} if, for every sequence $(\gamma_n)$ of elements of $\Gamma$, there exists two points $\xi_+$ and $\xi_-$ in $M$ an a subsequence $(\gamma_{\varphi(n)})$, such that for any compact subspace $K \subset M \setminus \left\lbrace \xi_- \right\rbrace$, the sequence $(\gamma_{\varphi(n)})$ of translates uniformly converges to $\xi_+$. 
\end{definition}

\begin{definition}[conical limit point]
 Let $\Gamma$ be a convergence group on a compact metrisable space $M$. A point $\zeta$ in $M$ is called a \textit{conical limit point} if there exists a sequence $(\gamma_n)$ of elements of $\Gamma$ and two points $\xi_- \neq \xi_+$ in $M$, such that $\gamma_n \zeta \ra \xi_-$ and $\gamma_n \zeta' \ra \xi_+$ for every $\zeta' \neq \zeta$ in $M$. The group $\Gamma$ is called a \textit{uniform convergence group} on $M$ if $M$ consists only of conical limit points.
\end{definition}

\begin{thm}[Bowditch \cite{BowditchTopologicalCharacterization}] Let $\Gamma$ be a uniform convergence group on a compact metrisable space $M$ with more than two elements. Then $\Gamma$ is hyperbolic and $M$ is $\Gamma$-equivariantly homeomorphic to the Gromov boundary of $\Gamma$. \qed
\label{Bowditch}
\end{thm}

\subsection{A combination theorem.}

We now prove that $G$ is a hyperbolic group, by proving that it is a uniform convergence group on its boundary $\partial G$.\\

So far, the topology on $\overline{EG}$ and $\partial G$ was defined by choosing a specific, although arbitrary, base point. In forthcoming proofs, we will choose neighbourhoods centered at points which are relevant to the geometry of the problem.

\begin{definition}
Let $\delta \geq 0$ such that the space $X$ is $\delta$-hyperbolic. We denote by $\langle.,.\rangle$ the Gromov product on $X$ and an extension to $\overline{X}$. For $\eta \in \partial X$, $k \geq 0$ and $x_0 \in X$ a basepoint, let 
$$W_k(\eta) = \left\{x \in \overline{X} \mbox{ such that } \langle x, \eta \rangle_{x_0} \geq k \right\} .$$
For $k \geq 0$, the family subsets $(W_k(\eta))$ form a basis of (not necessarily open) neighbourhoods of $\eta$ in $\overline{X}$.
\end{definition}

Recall that we chose a constant $D>0$ bigger than every $d_\xi, \xi \in \partial_{Stab} G$.

\begin{lem} Let $(g_n)$ be an injective sequence of elements of $G$, and suppose there exist vertices $v_0$ and $v_1$ of $X$ such that $g_n v_0=v_1$ for infinitely many $n$. Then there exist $\xi_+, \xi_- \in \partial G$ and a subsequence $(g_{\varphi(n)})$ such that for every compact subset $K$ of $\partial G \setminus \left\{\xi_-\right\}$, the sequence of translates $g_{\varphi(n)}K$ uniformly converges to $\xi_+$.
\label{convergencegroup1}
\end{lem}

\begin{proof} It is enough to prove the result when $g_n v_0 = v_0$ for infinitely many $n$. Since $G_{v_0}$ is hyperbolic, we can assume that there exists a subsequence of $(g_n)$, that we still denote $(g_n)$, and elements $\xi_+, \xi_- \in \partial G_{v_0}$ such that for every compact subset $K$ of $\partial G_{v_0} \setminus \left\{\xi_-\right\}$, the sequence of translates $g_nK$ uniformly converges to $\xi_+$. Throughout this proof, we choose $v_0$ as the basepoint.

For every vertex $v$ of $D(\xi_-)$, choose $U_v$ to be a neighbourhood of $\xi_-$ in $\partial G_{v_0}$. Choose a $\xi_-$-family $\cU'$ which is nested in $\left\{U_v, v \in V(\xi_-) \right\}$, and choose $\varepsilon \in (0,1)$. Let $K = \partial G \setminus V_{\cU', \varepsilon}(\xi_-)$.

Let $\sigma$ be a simplex of $X$ containing $v_0$. 

If $\sigma$ is not contained in $D(\xi_-)$, then the convergence property implies that, up to a subsequence, we can assume that the sequence of $g_n \partial G_\sigma$ uniformly converges to $\xi_+$ in $\partial G_{v_0}$. 

If $\sigma $ is contained in $D(\xi_-)$, then the subset $\partial G_\sigma \subset G_{v_0}$ consists of at least two points among which $\xi_-$. Since for any other point $\alpha$ of $\partial G_\sigma$ we have that $g_n\alpha$ tends to $\xi_+$, the convergence property implies that one of the following situation happens:
\begin{itemize}
\item $g_n G_\sigma$ only takes finitely many values of cosets, in which case we can find a subsequence $(g_n)$ such that $g_n\partial G_\sigma$ is constant and contains $\xi_+$. This means that we can write $g_n = g_n'.g$ where $g$ is in the stabiliser of $v_0$ and $g_n'$ in a sequence in the stabiliser of $\sigma$. 
Up to replacing $g_n$ by $g_n'$, we can assume that $g_n$ fixes $\sigma$.
\item $g_n G_\sigma$ takes infinitely many values of cosets, in which case we can find a subsequence $(g_n)$ such that $g_n \partial G_\sigma$ uniformly converges to $\xi_+$.
\end{itemize}
 
As domains are finite subcomplexes of $X$ by \ref{finitedomain}, we can iterate this procedure a finite number of times so as to obtain a subsequence $(g_n)$ and a subcomplex $F \subset D(\xi_-) \cap D(\xi_+)$ such that
\begin{itemize}
\item $F$ is fixed pointwise under all the $g_n$, 
\item for every simplex $\sigma$ in $(st(F) \setminus F)$ and every vertex $v$ of $\sigma \cap F$, we have that $g_n \partial G_\sigma$ uniformly converges to $\xi_+$ in $\partial G_{v}$.
\end{itemize}

\begin{center}
\begingroup%
  \makeatletter%
  \providecommand\color[2][]{%
    \errmessage{(Inkscape) Color is used for the text in Inkscape, but the package 'color.sty' is not loaded}%
    \renewcommand\color[2][]{}%
  }%
  \providecommand\transparent[1]{%
    \errmessage{(Inkscape) Transparency is used (non-zero) for the text in Inkscape, but the package 'transparent.sty' is not loaded}%
    \renewcommand\transparent[1]{}%
  }%
  \providecommand\rotatebox[2]{#2}%
  \ifx\svgwidth\undefined%
    \setlength{\unitlength}{435.5073456bp}%
    \ifx\svgscale\undefined%
      \relax%
    \else%
      \setlength{\unitlength}{\unitlength * \real{\svgscale}}%
    \fi%
  \else%
    \setlength{\unitlength}{\svgwidth}%
  \fi%
  \global\let\svgwidth\undefined%
  \global\let\svgscale\undefined%
  \makeatother%
  \begin{picture}(1,0.39590479)%
    \put(0,0){\includegraphics[width=\unitlength]{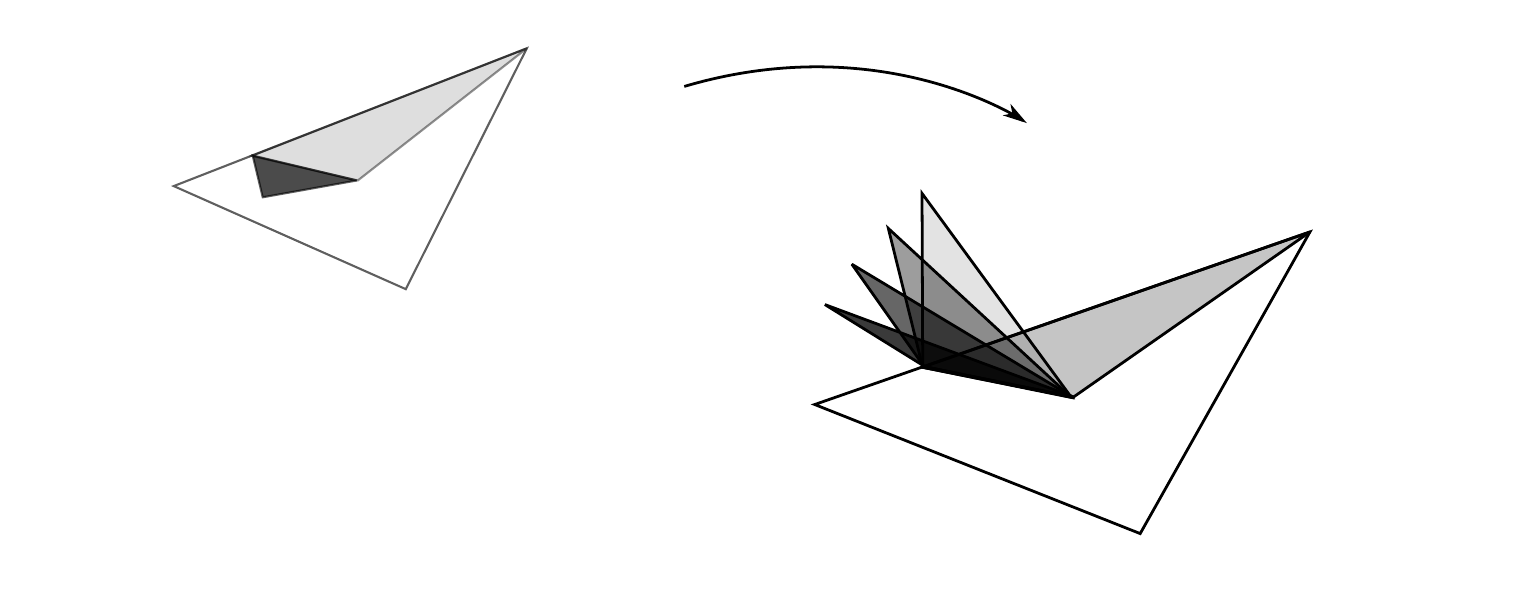}}%
    \put(0.3314003,0.39988401){\color[rgb]{0,0,0}\makebox(0,0)[lt]{\begin{minipage}{0.16663652\unitlength}\raggedright $v_0$\end{minipage}}}%
    \put(0.92666652,-0.32415746){\color[rgb]{0,0,0}\makebox(0,0)[lt]{\begin{minipage}{0.06242787\unitlength}\raggedright $v_1$\end{minipage}}}%
    \put(0.52935136,0.38518521){\color[rgb]{0,0,0}\makebox(0,0)[lt]{\begin{minipage}{0.17582123\unitlength}\raggedright $g_n$\end{minipage}}}%
    \put(0.7235065,0.19447672){\color[rgb]{0,0,0}\makebox(0,0)[lt]{\begin{minipage}{0.16532442\unitlength}\raggedright $g_n F$\end{minipage}}}%
    \put(0.2148412,0.26861975){\color[rgb]{0,0,0}\makebox(0,0)[lt]{\begin{minipage}{0.04739878\unitlength}\raggedright $\sigma$\end{minipage}}}%
    \put(0.50431897,0.23751029){\color[rgb]{0,0,0}\makebox(0,0)[lt]{\begin{minipage}{0.17057281\unitlength}\raggedright $g_n\sigma$\end{minipage}}}%
    \put(0.60399039,0.06895676){\color[rgb]{0,0,0}\makebox(0,0)[lt]{\begin{minipage}{0.0984074\unitlength}\raggedright $D(\xi_+)$\end{minipage}}}%
    \put(0.15699923,0.23054744){\color[rgb]{0,0,0}\makebox(0,0)[lt]{\begin{minipage}{0.0944711\unitlength}\raggedright $D(\xi_+)$\end{minipage}}}%
    \put(0.23534509,0.32006363){\color[rgb]{0,0,0}\makebox(0,0)[lt]{\begin{minipage}{0.0879106\unitlength}\raggedright $F$\end{minipage}}}%
    \put(-0.00260824,0.33019123){\color[rgb]{0,0,0}\makebox(0,0)[lt]{\begin{minipage}{0.05978176\unitlength}\raggedright $    $\end{minipage}}}%
    \put(0.84413496,0.28236275){\color[rgb]{0,0,0}\makebox(0,0)[lt]{\begin{minipage}{0.16124771\unitlength}\raggedright $v_0$\end{minipage}}}%
  \end{picture}%
\endgroup%

Figure $4$.
\end{center}

We now prove that, up to a subsequence, the sequence of translates $g_nK$ uniformly converges to $\xi_+$. Because of the definition of neighbourhoods of points of $\partial_{Stab} G$, we need to treat different cases.

Let $\sigma$ be a simplex of $F$ containing $v_0$, so that $G_\sigma \subset G_{v_0}$. By definition of $\xi_+$ and $\xi_-$, we already have that the sequence $g_n (\partial G_{v_0} \setminus U'_{v_0})$ uniformly converges to $\xi_+$ in $G_{v_1}$. Now let $v$ be another vertex of $\sigma$. We thus have that $g_n (\partial G_\sigma \setminus U_v')$ uniformly converges to $\xi_+$ in $\partial G_{v}$. This implies that there exists a subsequence, still denoted $(g_n)$, such that $g_n( \partial G_v \setminus U_v')$ uniformly converges to $\xi_+$. Since $F$ is finite, an easy induction shows that there exists a subsequence, still denoted $(g_n)$, such that $g_n(\partial G_v \setminus U_v')$ uniformly converges to $\xi_+$ in $\partial G_{gv}$ for every vertex $v$ of $F$.

Let $\widetilde{x} \in K$, and $x \in \bar{p}(\widetilde{x}) \setminus F$. Let $\sigma$ be the first simplex touched by $[v_0, x]$ after leaving $g^{-1}F$. It follows from the previous discussion that the sequence of simplices $(g_n \sigma)$ is such that for some (hence any) vertex $v$ of $\sigma \cap F$, the sequence of $(\partial G_{g_n \sigma})$ uniformly converges to $\xi_+$ in $\partial G_{v}$. It follows from the convegence criterion \ref{convergencecriterion} that the sequence $(g_n\widetilde{x})$ converges to $\xi_+$.  Since $\widetilde{x} \notin V_{\cU, \varepsilon}(\xi_-)$, we have $\partial G_{\sigma} \not\subset U_v$ for some (hence any) vertex $v$ of $F$. Since $\cU'$ is nested in $\left\{U_w, w \in V(\xi_-) \right\}$, it follows that 
$$  \partial G_{\sigma} \cap U_v' = \varnothing.$$
We already have that for every vertex $v$ of $F$, the sequence of $g_n.(\partial G_v \setminus U_v)$ uniformly converges to $\xi_+$ by the above discussion. As $F$ is a finite subcomplex of $X$, the convergence criterion \ref{convergencecriterion} now shows that the sequence $(g_n.K)$ uniformly converges to $\xi_+$.
\end{proof}

\begin{lem} Let $(g_n)$ be an injective sequence of elements of $G$. Suppose that for some (hence any) vertex $v$ the sequence $(g_nv)$ is bounded, but there does not exist vertices $v_0$ and $v_1$ of $X$ such that $g_n v_0=v_1$ for infinitely many $n$. Then there exist $\xi_+, \xi_- \in \partial G$ and a subsequence $(g_{\varphi(n)})$ such that for every compact subset $K$ of $\partial G \setminus \left\{\xi_-\right\}$, the sequence of translates $g_{\varphi(n)}K$ uniformly converges to $\xi_+$.
\label{convergencegroup2}
\end{lem}

\begin{proof} Choose a vertex $v_0$ and an element $\widetilde{x_0}$ in $EG_{v_0}$. As $\partial G$ is compact by \ref{compact} and $(g_nv_0)$ is bounded, we can choose a subsequence, still denoted $(g_n)$, and elements $\xi_+, \xi_- \in \partial_{Stab} G$ such that $g_n \widetilde{x_0}\ra \xi_+$ and $g_n^{-1}\widetilde{x_0} \ra \xi_-$. \\

\textit{Claim 1:} \begin{itemize}
\item For every $\eta \in \partial X$, the geodesic ray $[g_nv_0, g_n\eta)$ does not meet $D(\xi_+)$ for $n$ large enough.
\item For every $\xi \in \partial_{Stab} G$, the subset $\mbox{Geod}(g_nv_0, g_nD(\xi))$ does not meet $D(\xi_+)$ for $n$ large enough.\\
\end{itemize} 
\noindent Let $z \in \partial G$. If $z \in \partial X$, we denote by $D(z)$ the singleton $\left\{z\right\}$. By contradiction, suppose that there exists an infinite number of $n$ for which there exists $y_n \in D(\xi_+)$ and $x_n \in \mbox{Geod}(v_0, D(z))$ such that $g_n x_n = y_n$. As $(y_n)$ is bounded by \ref{finitedomain}, the assumption on $(g_n)$ implies that $(x_n)$ is bounded too. Since $x_n$ lies on $\mbox{Geod}(v_0, D(z))$ for every $n$, the containment lemma \ref{containment} and the finiteness lemma \ref{finite} imply that, up to a subsequence, we can assume that $x_n$ always lies in the same simplex $\sigma$ of $X$. Furthermore, since $D(\xi_+)$ is finite by \ref{finitedomain}, we can assume, up to a subsequence, that $y_n$ lies in a simplex $\sigma'$ of $X$ for every $n$. As the action of $G$ on $X$ is without inversion, this implies that $g_n \sigma = \sigma'$ for every $n$, which was exluded by assumption.\\

\textit{Claim 2:} For every $\xi$ in $\partial G$, the sequence $g_n \xi $ converges to $\xi_+$. \\

Here we choose the basepoint to be a vertex $v_0$ of $D(\xi_+)$. Let $\cU$ be a $\xi_+$-family, $\cU'$ a $\xi_+$-family that is $D$-nested in $\cU$ and $\varepsilon > 0$ (recall that $D$ is a constant such that a geodesic segment contained in the open star of the domain of any element of $\partial_{Stab}$ meets at most $D$ simplices). We split the proof of the claim in two cases. 

Let $\eta \in \partial X$.  As $[g_nv_0,v_0] \cup [v_0, g_n\eta)$ is not a geodesic for $n$ large enough because of the above claim, there is a path of simplices of length at most $D$ from the exit simplex $\sigma_{\xi_+, \varepsilon}(g_nv_0)$ to the exit simplex $\sigma_{\xi_+, \varepsilon}(g_n\eta)$ in $\mbox{st}(D(\xi_+)) \setminus D(\xi_+)$. As $\overline{EG_{\sigma_{\xi_+, \varepsilon}(g_nv_0)}} \subset U_{v}'$ for $n$ large enough and for some (hence every) vertex $v$ of $D(\xi_+) \cap \sigma_{\xi_+, \varepsilon}(g_nv_0)$, it follows from the fact that $\cU'$ is $D$-nested in $\cU$ that $\overline{EG_{\sigma_{\xi_+, \varepsilon}(g_n\eta)}} \subset U_v$ for $n$ large enough and for some (hence every) vertex $v$ of $D(\xi_+) \cap \sigma_{\xi_+, \varepsilon}(g_n\eta)$. It thus follows that $(g_n \eta)$ converges to $\xi_+$.

Let $\xi \in \partial_{Stab} G$.  As $\mbox{Geod}(g_nv_0, g_nD(\xi))$ does not meet $D(\xi_+)$ for $n$ large enough by the above claim, it follows that for $n$ large enough the paths $[g_nv_0,v] \cup [v, g_nx]$ are not geodesic for every $x \in D(\xi)$. Thus for $n$ large enough, there is a path of simplices of length at most $D$ from the exit simplex $\sigma_{\xi_+, \varepsilon}(g_nv_0)$ to the exit simplex $\sigma_{\xi_+, \varepsilon}(g_nx)$ in $\mbox{st}(D(\xi_+)) \setminus D(\xi_+)$. As $\overline{EG_{\sigma_{\xi_+, \varepsilon}(g_nv_0)}} \subset U_v'$ for $n$ large enough and for some (hence every) vertex $v$ of $D(\xi_+) \cap \sigma_{\xi_+, \varepsilon}(g_nv_0)$, it follows from the fact that $\cU'$ is $D$-nested in $\cU$ that, for $n$ large enough and for every $x \in D(\xi)$, $\overline{EG_{\sigma_{\xi_+, \varepsilon}(g_nx)}} \subset U_v$ for some (hence every) vertex $v$ of $D(\xi_+) \cap \sigma_{\xi_+, \varepsilon}(g_nx)$. It thus follows that $(g_n \xi)$ converges to $\xi_+$.

In the same way, we prove that for every $\xi \in \partial G$, the sequence $g_n^{-1}\xi$ converges to $\xi_-$.
To conclude the proof of the lemma, it remains to show that this convergence can be made uniform away from $\xi_-$:\\

\textit{Claim 3:} For every $\xi \neq \xi_-$ in $\partial G$, there is a subsequence $(g_n)$ and a neighbourhood $U$ of $\xi$ in $\partial G$ such that the sequence of $g_nU$ uniformly converges to $\xi_+$.\\

\noindent Once again, we split the proof in two cases.

Let $\xi \in \partial_{Stab} G$. We already have that $g_n\xi \ra \xi_+$ by the second claim. In order to find a $\xi$-family $\cU$ and a constant $\varepsilon$ such $g_nV_{\cU, \varepsilon}(\xi)$ uniformly converges to $\xi_+$, it is enough, using the same reasoning as in Claim $2$, to find a $\xi$-family $\cU$ and a constant $\varepsilon$ such that for every $x$ in $\mbox{Cone}_{\cU, \varepsilon}(\xi)$, the geodesic from $g_nv_0$ to $g_nx$ does not meet $D(\xi_+)$.\\
\noindent As $\xi \neq \xi_-$, we choose a $\xi$-family $\cU$, a $\xi_-$-family $\cU'$ and constants $\varepsilon, \varepsilon' \in (0,1)$ such that the neighbourhoods $V_{\cU, \varepsilon}(\xi)$ and $V_{\cU', \varepsilon'}(\xi_-)$ are disjoint. Up to a subsequence, we have by the first claim that $g_n D(\xi)$ does not meet $D(\xi_+)$. It now follows from the definition of $\cU$ and the fact that $g_n^{-1}\xi_+ \ra \xi_-$ that $\mbox{Cone}_{\cU, \varepsilon}(\xi)$ does not meet the sets $g_n^{-1}D(\xi_+)$, hence the sets $g_n\mbox{Cone}_{\cU, \varepsilon}(\xi)$ do not meet $D(\xi_+)$. Now this implies that for every $x$ in $\mbox{Cone}_{\cU, \varepsilon}(\xi)$, the geodesic from $g_nv_0$ to $g_nx$ does not meet $D(\xi_+)$: indeed, this geodesic must meet $g_nD(\xi)$ since the geodesic from $v_0$ to a point of $\mbox{Cone}_{\cU, \varepsilon}(\xi)$ must meet $D(\xi)$, and we already proved that a geodesic segment from $g_nv_0$ to a point of $g_nD(\xi)$ does not meet $D(\xi_+)$. Now the same proof as in the second claim shows that $g_n V_{\cU, \varepsilon}(\xi)$ uniformly converges to $\xi_+$.

Let $\eta \in \partial X$. We already know that $g_n\eta \ra \xi_+$ by the second claim. In order to find a neighbourhood $U$ of $\eta$ in $\overline{X}$ such that such $g_nV_{U}(\eta)$ uniformly converges to $\xi_+$, it is enough, using the same reasoning as in Claim $2$, to find a neighbourhood $U$ of $\eta$ in $\overline{X}$ such that for every $x$ in $U$, the geodesic from $g_nv_0$ to $g_nx$ does not meet $D(\xi_+)$.\\
\noindent First, notice that the distance from the geodesic rays $[g_nv_0, g_n\eta)$ to $D(\xi_+)$ is uniformly bounded below: indeed, if this was not the case, the same reasoning as in the first claim would imply the existence of simplices $\sigma, \sigma'$ of $X$ such that $g_n \sigma \cap \sigma' \neq \varnothing$. This in turn would imply that, up to a subsequence, there exist subsimplices $\tau \subset \sigma$ and $\tau' \subset \sigma'$ such that $g_n \tau = \tau'$, which was excluded. Thus, let $\varepsilon>0$ be such a uniform bound. Let also 
$$M = \underset{x \in D(\xi_+), n \geq 0}{\mbox{sup}} d(g_nv_0, x).$$

Now consider the neighbourhood $V_{M, \varepsilon}(\eta)$ of $\eta$ in $\overline{X}$. Let $x \in X$ be in that neighbourhood, and let $\gamma$ be a parametrisation of the geodesic from $v_0$ to $x$. Suppose by contradiction that the geodesic from $g_nv_0$ to $g_nx$ does meet $D(\xi_+)$. Then, by definition of $M$, the geodesic segment $g_n\gamma\big( [0,M] \big)$ meets $D(\xi_+)$. But as this geodesic segment is in the open $\varepsilon$-neighbourhood of $[g_nv_0, g_n\eta)$, we get our contradiction from the definition of $\varepsilon$.\\
\noindent Thus for every $x$ in $V_{M, \varepsilon}(\eta)$, the geodesic from $g_nv_0$ to $g_nx$ does not meet $D(\xi_+)$, and we are done.
\end{proof}

\begin{lem} Let $(g_n)$ be an injective sequence of elements of $G$, and suppose that for some (hence every) vertex $v_0$ of $X$, $d(v_0, g_n v_0) \ra \infty$. Since $(\overline{EG}, \partial G)$ is an $E\cZ$-structure for $G$ by \ref{EZstructure}, we can assume up to a subsequence that there exist $\xi_+, \xi_- \in \partial G$ such that for every compact subset $K \subset EG$, we have $g_n K \ra \xi_+$ and $g_n^{-1} K \ra \xi_-$. Then there exists a subsequence $(g_{\varphi(n)})$ such that for every compact subset $K$ of $\partial G \setminus \left\{\xi_-\right\}$, the sequence of translates $g_{\varphi(n)}K$ uniformly converges to $\xi_+$. 
\label{convergencegroup3}
\end{lem}

\begin{proof} 
If $\xi_- \in \partial X$, let $U$ be a neighbourhood of $\xi_-$ in $\partial X$ and $K = \partial G \setminus V_{U}(\xi_-)$. Since $X$ has finitely many isometry types of simplices, it follows from \ref{CAT(0)nesting} that we can choose a subneighbourhood $U'$ of $U$ containing $\xi_-$ and such that any path from $U' \cap X$ to $X \setminus U$ meets at least $D$ simplices.

If $\xi_- \in \partial_{Stab} G$, let $\cU$ be a $\xi_-$-family, and $\varepsilon \in (0,1)$, and let $K = \partial G \setminus V_{\cU, \varepsilon}(\xi_-)$. We also choose another $\xi_-$-family $\cU'$ which is $D$-nested and $D$-refined in $\cU$.\\
We want to prove that $(g_nK)$ uniformly converges to $\xi_+$.\\

\textit{Claim  1:} For every $k$, the following holds:
\begin{itemize}
\item If $\xi_- \in \partial X$, we have $g_n \big( \overline{X} \setminus U' \big)  \subset W_k(g_nv_0)$ for $n$ large enough.
\item $\xi_- \in \partial_{Stab} G$, we have $g_n\big( \overline{X} \setminus \widetilde{\mbox{Cone}}_{\cU', \varepsilon}(\xi_-) \big) \subset W_k(g_nv_0)$ for $n$ large enough.
\end{itemize}

\noindent We split the proof in two cases.

Suppose that $\xi_- \in \partial X$. First notice that since $g_n^{-1}v_0 \ra \xi_-$, there exists a constant $C$ such that for every $n \geq 0$ and every $x \notin U$, we have $\langle g_n^{-1}v_0,x\rangle_{v_0} ~\leq C$. Since we also have $d(g_n^{-1}v_0, v_0) \ra \infty$, the claim follows.

Suppose now that $\xi_- \in \partial_{Stab} G$. We start by proving by contradiction that there exists a constant $C$ such that for every $n \geq 0$ and every $x \notin \widetilde{\mbox{Cone}}_{\cU', \varepsilon}(\xi_-)$, we have $\langle g_n^{-1}v_0,x\rangle_{v_0} ~ \leq C$.\\
The containment lemma \ref{containment} yields a constant $m$ such that a path of length at most $\delta$ meets at most $m$ simplices, where $\delta$ is the hyperbolicity constant of $X$. Let $\cU''$ be a $\xi_+$-family that is $m$-nested in $\cU_+$. For every $n$, let $x_n$ be the point of $X$ met by the geodesic $[v_0, g_n^{-1}v_0] $ after leaving $D(\xi_-)$. Since we are reasoning by contradiction, then, up to a subsequence, there exist elements $y_n \notin \widetilde{\mbox{Cone}}_{\cU', \varepsilon}(\xi_-)$ such that $\langle g_n^{-1} v_0, y_n \rangle_{v_0} \ra \infty$. Now the hyperbolicity of $X$ implies that, for $n$ big enough, every geodesic segment $[x_n, y_n]$ meets the $\delta$- neighbourhood $[x_n, g_n^{-1}v_0]^\delta \setminus D^{\delta}(\xi_-)$. As $g_n^{-1}\xi_0 \ra \xi_-$, we have $g_n^{-1}v_0 \in \mbox{Cone}_{\cU'', \varepsilon}(\xi_-)$ for $n$ large enough, and the refinement lemma \ref{refinement} now implies that $y_n \in \mbox{Cone}_{\cU', \varepsilon}(\xi_+)$ for $n$ large enough, a contradiction.\\
Now the same reasoning as before shows that for every $k \geq 0$, there exists $N$ such that for every $n \geq N$ and every $x \notin \widetilde{\mbox{Cone}}_{\cU', \varepsilon}(\xi_-)$, $\langle v_0,x\rangle_{g_n^{-1}v_0} ~ \geq k$, hence $\langle g_nv_0,g_nx\rangle_{v_0} ~ \geq k$.\\

\textit{Claim 2:} $g_nK$ uniformly converges to $\xi_+$. \\

\noindent Once again, we split the proof in two cases.

Suppose that $\xi_+ \in \partial X$. Then, as $g_nv_0 \ra \xi_+$, it follows from the first claim that for every $k$, we have $g_n \big( \overline{X} \setminus U' \big)  \subset W_k(\eta)$ for $n$ large enough. By definition of $U'$, this implies that $g_n \bar{p}(K)  \subset W_k(\eta)$ for $n$ large enough. From the definition of the topology of $\overline{EG}$, it follows that $g_nK$ uniformly converges to $\xi_+$.

Suppose now that $\xi_+ \in \partial_{Stab} G$. Let $\cU_+$ be a $\xi_+$-family and $\varepsilon \in (0,1)$. Since $X$ is $\delta$-hyperbolic, let $m$ be a constant such that a path of length at most $\delta$ meets at most $m$ simplices, and let $\cU_+'$ be a $\xi_+$-family that is $m$-nested in $\cU_+$. For every $n$, let $x_n$ be the unique point met by $[v_0, g_nv_0]$ after leaving $D(\xi_+)$. As for every $k$ there exists $N$ such that $g_n\big( \overline{X} \setminus \widetilde{\mbox{Cone}}_{\cU', \varepsilon}(\xi_-) \big) \subset W_k(g_nv_0)$ for $n \geq N$, the hyperbolicty of $X$ implies that, for $n$ big enough, every geodesic segment from $x_n$ to a point of  $g_n\big( \overline{X} \setminus \widetilde{\mbox{Cone}}_{\cU', \varepsilon}(\xi_-) \big)$ meets the $\delta$-neighbourhood $[x_n, g_nv_0]^\delta \setminus D^{\delta}(\xi_+)$. As $g_n\xi_0 \ra \xi_+$, we have $g_nv_0 \in \mbox{Cone}_{\cU_+', \varepsilon}(\xi_+)$ for $n$ large enough, and the refinement lemma \ref{refinement} now implies that $g_n\big( \overline{X} \setminus \widetilde{\mbox{Cone}}_{\cU', \varepsilon}(\xi_-) \big) \subset \mbox{Cone}_{\cU_+, \varepsilon}(\xi_+)$ for $n$ large enough. But since $\cU'$ is $D$-nested and $D$-refined in $\cU$, this implies that $g_n\bar{p}(K) \subset \mbox{Cone}_{\cU_+, \varepsilon}(\xi_+)$ for $n$ large enough. Thus, $g_nK$ uniformly converges to $\xi_+$.
\end{proof}
 
\begin{cor} The group $G$ is a convergence group on $\partial G$. 
\label{convergencegroup}
\end{cor}

\begin{proof}
This follows from \ref{convergencegroup1}, \ref{convergencegroup2} and \ref{convergencegroup3}.
\end{proof}

To prove that $G$ is hyperbolic, it remains to show that every point of $\partial G$ is conical.

\begin{lem} Every point of $\partial G$ is a conical limit point for $\partial G$.
\label{conical}
\end{lem}
\begin{proof} Consider first an element in $\partial G_v$ for some vertex $v$ of $X$. It is a conical limit point for $G_v$ on $\partial G_v$, since $G_v$ is hyperbolic. Therefore it is a conical point for $G_v$ on $\partial G$, hence for $G$ since $G$ is a convergence group on $\partial G$ by \ref{convergencegroup}.\\

Now consider an element $\eta \in \partial X$. Since the action of $G$ on $X$ is cocompact, we can find a sequence $(g_n)$ of elements of $G$ and a simplex $\sigma$ such that ($g_n\sigma$) uniformly converges to $\eta$ in $\overline{X}$ and such that for every $n$, the geodesic ray $[v_0, \eta)$ meets $g_n\sigma$. Let $v$ be a vertex of $\sigma$ and $\widetilde{x} \in EG_{v}$. Up to a subsequence, we can assume that there exists $\xi_- \in \partial G$ such that $g_n^{-1}\widetilde{x}$ converges to $\xi_-$. Up to multiplying on the right the $g_n$ by elements of $G_{v}$, we can ensure that $\xi_-$ does not belong to $G_{v}$, and we still have that $\sigma$ meets the geodesic ray $[g_n^{-1}v_0, g_n^{-1}\eta)$ for every $n$. By \ref{convergencegroup3} we can assume, up to a subsequence, that for every $\xi \in \partial G \setminus \left\{\eta\right\}$ we have $g_n^{-1} \xi \ra \xi_-$. Hence it is enough to prove that $g_n^{-1}\eta$ does not converge to $\xi_-$, which we now prove by contradiction.

Suppose $g_n^{-1}\eta$ was converging to $\xi_-$. For every $n$, let $x_n$ be a point of $\sigma \cap[g_n^{-1}v_0, g_n^{-1}\eta)$. Since the geodesic ray $[g_n^{-1}v_0, g_n^{-1}\eta)$ meets $\sigma$ for every $n$, the Gromov product $\langle g_n^{-1}v_0, g_n^{-1}\eta\rangle_{v_0}$ is bounded. Thus, $\xi_-$ cannot belong to $\partial X$, and $\xi_- \in \partial_{Stab} G$. 

Now since both $g_n^{-1}\eta$ and $g_n^{-1}\widetilde{x}$ converge to $\xi_- \in \partial_{Stab} G$, both geodesics $[v, g_n^{-1}\eta)$ and $[v, g_n^{-1}v_0]$ must go through $D(\xi_-)$ for $n$ large enough. But \ref{compact1} implies that $[x_n, g_n^{-1}\eta)$ and $[x_n, g_n^{-1}v_0]$ also meet $D(\xi_-)$ for $n$ large enough. As $D(\xi-)$ is convex by \ref{finitedomain}, this implies that $x_n$ belongs to $D(\xi_-)$, hence so does $v$, which is absurd by construction of $(g_n)$. 
\end{proof}

\begin{cor} $G$ is a hyperbolic group and $\partial G$ is $G$-equivariantly homeomorphic to its Gromov boundary.
\label{uniformconvergencegroup}
\end{cor}

\begin{proof}
The group $G$ is a convergence group on $\partial G$ by \ref{convergencegroup}, and every point of $\partial G$ is conical by \ref{conical}, thus the result follows from \ref{Bowditch}.
\end{proof}

To conlude the proof of \ref{maintheorem2}, it remains to show that stabilisers embed as quasiconvex subsets.

\begin{prop} Stabilisers of simplices of $X$ are quasiconvex subgroups of $G$. 
\label{quasiconvexstabilisers}
\end{prop}
\begin{proof} It is enough to prove the result for the stabiliser of a vertex $v$ of $X$. Notice that, by \ref{inducedtopology}, the boundary of $G_v$ embeds $G_v$-equivariantly in $\partial G$, the latter being $G$-equivariantly homeomorphic to the Gromov boundary of $G$ by \ref{uniformconvergencegroup}. Hence, the result follows from a result of Bowditch \cite{BowditchConvergenceGroups} recalled in the introduction.
\end{proof} 

\begin{proof}[Proof of \ref{maintheorem2}:]
This follows from \ref{uniformconvergencegroup} and \ref{quasiconvexstabilisers}.
\end{proof}

\bibliographystyle{alpha}
\bibliography{biblioGGT}

\end{document}